\documentclass{amsart}

\usepackage{amsmath,amscd}
\numberwithin{equation}{section}

\usepackage{graphicx, mathtools}

\usepackage{subfig}

\usepackage{tikz}
\usetikzlibrary{arrows}
\usetikzlibrary{patterns}

\allowdisplaybreaks

\theoremstyle{plain}
\newtheorem{theorem}{Theorem}[section]

\newtheorem{lemma}[theorem]{Lemma}
\newtheorem{corollary}[theorem]{Corollary}

\newtheorem{problem}[theorem]{Problem}

\theoremstyle{definition}
\newtheorem{definition}[theorem]{Definition}

\theoremstyle{remark}
\newtheorem{remark}[theorem]{Remark}

\begin{document}

\title
[Lefschetz pencils on a complex projective plane]{Lefschetz pencils on a complex projective plane from a topological viewpoint}

\author[Ju A Lee]{Ju A Lee}
\date{\today}
\address{Ju A Lee \\ Research Institute of Mathematics \\ Seoul National University \\ Seoul 151-747, Korea}
\email{jualee@snu.ac.kr}


\begin{abstract}
In this article, we present a differential topological construction of symplectic Lefschetz pencils of genus $\frac{(d-1)(d-2)}{2}$ with $d^2$ base points and $3(d-1)^2$ critical points for arbitrary $d\geq 4$, analogous to the holomorphic Lefschetz pencils of curves of degree $d$ in $\mathbb{C}P^2$. Moreover, for the case $d=4$, we derive an explicit monodromy factorization of the genus $3$ holomorphic Lefschetz pencil on $\mathbb{C}P^2$ based on the braid monodromy technique and prove that it can also be topologically constructed by breeding the monodromy relations of the genus $1$ holomorphic Lefschetz pencils.  

\end{abstract}

\maketitle

\section{Introduction}

Any projective surface admits a (holomorphic) Lefschetz pencil. More generally, Donaldson \cite{Don:99} showed that every symplectic $4$-manifold admits a (symplectic) Lefschetz pencil. Conversely, the total space of any Lefschetz pencil admits a compatible symplectic form, as proven by Gompf \cite{Go:04}. On the other hand, Lefschetz pencils can be described topologically by means of their monodromy factorizations in the mapping class group of the fiber surface according to Kas \cite{Kas:80}, Matsumoto \cite{Mat:96}, and Baykur and Hayano \cite{BH:16}. Altogether, these results enable a topological characterization of projective surfaces, and more generally, of symplectic $4$-manifolds. For example, Hamada and Hayano studied the holomorphic Lefschetz pencils on $T^4$ due to Smith \cite{Sm:01}, from a topological point of view in \cite{HH:18}. Korkmaz and Ozbagci found a topological construction analogous to the genus $1$ holomorphic Lefschetz pencil on a complex projective plane $\mathbb{C}P^2$ in \cite{KO:08} by using the well-known relations in the mapping class group, and its equivalence with the holomorphic one was later verified by Hamada and Hayano in \cite{HH:21}. In this paper, we address the topological construction of the holomorphic Lefschetz pencils of all higher genera on $\mathbb{C}P^2$.

Throughout this paper, we denote the $n$-dimensional complex projective space by $\mathbb{P}^n$ for simplicity. First, we recall that the holomorphic Lefschtez pencil on the Veronese surface $v_d(\mathbb{P}^2)\subset\mathbb{P}^N$, the image under the Veronese embedding $v_d$ for each natural number $d$, is equivalent to the Lefschetz pencil of degree $d$ curves in $\mathbb{P}^2$ obtained by choosing a generic pair of homogeneous polynomials $p_0$ and $p_1$ of degree $d$ in $\mathbb{C}^3$.(cf.\cite{GS:99}) The chosen polynomials define two degree $d$ curves in $\mathbb{P}^2$ intersecting each other at $d^2$ points, say $Q_1,\cdots Q_{d^2}$. From this, one can consider the pencil $\{C_{[t_0:t_1]}~|~[t_0:t_1]\in\mathbb{P}^1\}$ of degree $d$ curves, where $C_{[t_0:t_1]}:=\{z\in\mathbb{P}^2~|~t_0p_0(z)+t_1p_1(z)=0\}$, and define the Lefschetz pencil \[h_d:\mathbb{P}^2\setminus\{Q_1,\cdots,Q_{d^2}\}\rightarrow\mathbb{P}^1, ~Q\mapsto [t_0:t_1],\] where $[t_0:t_1]\in\mathbb{P}^1$ is the unique parameter for which $C_{[t_0:t_1]}$ passes through $Q$. This holomorphic Lefschetz pencil has fiber genus $\frac{(d-1)(d-2)}{2}$, $d^2$ base points, and $3(d-1)^2$ many critical points. Moreover, after blowing up all the base points, it extends to a Lefschetz fibration $\tilde{h}_d:\mathbb{P}^2\#d^2\overline{\mathbb{P}^2}\rightarrow\mathbb{P}^1$ with $d^2$ many $(-1)$ sections.

Let $\Sigma_g^b$ be a compact oriented surface of genus $g$ with $b$ boundary components, and let $\Gamma_g^b$ be its mapping class group, defined as the group of isotopy classes of all orientation preserving self-diffeomorphisms of $\Sigma_g^b$ which are the identity on the boundary and the isotopies are assumed to fix the boundary points.  
For $d=1$, the pencil of lines has no singular fibers. For $d=2$, the monodromy factorization of the Lefschetz pencil of conic curves is known to be a Lantern relation, a well-known relation of the form $t_{c_3}t_{c_2}t_{c_1}=\prod_{j=1}^{4}t_{\delta_j}$ in $\Gamma_0^4$, where the right-hand-side is the product of the right-handed Dehn twists about four boundary parallel simple closed curves $\delta_j$ and the left-hand-side is the product of the right-handed Dehn twists about three interior curves $c_i$ in $\Sigma_0^4$ (cf.\cite{Aur:03, AK:08}). And for $d=3$, Korkmaz and Ozbagci \cite{KO:08} found a new boundary-interior relation in $\Gamma_1^9$, so called the $9$-holed torus relation, which has the form $t_{c_{12}}\cdots t_{c_2}t_{c_1}=\prod_{j=1}^{9}{t_{\delta_j}}$, as a differential topological construction analogous to the holomorphic Lefschetz pencil of cubic curves. At that time, it was not known if their topological construction in the mapping class group actually realizes the holomorphic one, but later it was verified by Hamada and Hayano in \cite{HH:21}. Now, it is natural to ask the following question.

\begin{problem}
\begin{itemize}
\item[(a)] Can we find such boundary-interior relations for arbitrary $d\geq4$ analogous to the holomorphic Lefschetz pencils of degree $d$ curves in $\mathbb{P}^2$, in a topological way by using the well-known relations in the mapping class group?
\item[(b)] Can we show the equivalence between the topologically constructed relations and the holomorphic Lefschetz pencil of degree $d$ curves in $\mathbb{P}^2$?
\end{itemize}
\end{problem}

Our first main theorem, Theorem \ref{thm:first} below, gives an answer to the above question (a) for arbitrary $d\geq 4$, and the second main theorem, Theorem \ref{thm:second}, provides an answer to (b) for the $d=4$ case. 

\begin{theorem}\label{thm:first}
For arbitrary $d\geq 4$, there exists a symplectic Lefschetz pencil $f_d:X_d\dashrightarrow\mathbb{P}^1$ of genus $g_d=\frac{(d-1)(d-2)}{2}$ with $d^2$ base points and $3(d-1)^2$ critical points whose total space $X_d$ is diffeomorphic to $\mathbb{P}^2$ and whose monodromy factorization is given by  
\[j(B_{3(d-2)^2}\cdots B_{d-1})i_{d-2}(\mathbb{B}_1)i_{d-3}(\mathbb{B}_2)\cdots i_2(\mathbb{B}_2)i_1(\mathbb{A})=\prod_{j=1}^{d^2}T_{\delta_j} ~\text{in}~ \Gamma_{g_d}^{d^2},\]
where 
\begin{itemize}
\item $B_{3(d-2)^2}\cdots B_{d-1}$ denotes a subword of the monodromy factorization of $f_{d-1}:X_{d-1}\dashrightarrow\mathbb{P}^1$;
\item $\mathbb{A}$ denotes a subword of length $9$ of the $8$-holed torus relation on $\mathbb{P}^2\#\overline{\mathbb{P}^2}$;
\item $\mathbb{B}_1,$ (or $\mathbb{B}_2$) denotes a subword of length $8$, (or $7$) of the $8$-holed torus relation on $\mathbb{P}^1\times\mathbb{P}^1$;
\item $i_1,\cdots, i_{d-2}$ represent the embeddings of $\Sigma_1^8$ into $\Sigma_{g_d}^{d^2}$ as in Figure \ref{fig:Embeddings}(a); and 
\item $j$ represents the embedding of $\Sigma_{g_{d-1}}^{(d-1)^2}$ into $\Sigma_{g_d}^{d^2}$ as in Figure \ref{fig:Embeddings}(b).
\end{itemize}
\end{theorem}

\begin{figure}
\subfloat[$i_1,\cdots,i_{d-2}$]{\includegraphics[width=0.4\textwidth]{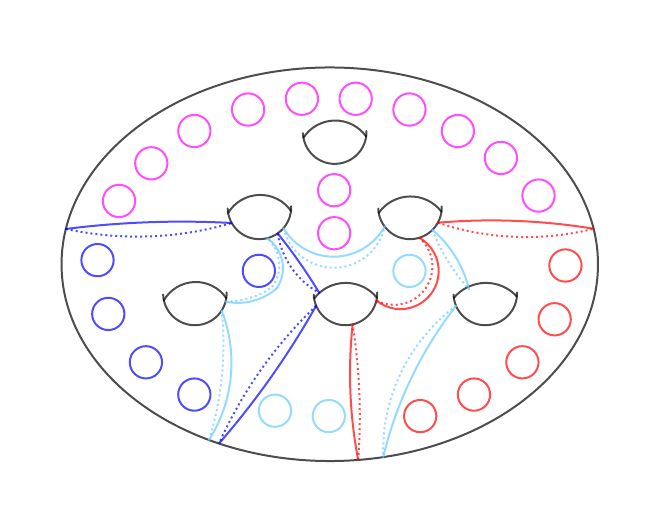}}
\hspace{0.1\textwidth}
\subfloat[$j$]{\includegraphics[width=0.4\textwidth]{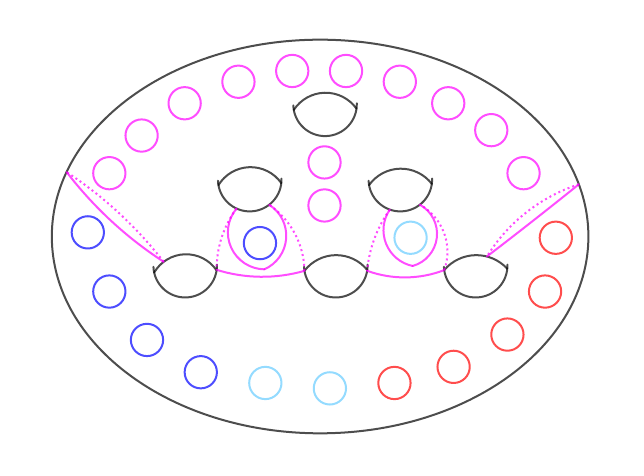}}
\caption{Embeddings of $\Sigma_1^8$ and $\Sigma_{g_{d-1}}^{(d-1)^2}$ into $\Sigma_{g_d}^{d^2}$ for $d=5$}
\label{fig:Embeddings}
\end{figure}

From this topological perspective, we can determine the location of $d^2$ many disjoint sections of the symplectic Lefschetz fibration whose total space is diffeomorphic to $\mathbb{P}^2\#d^2\overline{\mathbb{P}^2}$, as well as an explicit system of vanishing cycles for arbitrary $d\geq 4$. Another reason for the significance of finding the explicit words for Lefschetz pencils on $\mathbb{P}^2$ is its potential to serve as the building blocks for constructing new symplectic Lefschetz pencils or fibrations in the positive signature region. Moreover, such a boundary-interior relation itself corresponds to a symplectic surgery operation. In particular, the relation in the case of $d=4$ corresponds to the symplectic sum along two genus-$3$ symplectic surfaces with self-intersections $-16$ and $16$, respectively (cf.\cite{EG:10}).
Also, as a byproduct, in the proof of Theorem \ref{thm:first}, we obtain a genus $(d-2)$ Lefschetz pencil on $\mathbb{P}^2\#\overline{\mathbb{P}^2}$ with $4d-4$ base points and $8d-12$ critical points for each $d\geq 4$.
\begin{remark}
In fact, the total space $X_d$ in Theorem \ref{thm:first} is symplectomorphic to $\mathbb{P}^2$ since Taubes proved that $\mathbb{P}^2$ has a unique symplectic structure in \cite{Tau:95}. However, we don't know if $f_d:X_d\dashrightarrow\mathbb{P}^1$ is holomorphic or not. A Lefschetz pencil $f:X\dashrightarrow\mathbb{P}^1$ is said to be holomorphic if there exists a complex structure on $X$ such that $f$ is holomorphic and we can take biholomorphic $\phi,\psi,$ and $\xi$ in the definition of a Lefschetz pencil in section~\ref{sec:LP}.
\end{remark}
Auroux and Katzarkov \cite{AK:08} found a degree doubling formula which allows us to obtain the monodromy factorization of the Lefschetz pencil of degree $2d$ curves from those of the Lefschetz pencil of degree $d$ curves based on the braid monodromy technique. 
The above Theorem \ref{thm:first} can be considered as another inductive formula of the monodromy factorization for arbitrary $d$ (not only for even $d$) from that for $d-1$ based on the topological construction, so called the breeding method. 

\begin{corollary}\label{cor}
For arbitrary $d\geq 4$, there exists a symplectic Lefschetz pencil $f_d:X_d\dashrightarrow\mathbb{P}^1$ of genus $g_d=\frac{(d-1)(d-2)}{2}$ with $d^2$ base points and $3(d-1)^2$ critical points whose total space $X_d$ is diffeomorphic to $\mathbb{P}^2$ which can be constructed topologically by breeding $\frac{(d-1)(d-2)}{2}$ copies of the monodromy relations of genus $1$ holomorphic Lefschetz pencils on $\mathbb{P}^2$.  
\end{corollary}

\begin{remark}
We conjecture that the holomorphic Lefschetz pencil $h_d:\mathbb{P}^2\dashrightarrow\mathbb{P}^1$ for any $d\geq 4$ can be constructed topologically by breeding $\frac{(d-1)(d-2)}{2}$ copies of genus $1$ holomorphic Lefschetz pencils as in Corollary \ref{cor}, where we only need several $8$-holed torus relations and a single $9$-holed torus relation. We prove this for $d=4$ in the following Theorem \ref{thm:second} and leave the case $d\geq 5$ for future research.   
\end{remark}

\begin{theorem}\label{thm:second} 
\begin{itemize}
\item[(a)] The monodromy factorization of the genus $3$ holomorphic Lefschetz pencil $h_4:\mathbb{P}^2\dashrightarrow\mathbb{P}^1$, which is the composition of the Veronese embedding $v_{4}:\mathbb{P}^2\hookrightarrow\mathbb{P}^{14}$ of degree $4$ and a generic projection $\mathbb{P}^{14}\dashrightarrow\mathbb{P}^1$ is given by the following relation in $\Gamma_3^{16}$:
$$t_{c_{27}}\cdots t_{c_2}t_{c_1}=t_{\delta_1}t_{\delta_2}\cdots t_{\delta_{16}},$$
where the vanishing cycles $c_i$'s are shown as in Figure~\ref{fig:findvanishingcycles}. 
\item[(b)] 
The genus $3$ holomorphic Lefschetz pencil $h_4:\mathbb{P}^2\dashrightarrow\mathbb{P}^1$ can be constructed topologically by breeding three copies of the monodromy relations corresponding to the genus $1$ holomorphic Lefschetz pencils. 
 
\end{itemize}
\end{theorem}
The proof of Theorem~\ref{thm:second}(a) relies on the braid monodromy technique, developed by Moishezon and Teicher in \cite{MT:BGT1, MT:BGT2, MT:BGT3, MT:BGT4}, which we review in section~\ref{subsec:braidmonodromy}, and we introduce the $9$-holed torus breeding in section~\ref{subsec:breeding} for
the topological construction of Theorem~\ref{thm:first} and Theorem~\ref{thm:second}(b) . 

\section*{Acknowledgements}
The author is grateful to Jongil Park and Ki-Heon Yun for their interests and comments, and to the organizers of the conference "Current Developments in Algebra and Topology 2024" and "Deformation of Complex Singularities and Related topics" for the opportunity to present this work. This work was supported by the National Research Foundation of Korea (NRF) grant funded by the Korean Government (RS-2024-00392067). 

\section{Preliminaries}
\subsection{Lefschetz pencil and its monodromy factorization} \label{sec:LP}
Let X be a smooth, connected, closed, oriented $4$-manifold. 
A Lefschetz pencil on $X$ is a smooth surjective map $f:X\setminus B\rightarrow \mathbb{P}^1$ defined on the complement of a nonempty finite subset $B\subset X$ that is a submersion except at $C=\{p_1,\cdots,p_n\}$ satisfying the following conditions:
\begin{itemize}
\item for any critical point $p\in C$, 
there exist orientation-preserving complex charts $(U,\phi:U\rightarrow\mathbb{C}^2)$ about $p$ and $(V,\psi:V\rightarrow\mathbb{C})$ about $f(p)$ 
such that $\psi\circ f\circ\phi^{-1}(z,w)=z^2+w^2$, 
\item for any $b\in B$, called the base point, there exists a complex chart $(U,\phi)$ of $b$ compatible with the orientation of $X$ and an orientation preserving diffeomorphism $\xi:\mathbb{P}^1\rightarrow\mathbb{P}^1$ such that $\xi\circ f\circ\phi^{-1}(z,w)=[z:w]$ 
\item the restriction $f|_{crit(f)}$ is injective.
\end{itemize}

When the base locus $B$ is empty, the map $f:X\rightarrow\mathbb{P}^1$ defined above is called a Lefschetz fibration.
A fiber of a Lefschetz pencil or fibration containing a critical point is called a singular fiber, and it has a nodal singularity at the critical point by definition. When the regular fiber has genus $g$, we say that $(X,f)$ is a genus-$g$ pencil or fibration. Each singular fiber can be obtained from the nearby regular fiber by collapsing a simple closed curve to the critical point. Such a collapsing curve in the regular fiber, associated to each path from the regular value to the critical value, is called the vanishing cycle. And the local monodromy along the loop around one critical value is given by the right-handed Dehn twist along the vanishing cycle. Globally, according to Kas\cite{Kas:80} and Matsumoto\cite{Mat:96}, a genus $g$ Lefschetz fibration is completely characterized by its monodromy representation $\rho:\pi_1(S^2\setminus\{q_1,\cdots,q_n\},q_0)\rightarrow \Gamma_g$, up to equivalence relations, where $q_i=f(p_i)$ for $1\leq i\leq n$ and $\Gamma_g$ denotes the mapping class group of the closed oriented surface $\Sigma_g$ of genus $g$, the group of isotopy classes of all orientation preserving self-diffeomorphisms of $\Sigma_g$. In particular, for a given genus $g$-Lefschetz fibration over a $2$-sphere with $n$ critical points, one can associate a factorization of the identity into a product of $n$ right-handed Dehn twists in $\Gamma_g$. 
In fact, such a factorization is determined up to certain eqivalence relations coming from the choice of the ordered basis of $\pi_1(S^2\setminus\{q_1,\cdots,q_n\},q_0)$ and that of the identification of the fixed regular fiber $f^{-1}(q_0)$ with the standard surface $\Sigma_g$. 
Conversely, provided $g\geq 2$, given such a relation $t_{a_n}\cdots t_{a_2}t_{a_1}=1$ in $\Gamma_g$, one can construct a genus $g$ Lefschetz fibration over a $2$-sphere up to isomorphism. \\
Similarly, for $g\geq 2$, $g=1,b\geq 1$, or $g=0, b\geq 3$, there is a one-to-one correspondence between genus-$g$ Lefschetz pencils $(X,f)$ with $n$ critical points and $b$ base points, up to isomorphism, and positive factorizations of the boundary twist of the form 
\[t_{c_n}\cdots t_{c_2}t_{c_1}=t_{\delta_1}\cdots t_{\delta_b} ~~\text{in}~~ \Gamma_g^{b},\]
up to Hurwitz equivalence relation described below, where $\delta_1,\cdots,\delta_b$ are distinct boundary parallel simple closed curves of $\Sigma_g^b$ and each $c_i$ is a vanishing cycle corresponding to each critical point $p_i$ of $f$. 
The blow up at all the base points of a genus $g$-Lefschetz pencil $(X,f)$ with $b$ base points induces a genus $g$-Lefschetz fibration with $b$ many disjoint sections of self-intersection $-1$. Hence, there exists a positive relation $t_{c_n}\cdots t_{c_2}t_{c_1}=id$ in $\Gamma_g$, and we can find a lift $t_{c_n'}\cdots t_{c_2'}t_{c_1'}=1$ in $\Gamma_{g,b}$ preserving the set of $b$ distinguished points identified with the intersection of sections and the reference fiber, and its further lift $t_{c_n''}\cdots t_{c_2''}t_{c_1''}=t_{\delta_1}t_{\delta_2}\cdots t_{\delta_b}$ in $\Gamma_g^b$ under the capping homomorphism $\Gamma_g^b\rightarrow \Gamma_g$ which is surjective. In fact, the power $m_j$ of each boundary twist $t_{\delta_j}$ in the lift in $\Gamma_g^b$ encodes the self-intersection $-m_j$ of the corresponding section of the Lefschetz fibration in general. Conversely, such a relation in $\Gamma_g^b$ determines a genus $g$ Lefschetz fibration with $b$ disjoint $(-1)$ sections, and a genus $g$ Lefschetz pencil obtained by blowing down those $(-1)$ sections. \\
Two Lefschetz pencils $f_0:X_0\setminus B_0\rightarrow\mathbb{P}^1$ and $f_1:X_1\setminus B_1\rightarrow\mathbb{P}^1$ are said to be isomorphic if there exist orientation-preserving diffeomorphisms $\Phi:X_0\rightarrow X_1$ and $\phi:\mathbb{P}^1\rightarrow \mathbb{P}^1$ s.t. $\Phi(B_0)=B_1$ and $f_1\circ\Phi=\phi\circ f_0$.
Two monodromy factorizations are said to be Hurwitz equivalent if one can be obtained from the other by a finite sequence of the following two operations:
\begin{itemize}
\item $t_{c_n}\cdots t_{c_{i+1}}t_{c_i}\cdots t_{c_1} \sim t_{c_n}\cdots t_{t_{c_{i+1}}(c_i)}t_{c_{i+1}}\cdots t_{c_1}$
\item $t_{c_n}\cdots t_{c_1} \sim t_{\phi(c_n)}\cdots t_{\phi(c_1)}$ for some $\phi\in Mod(\Sigma_g^b;\{s_1,\cdots,s_b\})$
\end{itemize}
The first kind of operation, called the elementary transformation, comes from the choice of the ordered basis $(l_1,\cdots,l_n)$ of $\pi_1(S^2\setminus \{q_1,\cdots,q_n\}, q_0)$ such that each $l_i$ encircles one $q_i$ and $l_1l_2\cdots l_n$ is homotopically trivial. It is often given by a Hurwitz path system, that is an ordered system of simple paths $(\alpha_1,\cdots, \alpha_n)$ such that each $\alpha_i$ is a path from $q_0$ to $q_i$ and ordered counterclockwise around $q_0$, by connecting $q_0$ with a small circle around $q_i$ oriented counterclockwise by $\alpha_i$. 
The second kind of operation, called the global conjugation, comes from the choice of the identification of $\overline{f^{-1}(q_0)}\setminus \nu(B)$ with $\Sigma_g^b$. Note that we can take $\phi$ in $Mod(\Sigma_g^b;\{s_1,\cdots,s_b\})$, a larger group containing $\Gamma_g^b$ as a subgroup, allowing the base points to be interchanged (cf.\cite{BH:16})

\subsection{Braid monodromy technique}
\label{subsec:braidmonodromy}
Let $X\subset\mathbb{P}^N$ be a nonsingular projective surface and we consider generic projections $\pi:\mathbb{P}^N\dashrightarrow\mathbb{P}^2$ and $\pi':\mathbb{P}^2\dashrightarrow\mathbb{P}^1$. Then the restriction $\phi=\pi|_X:X\rightarrow\mathbb{P}^2$ is a branched covering whose branch curve $S\subset\mathbb{P}^2$ is a curve which has only nodes or cusp singularities and the composition $f=\pi'\circ\phi:X\dashrightarrow\mathbb{P}^1$ is a Lefschetz pencil whose fibers are given by the preimage of the fibers of $\pi'$ under $\phi$ and the base locus is given by the preimage of the base point of the pencil of lines on $\mathbb{P}^2$ under $\phi$. One can easily check that the set of critical points of $f$ consists of the critical points of $\phi$ whose image is a point of $S$ that is tangent to the fiber of $\pi'$. 
Moreover, one can read the monodromy of the Lefschetz pencil from the braid monodromy of the branch curve $S$ of $\phi:X\rightarrow\mathbb{P}^2$ particularly around the vertical tangent points, or the branch points of $\pi'|_S:S\rightarrow\mathbb{P}^1$.

For a given algebraic curve $S\subset\mathbb{C}^2$ of degree $n$ possibly with nodes or cusps, the braid monodromy of $S$ can be defined as follows. First we consider the projection $\pi:\mathbb{C}_{x,y}^2\rightarrow\mathbb{C}_{x}^1$ whose generic fiber intersects $S$ transversely. Let $N=\{x\in\mathbb{C}^1_{x}~|~\text{\#}(\pi^{-1}(x)\cap S)<n\}$ and $N'$ be the set of all special points of $S$  which are either nodes, cusps, or the smooth points of $S$ that are tangent to the fiber of $\pi$. Then $N=\pi(N')$. We assume that each fiber $\pi^{-1}(x)$ contains at most one special point. We take $E$ (resp. $D$) be a closed disk on the $x$-axis (resp. $y$-axis) such that $N'$ is contained in $E\times D$ and $N$ is contained in the interior of $E$. We fix a base point $u\in\partial E$ as a real point far enough from $N$ and we denote the set of the intersection points between $\pi^{-1}(u)$ and $S$ by $K=\{q_1,\cdots, q_n\}$. 

\begin{definition}
We define the braid monodromy of $S$ with respect to $(E\times D, \pi, u)$ as a well-defined group homomorphism
\[\rho:\pi_1(E\setminus N,u)\rightarrow B_n[D,K]\]
that maps each loop $l(t)$ to the braid induced by the motion $\{\pi_2\circ\gamma_1(t), \cdots, \pi_2\circ\gamma_n(t)\}$ of $n$ distinct points in $D$, where $\gamma_1(t),\cdots, \gamma_n(t)$ are $n$ paths in $(E\setminus N)\times D$ starting at $q_1,\cdots, q_n$ which are the liftings of the loop $l(t)$ and $\pi_2:E\times D\rightarrow D$ is the canonical projection.
\end{definition} 

One can also define the braid monodromy of an algebraic curve $\overline{S}\subset\mathbb{P}^2$ of degree $n$ by choosing a generic line $L_{\infty}$ at infinity so that it intersects $\overline{S}$ transversely and an affine coordinate $(x,y)$ in $\mathbb{C}^2=\mathbb{P}^2\setminus L_{\infty}$ so that the projection $\pi_1:S\rightarrow\mathbb{C}^1$ given by $(x,y)\mapsto x$ be generic.

Moishezon and Teicher computed the braid monodromy of the branch curve of the generic projection $\pi|_{V_3}:V_3\rightarrow\mathbb{P}^2$ from the Veronese surface $V_3$ of order 3 by using the degeneration and the regeneration technique in \cite{MT:BGT4}. We recall the definition of the projective degeneration below and the regeneration refers to the reverse deformation of the degeneration. 

\begin{definition}
We say that a projective surface $X_0\subset\mathbb{P}^N$ is a projective degeneration of another projective surface $X\subset\mathbb{P}^N$ if there exists an algebraic variety $W$ with its embedding $F:W\subset\mathbb{P}^N\times\mathbb{C}$ such that for the natural projection $\pi:W\rightarrow\mathbb{C}$ which is flat, the restriction $F|_{\pi^{-1}(t)}:\pi^{-1}(t)\rightarrow\mathbb{P}^N$ is a projective embedding for each $t$, $\pi^{-1}(0)\cong X_0$ and $\pi^{-1}(\epsilon)\cong X,$ for $\epsilon\neq0$.
\end{definition}

In \cite{MT:BGT3}, Moishezon and Teicher found a projective degneration of the surface $V_3$ into a union of planes intersecting along lines. In \cite{MT:BGT4}, they computed the braid monodromy of the branch curve of the projection from the degenerated surface $(V_3)_0$ to $\mathbb{P}^2$, which is a line arrangement in $\mathbb{P}^2$, and then analyzed how the branch curve and its braid monodromy change after regeneration. 
The branch curve, after regeneration, consists of the local configuration obtained after regeneration around each vertex, where $m\geq 2$ lines intersect in $X_0$, and additional nodes created after projecting to $\mathbb{P}^2$. In general, the degenerated branch curve may have singularity of arbitrary multiplicity $m$. However, if we restrict our attention to the Veronese surface $V_d$,
then only $m=2$ or $6$ is possible. Moishezon and Teicher computed the local braid monodromy $H_v$ around each vertex $v$ of type $2$-point or $6$-point after regeneration. 
\begin{lemma}\cite[Lemma 1]{MT:BGT4} 
Let $v=L_i\cap L_j$ be a $2$-point. Assume $i<j$ and $L_i$ regenerates before $L_j$ does. Then after regeneration around $v$, one branch point appears and the braid monodromy around the branch point appearing is given by the half-twist about the path shown in Figure \ref{fig:2pt}(b).

\label{lem:2pt}
\end{lemma}

\begin{figure}[h]
\centering
\begin{tikzpicture}
\fill (0.5,0) circle (1.5pt);
\fill (-0.5,0) circle (1.5pt);
\draw (-0.5,0) to[out=0,in=180] (0.5,0);
\node at (0.5,-0.5) {j};
\node at (-0.5,-0.5) {i};
\node at (0,-1) {(a)};
\fill (4.5,0) circle (1.5pt);
\fill (5,0) circle (1.5pt);
\fill (5.5,0) circle (1.5pt);
\fill (6,0) circle (1.5pt);
\draw (4.5,0) to[out=330,in=180] (6,-0.3);
\draw (6,-0.3) to[out=0,in=360] (6,0.3);
\draw (6,0.3) to[out=180,in=30] (5,0);
\node at (4.5,-0.5) {i};
\node at (5,-0.5) {i'};
\node at (5.5,-0.5) {j};
\node at (6,-0.5) {j'};
\node at (5.25,-1) {(b)};
\end{tikzpicture}

\caption{Lefschetz vanishing cycles of (a) a $2$-point and (b) the branch point after regeneration around a $2$-point}
\label{fig:2pt}
\end{figure}
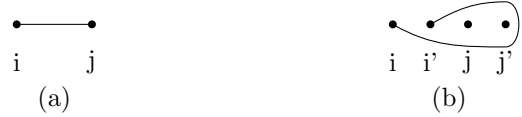

\begin{center}
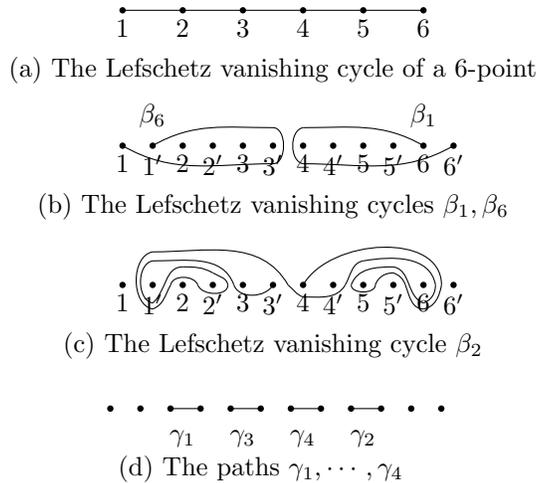
\begin{figure}[h]
\begin{tikzpicture}[scale=0.8]
\fill (1,0) circle (1.5pt) node[anchor=north] {$1$};

\fill (2,0) circle (1.5pt) node[anchor=north] {$2$};

\fill (3,0) circle (1.5pt) node[anchor=north] {$3$};

\fill (4,0) circle (1.5pt) node[anchor=north] {$4$};

\fill (5,0) circle (1.5pt) node[anchor=north] {$5$};

\fill (6,0) circle (1.5pt) node[anchor=north] {$6$};

\draw (1,0) to[out=0,in=180] (2,0);
\draw (2,0) to[out=0,in=180] (3,0);
\draw (3,0) to[out=0,in=180] (4,0);
\draw (4,0) to[out=0,in=180] (5,0);
\draw (5,0) to[out=0,in=180] (6,0);

\node at (3.5,-1) {(a) The Lefschetz vanishing cycle of a $6$-point};
\end{tikzpicture}

\begin{tikzpicture}[scale=0.8]
\fill (1,0) circle (1.5pt) node[anchor=north] {$1$};
\fill (1.5,0) circle (1.5pt) node[anchor=north] {$1'$};
\draw (1,0) to[out=330,in=180] (3.5,-0.3);
\draw (3.5,-0.3) to[out=0,in=360] (3.5,0.3);
\draw (3.5,0.3) to[out=180,in=30] (1.5,0);

\fill (2,0) circle (1.5pt) node[anchor=north] {$2$};
\fill (2.5,0) circle (1.5pt) node[anchor=north] {$2'$};
\fill (3,0) circle (1.5pt) node[anchor=north] {$3$};
\fill (3.5,0) circle (1.5pt) node[anchor=north] {$3'$};
\fill (4,0) circle (1.5pt) node[anchor=north] {$4$};
\fill (4.5,0) circle (1.5pt) node[anchor=north] {$4'$};
\fill (5,0) circle (1.5pt) node[anchor=north] {$5$};
\fill (5.5,0) circle (1.5pt) node[anchor=north] {$5'$};
\fill (6,0) circle (1.5pt) node[anchor=north] {$6$}; 
\fill (6.5,0) circle (1.5pt) node[anchor=north] {$6'$};

\draw (6,0) to[out=150,in=360] (4,0.3);
\draw (4,0.3) to[out=190,in=170] (4,-0.3);
\draw (4,-0.3) to[out=0,in=210] (6.5,0);
\node at (6,0.5) {$\beta_1$};
\node at (1.5,0.5) {$\beta_6$};

\node at (3.5,-1) {(b) The Lefschetz vanishing cycles $\beta_1, \beta_6$};
\end{tikzpicture}

\begin{tikzpicture}[scale=0.8]
\fill (1,0) circle (1.5pt) node[anchor=north] {$1$};
\fill (1.5,0) circle (1.5pt) node[anchor=north] {$1'$};
\fill (2,0) circle (1.5pt) node[anchor=north] {$2$};
\fill (2.5,0) circle (1.5pt) node[anchor=north] {$2'$};
\fill (3,0) circle (1.5pt) node[anchor=north] {$3$};
\fill (3.5,0) circle (1.5pt) node[anchor=north] {$3'$};
\fill (4,0) circle (1.5pt) node[anchor=north] {$4$};
\fill (4.5,0) circle (1.5pt) node[anchor=north] {$4'$};
\fill (5,0) circle (1.5pt) node[anchor=north] {$5$};
\fill (5.5,0) circle (1.5pt) node[anchor=north] {$5'$};
\fill (6,0) circle (1.5pt) node[anchor=north] {$6$}; 
\fill (6.5,0) circle (1.5pt) node[anchor=north] {$6'$};

\draw (3.5,0) to[out=240,in=290] (2.85,0);
\draw (2.85,0) to[out=110,in=0] (1.5,0.4);
\draw (1.5,0.4) to[out=190,in=170] (1.5,-0.3);
\draw (1.5,-0.3) to[out=40,in=180] (2,0.3);
\draw (2,0.3) to[out=0,in=150] (2.5,0.2);
\draw (2.5,0.2) to[out=330, in=100] (2.75,0);
\draw (2.75,0) to[out=270,in=290] (2.25,0);
\draw (2.25,0) to[out=110,in=80] (1.8,0);
\draw (1.8,0) to[out=240,in=20] (1.5,-0.4);
\draw (1.5,-0.4) to[out=170,in=190] (1.5,0.55);
\draw (1.5,0.55) to[out=360,in=120] (3.75,0);
\draw (3.75,0) to[out=320,in=160] (4,-0.2);
\draw (4,-0.2) to[out=0,in=240] (4.7,0);
\draw (4.7,0) to[out=60,in=190] (5,0.4);
\draw (5,0.4) to[out=10,in=100] (6.2,0);
\draw (6.2,0) to[out=270,in=20] (6,-0.3);
\draw (6,-0.3) to[out=160,in=350] (5.5,0.3);
\draw (5.5,0.3) to[out=180,in=70] (4.8,0);
\draw (4.8,0) to[out=270,in=270] (5.2,0);
\draw (5.2,0) to[out=70,in=180] (5.5,0.2);
\draw (5.5,0.2) to[out=340,in=110] (5.7,0);
\draw (5.7,0) to[out=290,in=180] (6,-0.4);
\draw (6,-0.4) to[out=10,in=260] (6.3,0);

\draw (6.3,0) to[out=90,in=60] (4,0);

\node at (3.5,-1) {(c) The Lefschetz vanishing cycle $\beta_2$};
\end{tikzpicture}

\begin{tikzpicture}[scale=0.8]
\node at (1,0.5) {}; 
\fill (1,0) circle (1.5pt); 
\fill (1.5,0) circle (1.5pt); 
\fill (2,0) circle (1.5pt); 
\draw (2,0) to[out=0,in=180] (2.5,0);
\node at (2.2,-0.5) {$\gamma_1$};
\fill (2.5,0) circle (1.5pt); 
\fill (3,0) circle (1.5pt); 
\draw (3,0) to[out=0,in=180] (3.5,0);
\node at (3.2,-0.5) {$\gamma_3$};
\fill (3.5,0) circle (1.5pt); 
\fill (4,0) circle (1.5pt); 
\draw (4,0) to[out=0,in=180] (4.5,0);
\node at (4.2,-0.5) {$\gamma_4$};
\fill (4.5,0) circle (1.5pt); 
\fill (5,0) circle (1.5pt); 
\draw (5,0) to[out=0,in=180] (5.5,0);
\node at (5.2,-0.5) {$\gamma_2$};
\fill (5.5,0) circle (1.5pt);
\fill (6,0) circle (1.5pt);  
\fill (6.5,0) circle (1.5pt); 

\node at (3.5,-1) {(d) The paths $\gamma_1,\cdots,\gamma_4$};
\end{tikzpicture}
\caption{Lefschetz vanishing cycles of the branch points after regeneration around a $6$-point}
\label{fig:6point}
\end{figure}
\end{center}

\begin{lemma}\cite[Lemma 5,6,7,8,9, and 11]{MT:BGT4} 
Let $v=L_1\cap L_ 2\cap L_3\cap L_4\cap L_5\cap L_6$ be a $6$-point. Assume $L_1$ and $L_6$ regenerate first at the same time, then $L_2$ and $L_5$ regenerate at the same time, and lastly $L_3$ and $L_4$ regenerate.
After regeneration around a $6$-point, six branch points appear and there exists a system of simple paths $(\alpha_1,\cdots,\alpha_6)$ from a regular value $a_0$ to one of the image of the six branch points such that the corresponding Lefschetz vanishing cycle of the braid monodromy is given by the path $\beta_i$ (for $i=1,\cdots 6$), where $\beta_1$ and $\beta_6$ are shown in Figure~\ref{fig:6point}(b), $\beta_2$ is shown in Figure~\ref{fig:6point}(c), and
$\beta_3=\tau_{\gamma_3}^{-1}\tau_{\gamma_4}^{-1}(\beta_2), \beta_4=\tau_{\gamma_1}^{-1}\tau_{\gamma_2}^{-1}(\beta_2), \beta_5=\tau_{\gamma_1}^{-1}\tau_{\gamma_2}^{-1}\tau_{\gamma_3}^{-1}\tau_{\gamma_4}^{-1}(\beta_2)$
where the paths $\gamma_1,\cdots,\gamma_4$ are given in Figure~\ref{fig:6point}(d).
\label{lem:6pt}
\end{lemma}

\section{Topological construction of symplectic Lefschetz pencils on $\mathbb{P}^2$}
\begin{center}
\begin{figure}
\includegraphics[width=0.55\textwidth]{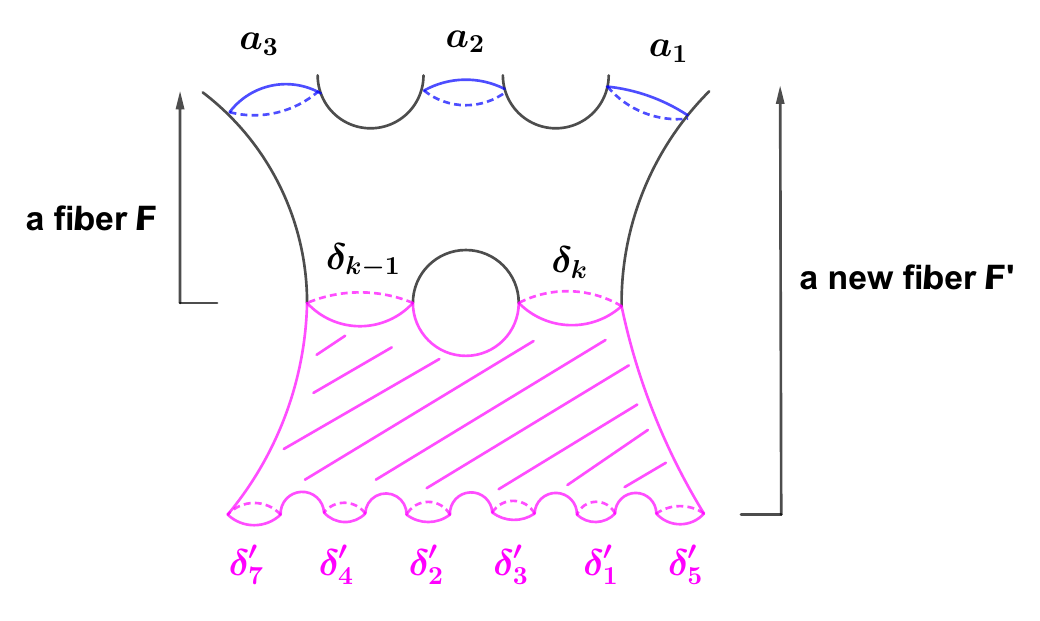}
\caption{A $9$-holed torus breeding}
\label{fig:nineholedtorusbreeding}
\end{figure}
\end{center}
\subsection{A $9$-holed torus breeding}
\label{subsec:breeding}
In order to prove Theorem \ref{thm:first}, we employ the breeding technique in the mapping class group which is useful for constructing a new Lefschetz pencil from the old one. For example, one can refer to \cite{KO:08,Ham:17} for a lantern breeding and \cite{Tan:12} for an $8$-holed torus breeding. Here, we introduce a $9$-holed torus breeding which allows to construct a genus $g+1$ Lefschetz pencil with $k+4$ base points and $N+7$ critical points from a given genus $g$ Lefschetz pencil with $k$ base points and $N$ critical points.
Suppose that we have a positive factorization of the boundary multitwist of the following form :
\[(t_{a_3} t_{a_2} t_{a_1})\cdot W = (t_{\delta_k} t_{\delta_{k-1}})\cdot t_{\delta_{k-2}}\cdots t_{\delta_1} ~\text{in} ~ \Gamma_g^k,\]
where three vanishing cycles $\{a_1,a_2,a_3\}$ of the corresponding Lefschetz pencil together with its two boundary parallel curves $\{\delta_{k-1},\delta_k\}$ cobound a subsurface diffeomorphic to $\Sigma_0^5$ embedded in $\Sigma_g^{k}$ as in Figure \ref{fig:nineholedtorusbreeding}. Then we can obtain the following new positive factorization of the boundary multitwist :
\[(t_{B_{12}}\cdots t_{B_4} t_{B_3})\cdot W = (t_{\delta_7'}t_{\delta_5'}t_{\delta_4'}t_{\delta_3'}t_{\delta_2'}t_{\delta_1'})\cdot t_{\delta_{k-2}}\cdots t_{\delta_1} ~\text{in} ~ \Gamma_{g+1}^{k+4}.\] 
where the new fiber diffeomorphic to $\Sigma_{g+1}^{k+4}$ is obtained by attaching $\Sigma_0^8$, the pink region in Figure \ref{fig:nineholedtorusbreeding}, to the old fiber diffeomorphic to $\Sigma_g^{k}$ along $\delta_{k-1}$ and $\delta_k$.
This is a consequence of a $9$-holed torus relation of the form \[t_{B_{12}}\cdots t_{B_4}t_{B_3}t_{\delta_k}t_{\delta_{k-1}}=t_{\delta_7'}t_{\delta_5'}t_{\delta_4'}t_{\delta_3'}t_{\delta_2'}t_{\delta_1'}t_{a_3}t_{a_2}t_{a_1},\]
whose support is a subsurface diffeomorphic to $\Sigma_1^9$ cobounded by $a_1,a_2,a_3$ and $\delta_1',\delta_2',\delta_3',\delta_4',\delta_5',\delta_7'$ as shown in Figure \ref{fig:nineholedtorusbreeding}. 
We call this operation of building a new Lefschetz pencil the $9$-holed torus breeding. 

\subsection{Topological construction of Lefschetz pencil on $\mathbb{P}^2$ for $d=4$}

\begin{center}
\begin{figure}
\includegraphics[width=0.5\textwidth]{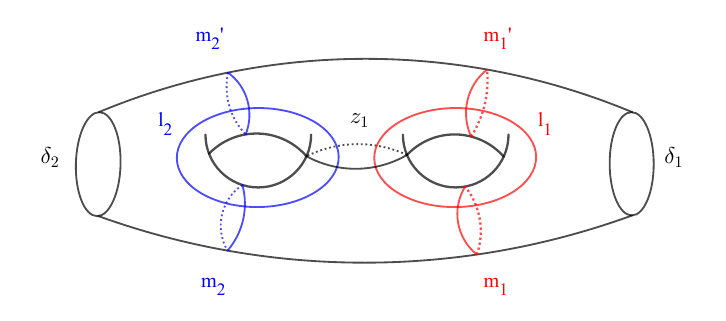}
\caption{Two copies of the three holed torus in $\Sigma_2^2$}
\label{fg:twothreeholedtorus}
\end{figure}

\begin{figure}
\includegraphics[width=0.6\textwidth]{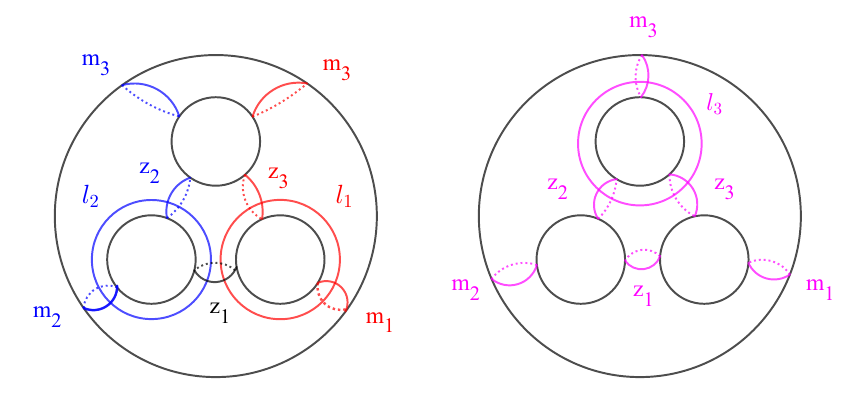}
\caption{Embedding of $\Sigma_2^2$ and $\Sigma_1^3$ into $\Sigma_3$}
\label{fg:embeddingsintotheclosedthree}
\end{figure}
\end{center}

\begin{lemma}
\begin{itemize}
\item[(a)] In $\Gamma_3$, the following relation holds
\[(t_{m_3}t_{z_2}t_{z_3}t_{t_{m_3}^{-1}(l_3)}t_{z_2}t_{z_3}t_{l_3}t_{z_2}t_{z_3}t_{t_{m_3}(l_3)})(t_{t_{m_2}^{-1}(l_2)}t_{z_2}t_{z_1}t_{l_2}t_{z_2}t_{z_1}t_{t_{m_2}(l_2)}t_{m_2})\]
\[(t_{m_1}t_{z_1}t_{t_{m_1}^{-1}(l_1)}t_{z_3}t_{z_1}t_{l_1}t_{z_3}t_{z_1}t_{t_{m_1}(l_1)})=1,\]
where all the Dehn twist curves are as in Figure \ref{fg:embeddingsintotheclosedthree}.
This corresponds to a genus $3$ Lefschetz fibration over $S^2$ with $27$ critical points whose total space is diffeomorphic to $\mathbb{P}^2\#16\overline{\mathbb{P}^2}$.
\item[(b)] Moreover, this genus $3$ Lefschetz fibration can be obtained by blowing up at all the base points from the genus $3$ Lefschetz pencil $f_4:X_4\dashrightarrow \mathbb{P}^1$ with $16$ base points and $27$ critical points whose monodromy factorization is given by the following relation in $\Gamma_3^{16}$: 
\[(t_{a_1^{(3)}}t_{b_2^{(3)}}t_{b_3^{(3)}}t_{a_4^{(3)}}t_{b_4^{(3)}}t_{b_5^{(3)}}t_{b_6^{(3)}}t_{\tilde{b}_7^{(3)}}t_{\tilde{b}_8^{(3)}}t_{\tilde{b}_1^{(3)}})(t_{b_7^{(2)}}t_{b_8^{(2)}}t_{a_1^{(2)}}t_{b_1^{(2)}}t_{b_2^{(2)}}t_{\tilde{b}_3^{(2)}}t_{\tilde{b}_4^{(2)}}t_{\tilde{b}_6^{(2)}})\]\[(t_{\tilde{b}_1^{(1)}}t_{\tilde{b}_2^{(1)}}t_{\tilde{b}_3^{(1)}}t_{b_5^{(1)}}t_{b_6^{(1)}}t_{a_7^{(1)}}t_{b_7^{(1)}}t_{b_8^{(1)}}t_{b_9^{(1)}})\]
\[=(t_{\delta_7'}t_{\delta_4'}t_{\delta_2'}t_{\delta_1'}t_{\delta_3'}t_{\delta_5'})(t_{d_3}t_{d_5}t_{d_6}t_{d_7}t_{d_8})(t_{\delta_4}t_{\delta_5}t_{\delta_7}t_{\delta_8}t_{\delta_9}).\]
where all the vanishing cycles are as in Figure \ref{fig:vanishingcycles3}.
\end{itemize}
\begin{figure}
\includegraphics[width=1.0\textwidth]{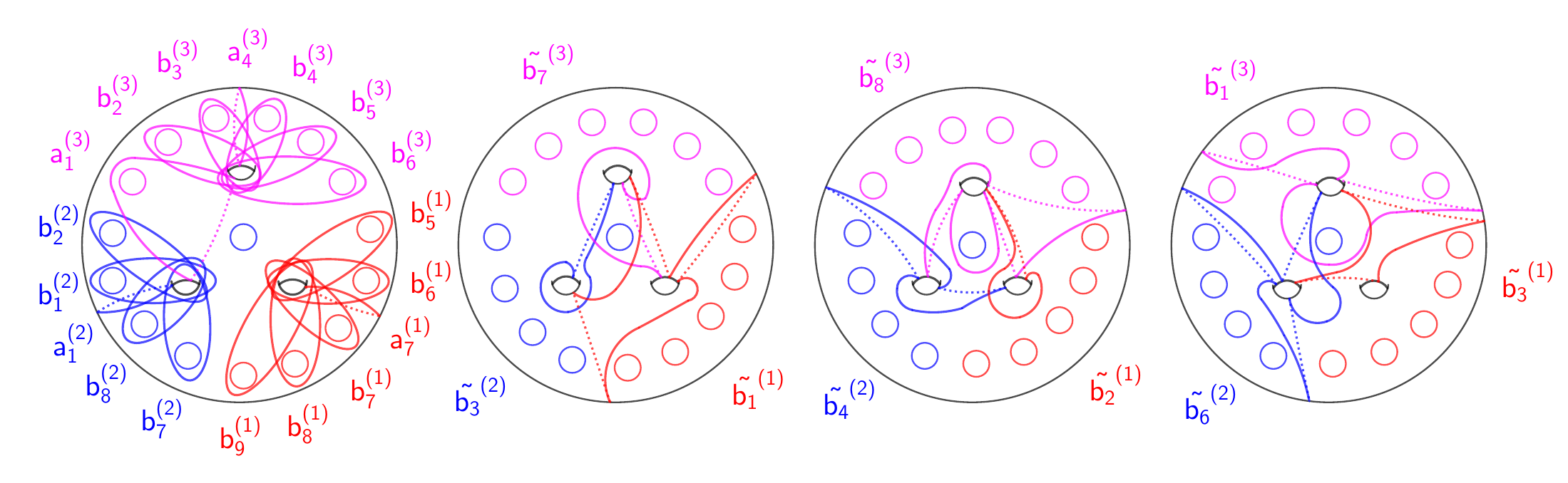}
\caption{Vanishing cycles for $f_4:X_4\dashrightarrow\mathbb{P}^1$}
\label{fig:vanishingcycles3}
\end{figure}
\begin{proof}
(a) First, we construct a relation of a genus $2$ Lefschetz pencil with two base points and $20$ critical points by breeding two copies of the $3$-holed torus relations. For the embeddings of $\Sigma_1^3$ into $\Sigma_2^2$, see the Figure ~\ref{fg:twothreeholedtorus}. 
The following relations hold in $\Gamma_2^2$, 
\[(t_{m_2}t_{m_2'}t_{z_1}t_{l_2})^3=t_{\delta_2}t_{m_1}t_{m_1'},\]
\[(t_{z_1}t_{l_1})(t_{m_1}t_{m_1'}t_{z_1}t_{l_1})^2(t_{m_1}t_{m_1'})=t_{m_2}t_{m_2'}t_{\delta_1}.\]
Hence, 
\begin{equation}
(t_{z_1}t_{l_2})(t_{m_2}t_{m_2'}t_{z_1}t_{l_2})^2(t_{z_1}t_{l_1})(t_{m_1}t_{m_1'}t_{z_1}t_{l_1})^2=t_{\delta_1}t_{\delta_2}.
\label{eq:fromtwothreeholedtorus}
\end{equation}
Next, we consider the following third copy of the three holed torus relation whose support is embedded in $\Sigma_3$ as in Figure ~\ref{fg:embeddingsintotheclosedthree}:
\[(t_{m_3}t_{z_2}t_{z_3}t_{l_3})^3=t_{m_2}t_{z_1}t_{m_1}.\]
After several Hurwitz moves and a cyclic permutation, we have
\begin{equation}
(t_{m_3}t_{m_3})(t_{z_2}t_{z_3}t_{l_3}t_{m_3}t_{z_2}t_{z_3}t_{l_3}t_{z_2}t_{z_3}t_{t_{m_3}(l_3)})=t_{m_2}t_{z_1}t_{m_1}.
\label{eq:thirdthreeholedtorus}
\end{equation}
Moreover, we can embed the relation \eqref{eq:fromtwothreeholedtorus} in $\Gamma_2^2$ into $\Gamma_3$ as 
\[(t_{z_1}t_{l_2})(t_{m_2}t_{z_2}t_{z_1}t_{l_2})^2(t_{z_1}t_{l_1})(t_{m_1}t_{z_3}t_{z_1}t_{l_1})^2=t_{m_3}t_{m_3}\]
where the curves are as shown in Figure ~\ref{fg:embeddingsintotheclosedthree}. The Hurwitz moves and a cyclic permutation yield \begin{equation}
(t_{t_{m_2}^{-1}(l_2)}t_{z_2}t_{z_1}t_{l_2}t_{m_2}t_{z_2}t_{z_1}t_{l_2})(t_{z_1}t_{l_1}t_{m_1}t_{z_3}t_{z_1}t_{l_1}t_{z_3}t_{z_1}t_{t_{m_1}(l_1)})(t_{m_1}t_{z_1}t_{m_2})=t_{m_3}t_{m_3}.
\label{eq:firsttwothreeholedtorus}
\end{equation}
Now, by breeding this relation \eqref{eq:firsttwothreeholedtorus} with the relation \eqref{eq:thirdthreeholedtorus} we get the following desired relation in $\Gamma_3$: 
\begin{multline}
(t_{z_2}t_{z_3}t_{l_3}t_{m_3}t_{z_2}t_{z_3}t_{l_3}t_{z_2}t_{z_3}t_{t_{m_3}(l_3)})(t_{t_{m_2}^{-1}(l_2)}t_{z_2}t_{z_1}t_{l_2}t_{m_2}t_{z_2}t_{z_1}t_{l_2})\\
(t_{z_1}t_{l_1}t_{m_1}t_{z_3}t_{z_1}t_{l_1}t_{z_3}t_{z_1}t_{t_{m_1}(l_1)})=1.
\label{eq:genusthreeLF}
\end{multline}
We will show, in the proof of the second statement (b) of this Lemma, that the corresponding genus $3$ Lefschetz fibration over $S^2$ admits $16$ disjoint $(-1)$-sections, denoted by $S_1,\cdots, S_{16}$. According to the theorem by Gompf, the total space $\tilde{X_4}=X_4\#16\overline{\mathbb{P}^2}$ admits a symplectic form $w$ such that all these sections are symplectic. Hence, we may assume that $[S_i]\cdot [w]>0$ for all $i$ by changing the orienation if necessary. Now, by the result of Sato in \cite{Sa:08}, if $\tilde{X_4}$ is neither rational nor ruled, then the number of $E_i$'s, the poincare dual of the homology classes $[S_i]$ of smoothly embedded spheres of square $-1$ such that $E_i\cdot[w]>0$, must be at most $2g-2=2$, which is a contradiction. Therefore, $\tilde{X_4}$ is either rational or ruled. 
Moreover, since the vanishing cycles include all the meridian and the longitude curves $\{m_1,l_1,m_2,l_2,m_3,l_3\}$ its fundamental group is trivial and hence $\tilde{X_4}$ cannot be irrational ruled. The Euler characteristic of the total space of a genus $3$ Lefschetz fibration over $S^2$ with $27$ critical points is $2\cdot(2-6)+27=19$. The signature of the total space of the Lefschetz fibration is equal to $-15$, determined by the signature of the relations by Endo and Nagami in \cite{EN:04}, since the signature $I_3$ of the $3$-holed torus relation $E$ is equal to $-5$. Therefore, $\tilde{X_4}$ is diffeomorphic to $\mathbb{P}^2\#16\overline{\mathbb{P}^2}$. \\
(b) In order to find the lift of the relation \eqref{eq:genusthreeLF} in $\Gamma_3$ to the desired relation in $\Gamma_3^{16}$, we will breed two copies of the $8$-holed torus relations and one copy of the $9$-holed torus relation instead of breeding three copies of the $3$-holed torus relations. \\
\textbf{Step1} : First, we need to rearrange the monodromy factorizations for three building blocks so that we can obtain the desired relation by breeding method. \\
For simplicity, we will denote the right-handed Dehn twist along $\alpha$ also by $\alpha$, and its inverse by $\overline{\alpha}$.
To obtain the first building block, we begin with the $9$-holed torus relation corresponding to the genus $1$ holomorphic Lefschetz pencil $f_n:\mathbb{P}^2\dashrightarrow\mathbb{P}^1$ which is the composition of the Veronese embedding $v_3:\mathbb{P}^2\hookrightarrow\mathbb{P}^9$ of degree $3$ and a generic projection $\mathbb{P}^9\dashrightarrow\mathbb{P}^1$ (we adopt the same notation for the vanishing cycles below as in Figure $24$(a) in \cite{HH:21}) and apply the following sequence of Hurwitz moves.  
\begin{align*}
\prod_{j=1}^{9}T_{\delta_j}& =a_1b_1b_2b_3a_4b_4b_5b_6a_7b_7b_8b_9\\
                           &\simeq a_1b_1b_2b_3c_{11}'c_{12}b_5b_6a_7b_7b_8b_9\\
                           &\simeq a_1b_1b_2c_{11}'{}_{\overline{c}_{11}'}(b_3)a_4b_5b_6a_7b_7b_8b_9\\
                           &\simeq a_1b_1c_{11}b_2{}_{\overline{c}_{11}'}(b_3)a_4b_5b_6a_7b_7b_8b_9\\
                           &\simeq a_1c_{11}{}_{\overline{c}_{11}}(b_1)b_2{}_{\overline{c}_{11}'}(b_3)a_4b_5b_6a_7b_7b_8b_9\\
                           &\simeq c_{11}a_1a_4\widetilde{b_1}\widetilde{b_2}\widetilde{b_3}b_5b_6a_7b_7b_8b_9
\end{align*}
where $\widetilde{b_1}={}_{\overline{a}_4\overline{c}_{11}}(b_1), \widetilde{b_2}={}_{\overline{a}_4}(b_2), \widetilde{b_3}={}_{\overline{a}_4\overline{c}_{11}'}(b_3)$. If we cap off one boundary component $\delta_1$, then we obtain the $8$-holed torus relation corresponding to the genus $1$ holomorphic Lefschetz pencil on $\mathbb{P}^2\#\overline{\mathbb{P}^2}$. 
Next, we start with the $8$-holed torus relation corresponding to the genus $1$ holomorphic Lefschetz pencil $f_s:\mathbb{P}^1\times\mathbb{P}^1\dashrightarrow\mathbb{P}^1$ of curves of bidegree $(2,2)$ (we use the same notation for the vanishing cycles below as in Figure $27$(a) in \cite{HH:21}) and apply the following sequence of Hurwitz moves to obtain the second building block. 
\begin{align*}
\prod_{j=1}^{8}T_{\delta_j}&=a_1b_1b_2a_3b_3b_4a_5b_5b_6a_7b_7b_8\\
                           &\simeq a_1b_1b_2a_3b_3b_4{}_{\overline{c}_5'}(c_3){}_{a_5}(b_6)a_5a_7b_7b_8\\
                           &\simeq a_1b_1b_2a_3b_3c_3b_4{}_{a_5}(b_6)c_1a_5b_7b_8\\
                           &\simeq a_1b_1b_2a_3b_3{}_{c_3}(b_4){}_{c_3a_5}(b_6)c_3c_1a_5b_7b_8\\
                           &\simeq a_5b_7b_8a_1b_1b_2a_3b_3{}_{c_3}(b_4){}_{c_3a_5}(b_6)c_3c_1\\
                           &\simeq a_3a_5b_7b_8a_1b_1b_2\widetilde{b_3}\widetilde{b_4}\widetilde{b_6}c_3c_1			
\end{align*}
where $\widetilde{b_3}={}_{a_3}(b_3), \widetilde{b_4}={}_{a_3c_3}(b_4), \widetilde{b_6}={}_{a_3c_3a_5}(b_6)$.
The second-to-last equivalence is obtained by a cyclic permutation, which can be derived as a sequence of elementrary transformations. 
Finally, in order to obtain the third building block we apply the following sequence of Hurwitz moves to the $9$-holed torus relation corresponding to the genus $1$ holomorphic Lefschetz pencil of degree $3$ curves in $\mathbb{P}^2$. 
\begin{align*}
\prod_{j=1}^{9}T_{\delta_j}&=a_1b_1b_2b_3a_4b_4b_5b_6a_7b_7b_8b_9\\
                           &\simeq b_2b_3a_4b_4b_5b_6a_7b_7b_8b_9a_1b_1\\
                           &\simeq b_2b_3a_4b_4b_5b_6a_7b_7b_8b_9{}_{a_1}(b_1)a_1\\
                           &\simeq b_2b_3a_4b_4b_5b_6B_1b_7b_8{}_{a_1}(b_1)B_2a_1\\
                           &\simeq a_1b_2b_3a_4b_4b_5b_6B_1b_7b_8{}_{a_1}(b_1)B_2\\
                           &\simeq a_1b_2b_3a_4b_4b_5b_6\widetilde{b_7}\widetilde{b_8}\widetilde{b_1}B_1B_2
\end{align*}
where we denote $B_1=a_7=c_1, B_2={}_{\overline{c}_6'}(b_9)=c_2$ and $\widetilde{b_7}={}_{B_1}(b_7), \widetilde{b_8}={}_{B_1}(b_8), \widetilde{b_1}={}_{B_1a_1}(b_1)$. The first and the fourth equivalence is obtained by a cyclic permutation. \\
\textbf{Step2}: We will construct the desired relation in $\Gamma_3^{16}$ by breeding three relations obtained in the first step. First, we breed the following two copies of the $8$-holed torus relations embedded in $\Sigma_2^{12}$ as in Figure \ref{twoembeddingsof8holedtorus}
\begin{figure}
\includegraphics[width=0.8\textwidth]{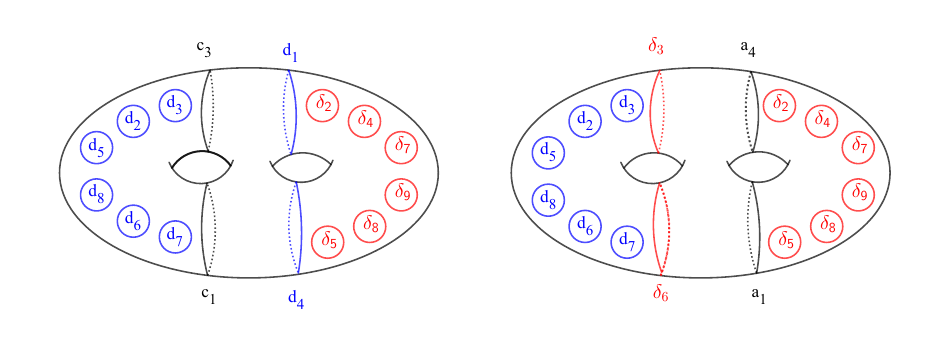}
\caption{Two embeddings of $\Sigma_1^8$ into $\Sigma_2^{12}$}
\label{twoembeddingsof8holedtorus}
\end{figure}
\begin{figure}
\includegraphics[width=0.4\textwidth]{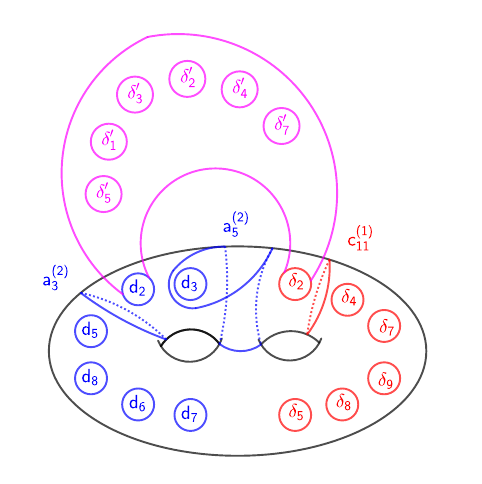}
\caption{A $9$-holed torus breeding applied to the genus $2$ Lefschetz pencil with $12$ base points}
\label{fig:9holedtorusbreeding}
\end{figure}
\begin{align*}
\prod_{j=1}^{8}T_{d_j}&=a_3^{(2)}a_5^{(2)}b_7^{(2)}b_8^{(2)}a_1^{(2)}b_1^{
(2)}b_2^{(2)}\widetilde{b_3}^{(2)}\widetilde{b_4}^{(2)}\widetilde{b_6}^{(2)}c_3^{(2)}c_1^{(2)}\\
\prod_{j=2}^{9}T_{\delta_j}&=a_1^{(1)}a_4^{(1)}c_{11}^{(1)}\widetilde{b_1}^{(1)}\widetilde{b_2}^{(1)}\widetilde{b_3}^{(1)}b_5^{(1)}b_6^{(1)}a_7^{(1)}b_7^{(1)}b_8^{(1)}b_9^{(1)}
\end{align*}
to obtain the following relation in $\Gamma_2^{12}$
\[\begin{split}
&(d_7d_6d_8d_5d_2d_3)(\delta_2\delta_4\delta_7\delta_9\delta_8\delta_5)\\
=&(a_3^{(2)}a_5^{(2)}b_7^{(2)}b_8^{(2)}a_1^{(2)}b_1^{
(2)}b_2^{(2)}\widetilde{b_3}^{(2)}\widetilde{b_4}^{(2)}\widetilde{b_6}^{(2)})(c_{11}^{(1)}\widetilde{b_1}^{(1)}\widetilde{b_2}^{(1)}\widetilde{b_3}^{(1)}b_5^{(1)}b_6^{(1)}a_7^{(1)}b_7^{(1)}b_8^{(1)}b_9^{(1)}),
\end{split}\]
by cancelling $c_3^{(2)}c_1^{(2)}$ with $\overline{\delta_6}\overline{\delta_3}$ and $\overline{d_1}\overline{d_4}$ with $a_1^{(1)}a_4^{(1)}$, respectively.
Using the commutativity relations, we can rewrite this relation in the following form.  
\begin{multline}
(d_3d_5d_6d_7d_8)(\delta_4\delta_5\delta_7\delta_8\delta_9)\\
=(\overline{d}_2\overline{\delta}_2a_3^{(2)}a_5^{(2)}c_{11}^{(1)})(b_7^{(2)}b_8^{(2)}a_1^{(2)}b_1^{
(2)}b_2^{(2)}\widetilde{b_3}^{(2)}\widetilde{b_4}^{(2)}\widetilde{b_6}^{(2)})\\
(\widetilde{b_1}^{(1)}\widetilde{b_2}^{(1)}\widetilde{b_3}^{(1)}b_5^{(1)}b_6^{(1)}a_7^{(1)}b_7^{(1)}b_8^{(1)}b_9^{(1)}).
\label{eq:genustwo}
\end{multline}
Now, we apply the $9$-holed torus breeding, as in Figure \ref{fig:9holedtorusbreeding}, of the following form
\[\delta_7'\delta_4'\delta_2'\delta_1'\delta_3'\delta_5'=(a_1^{(3)}b_2^{(3)}b_3^{(3)}a_4^{(3)}b_4^{(3)}b_5^{(3)}b_6^{(3)}\widetilde{b_7}^{(3)}\widetilde{b_8}^{(3)}\widetilde{b_1}^{(3)})(\overline{\delta}_6'\overline{\delta}_8'\overline{\delta}_9'B_2B_1),\]
to the above relation \eqref{eq:genustwo} in $\Gamma_2^{12}$ by identifying $\delta_6',\delta_8',\delta_9'$ with $a_3^{(2)},a_5^{(2)},c_{11}^{(1)}$ and $B_2,B_1$ with $d_2,\delta_2$, respectively.
Consequently, we get the desired relation in $\Gamma_3^{16}$:
\[
\begin{split}&(\delta_7'\delta_4'\delta_2'\delta_1'\delta_3'\delta_5')(d_3d_5d_6d_7d_8)(\delta_4\delta_5\delta_7\delta_8\delta_9)\\
=&(a_1^{(3)}b_2^{(3)}b_3^{(3)}a_4^{(3)}b_4^{(3)}b_5^{(3)}b_6^{(3)}\widetilde{b_7}^{(3)}\widetilde{b_8}^{(3)}\widetilde{b_1}^{(3)})(b_7^{(2)}b_8^{(2)}a_1^{(2)}b_1^{
(2)}b_2^{(2)}\widetilde{b_3}^{(2)}\widetilde{b_4}^{(2)}\widetilde{b_6}^{(2)})\\
&(\widetilde{b_1}^{(1)}\widetilde{b_2}^{(1)}\widetilde{b_3}^{(1)}b_5^{(1)}b_6^{(1)}a_7^{(1)}b_7^{(1)}b_8^{(1)}b_9^{(1)}).
\end{split}\]
 
\textbf{Step3}: To finish the proof of (b), we show that after taking $\lambda:\Sigma_3^{16}\rightarrow\Sigma_3$ obtained by capping the boundary components with $16$ disks, it is equivalent to the relation obtained in (a). 

\[(\widetilde{b_1}^{(1)}\widetilde{b_2}^{(1)}\widetilde{b_3}^{(1)}b_5^{(1)}b_6^{(1)}a_7^{(1)}b_7^{(1)}b_8^{(1)}b_9^{(1)})\]
\begin{align*}
\xrightarrow{\lambda} &~ {}_{\overline{m}_1\overline{m}_1\overline{z}_3\overline{l}_1}(z_1)\cdot{}_{l_1}(z_3)\cdot{}_{l_1z_1m_1l_1}(z_3)\cdot l_1\cdot l_1\cdot m_1\cdot l_1\cdot l_1\cdot l_1\\
\simeq &~ {}_{\overline{m}_1\overline{m}_1\overline{z}_3\overline{l}_1}(z_1)\cdot l_1\cdot z_3\cdot{}_{z_1m_1l_1}(z_3)\cdot l_1\cdot m_1\cdot l_1\cdot l_1\cdot l_1\\
\simeq &~ {}_{\overline{m}_1\overline{m}_1\overline{z}_3\overline{l}_1}(z_1)\cdot l_1\cdot z_3\cdot{}_{z_1m_1l_1}(z_3)\cdot l_1\cdot{}_{m_1}(l_1)\cdot{}_{m_1}(l_1)\cdot m_1\\
\simeq &~ {}_{\overline{m}_1\overline{m}_1\overline{z}_3\overline{l}_1}(z_1)\cdot l_1\cdot z_3\cdot{}_{z_1m_1l_1}(z_3)\cdot m_1\cdot m_1\cdot l_1\cdot{}_{m_1}(l_1)\cdot m_1\\
\simeq &~ {}_{\overline{m}_1\overline{m}_1\overline{z}_3\overline{l}_1}(z_1)\cdot l_1\cdot z_3\cdot{}_{z_1m_1l_1}(z_3)\cdot m_1\cdot{}_{m_1}(l_1)\cdot{}_{m_1m_1}(l_1)\cdot m_1\cdot m_1\\
\simeq &~ m_1\cdot m_1\cdot{}_{\overline{m}_1\overline{m}_1\overline{z}_3\overline{l}_1}(z_1)\cdot l_1\cdot z_3\cdot{}_{z_1m_1l_1}(z_3)\cdot m_1\cdot{}_{m_1}(l_1)\cdot{}_{m_1m_1}(l_1)\\
\simeq &~ {}_{\overline{z}_3\overline{l}_1}(z_1)\cdot m_1\cdot m_1\cdot l_1\cdot z_3\cdot{}_{\overline{z}_3 m_1\overline{l}_1}(z_1)\cdot m_1\cdot{}_{m_1}(l_1)\cdot{}_{m_1m_1}(l_1)\\
\simeq &~ {}_{\overline{z}_3\overline{l}_1}(z_1)\cdot m_1\cdot m_1\cdot l_1\cdot{}_{m_1\overline{l}_1}(z_1)\cdot m_1\cdot z_3\cdot{}_{m_1}(l_1)\cdot{}_{m_1m_1}(l_1)\\
\simeq &~ {}_{\overline{z}_3\overline{l}_1}(z_1)\cdot m_1\cdot m_1\cdot l_1\cdot m_1\cdot{}_{\overline{l}_1}(z_1)\cdot z_3\cdot m_1\cdot{}_{m_1}(l_1)\\
\simeq &~ {}_{\overline{z}_3\overline{l}_1}(z_1)\cdot m_1\cdot m_1\cdot{}_{l_1}(m_1)\cdot l_1\cdot{}_{\overline{l}_1}(z_1)\cdot z_3\cdot m_1\cdot{}_{m_1}(l_1)\\
\simeq &~ {}_{\overline{z}_3\overline{l}_1}(z_1)\cdot m_1\cdot l_1\cdot m_1\cdot z_1\cdot l_1\cdot z_3\cdot m_1\cdot{}_{m_1}(l_1)\\
\simeq &~ m_1\cdot{}_{\overline{m}_1\overline{z}_3\overline{l}_1}(z_1)\cdot m_1\cdot{}_{\overline{m}_1}(l_1)\cdot z_1\cdot l_1\cdot z_3\cdot m_1\cdot{}_{m_1}(l_1)\\
\simeq &~ m_1\cdot z_1\cdot{}_{\overline{z}_3\overline{m}_1}(l_1)\cdot m_1\cdot{}_{\overline{z}_1\overline{m}_1}(l_1)\cdot l_1\cdot m_1\cdot z_3\cdot{}_{m_1}(l_1)\\
\simeq &~ m_1\cdot z_1\cdot{}_{\overline{z}_3\overline{m}_1}(l_1)\cdot m_1\cdot{}_{\overline{m}_1l_1}(z_1)\cdot m_1\cdot{}_{\overline{m}_1}(l_1)\cdot z_3\cdot{}_{m_1}(l_1)\\
\simeq &~ m_1\cdot z_1\cdot{}_{\overline{z}_3\overline{m}_1}(l_1)\cdot{}_{l_1}(z_1)\cdot m_1\cdot m_1\cdot{}_{\overline{m}_1}(l_1)\cdot z_3\cdot{}_{m_1}(l_1)\\
\simeq &~ m_1\cdot z_1\cdot{}_{\overline{m}_1}(l_1)\cdot z_3\cdot{}_{l_1}(z_1)\cdot l_1\cdot l_1\cdot z_3\cdot{}_{m_1}(l_1)\\
\simeq &~ m_1\cdot z_1\cdot{}_{\overline{m}_1}(l_1)\cdot z_3\cdot l_1\cdot z_1\cdot l_1\cdot z_3\cdot{}_{m_1}(l_1)\\
\simeq &~ m_1\cdot z_1\cdot{}_{\overline{m}_1}(l_1)\cdot z_3\cdot l_1\cdot{}_{z_1}(l_1)\cdot z_1\cdot z_3\cdot{}_{m_1}(l_1)\\
\simeq &~ m_1\cdot z_1\cdot{}_{\overline{m}_1}(l_1)\cdot z_3\cdot z_1\cdot l_1\cdot z_1\cdot z_3\cdot{}_{m_1}(l_1)
\end{align*}

\[(b_7^{(2)}b_8^{(2)}a_1^{(2)}b_1^{
(2)}b_2^{(2)}\widetilde{b_3}^{(2)}\widetilde{b_4}^{(2)}\widetilde{b_6}^{(2)})\]
\begin{align*}
\xrightarrow{\lambda} &~ l_2\cdot l_2\cdot m_2\cdot l_2\cdot l_2\cdot{}_{\overline{l}_2}(z_2)\cdot{}_{m_2\overline{l}_2}(z_1)\cdot{}_{m_2m_2\overline{l}_2}(z_2)\\
\simeq &~ m_2\cdot{}_{\overline{m}_2}(l_2)\cdot{}_{\overline{m}_2}(l_2)\cdot l_2\cdot z_2\cdot l_2\cdot{}_{m_2\overline{l}_2}(z_1)\cdot{}_{m_2m_2\overline{l}_2}(z_2)\\
\simeq &~ {}_{\overline{m}_2}(l_2)\cdot{}_{\overline{m}_2}(l_2)\cdot l_2\cdot z_2\cdot l_2\cdot m_2\cdot{}_{\overline{l}_2}(z_1)\cdot{}_{m_2\overline{l}_2}(z_2)\\
\simeq &~ {}_{\overline{m}_2}(l_2)\cdot l_2\cdot m_2\cdot z_2\cdot l_2\cdot m_2\cdot{}_{\overline{l}_2}(z_1)\cdot{}_{m_2\overline{l}_2}(z_2)\\
\simeq &~ {}_{\overline{m}_2}(l_2)\cdot z_2\cdot{}_{\overline{z}_2}(l_2)\cdot m_2\cdot{}_{l_2}(m_2)\cdot z_1\cdot{}_{\overline{m}_2l_2}(z_2)\cdot l_2\\
\simeq &~ {}_{\overline{m}_2}(l_2)\cdot z_2\cdot{}_{l_2}(z_2)\cdot l_2\cdot m_2\cdot z_1\cdot{}_{\overline{m}_2l_2}(z_2)\cdot l_2\\
\simeq &~ {}_{\overline{m}_2}(l_2)\cdot z_2\cdot l_2\cdot z_2\cdot z_1\cdot{}_{l_2}(z_2)\cdot m_2\cdot l_2\\
\simeq &~ {}_{\overline{m}_2}(l_2)\cdot z_2\cdot z_1\cdot{}_{\overline{z}_1}(l_2)\cdot z_2\cdot{}_{\overline{z}_2}(l_2)\cdot m_2\cdot l_2\\
\simeq &~ {}_{\overline{m}_2}(l_2)\cdot z_2\cdot z_1\cdot{}_{l_2}(z_1)\cdot l_2\cdot z_2\cdot m_2\cdot l_2\\
\simeq &~ {}_{\overline{m}_2}(l_2)\cdot z_2\cdot z_1\cdot l_2\cdot z_1\cdot z_2\cdot m_2\cdot l_2
\end{align*}

\[(a_1^{(3)}b_2^{(3)}b_3^{(3)}a_4^{(3)}b_4^{(3)}b_5^{(3)}b_6^{(3)}\widetilde{b_7}^{(3)}\widetilde{b_8}^{(3)}\widetilde{b_1}^{(3)})\]
\begin{align*}
\xrightarrow{\lambda} &~ z_2\cdot l_3\cdot l_3\cdot m_3\cdot l_3\cdot l_3\cdot l_3\cdot{}_{\overline{l}_3}(z_3)\cdot{}_{\overline{z}_3m_3\overline{l}_3}(z_2)\cdot{}_{m_3m_3}(l_3)\\
\simeq &~ m_3\cdot z_2\cdot{}_{\overline{m}_3}(l_3)\cdot{}_{\overline{m}_3}(l_3)\cdot l_3\cdot l_3\cdot z_3\cdot l_3\cdot{}_{\overline{z}_3m_3\overline{l}_3}(z_2)\cdot{}_{m_3m_3}(l_3)\\
\simeq &~ z_2\cdot{}_{\overline{m}_3}(l_3)\cdot{}_{\overline{m}_3}(l_3)\cdot l_3\cdot l_3\cdot{}_{z_3}(l_3)\cdot{}_{m_3\overline{l}_3}(z_2)\cdot z_3\cdot{}_{m_3m_3}(l_3)\cdot m_3\\
\simeq &~ z_2\cdot{}_{\overline{m}_3}(l_3)\cdot{}_{\overline{m}_3}(l_3)\cdot l_3\cdot l_3\cdot{}_{z_3}(l_3)\cdot
m_3\cdot{}_{\overline{l}_3}(z_2)\cdot z_3\cdot{}_{m_3}(l_3)\\
\simeq &~ z_2\cdot{}_{\overline{m}_3}(l_3)\cdot{}_{\overline{m}_3}(l_3)\cdot l_3\cdot z_3\cdot{}_{l_3}(m_3)\cdot z_2\cdot l_3\cdot z_3\cdot{}_{m_3}(l_3)\\
\simeq &~ z_2\cdot{}_{\overline{m}_3}(l_3)\cdot{}_{\overline{m}_3}(l_3)\cdot l_3\cdot{}_{z_3l_3}(m_3)\cdot z_3\cdot{}_{z_2}(l_3)\cdot z_2\cdot z_3\cdot{}_{m_3}(l_3)\\
\simeq &~ z_2\cdot{}_{\overline{m}_3}(l_3)\cdot{}_{\overline{m}_3}(l_3)\cdot{}_{z_3l_3}(m_3)\cdot{}_{l_3}(z_3)\cdot z_2\cdot l_3\cdot z_2\cdot z_3\cdot{}_{m_3}(l_3)\\
\simeq &~ z_2\cdot{}_{\overline{m}_3}(l_3)\cdot z_3\cdot{}_{\overline{m}_3}(l_3)\cdot{}_{l_3}(z_3)\cdot z_2\cdot l_3\cdot z_2\cdot z_3\cdot{}_{m_3}(l_3)\\
\simeq &~ z_2\cdot{}_{\overline{m}_3}(l_3)\cdot{}_{z_3\overline{m}_3}(l_3)\cdot l_3\cdot z_2\cdot z_3\cdot l_3\cdot z_2\cdot z_3\cdot{}_{m_3}(l_3)\\
\simeq &~ z_2\cdot l_3\cdot m_3\cdot{}_{\overline{l}_3z_3\overline{m}_3}(l_3)\cdot z_2\cdot z_3\cdot l_3\cdot z_2\cdot z_3\cdot{}_{m_3}(l_3)\\
\simeq &~ z_2\cdot l_3\cdot{}_{\overline{l}_3}(z_3)\cdot m_3\cdot z_2\cdot z_3\cdot l_3\cdot z_2\cdot z_3\cdot{}_{m_3}(l_3)\\
\simeq &~ z_2\cdot z_3\cdot l_3\cdot m_3\cdot z_2\cdot z_3\cdot l_3\cdot z_2\cdot z_3\cdot{}_{m_3}(l_3)\\
\simeq &~ m_3\cdot z_2\cdot z_3\cdot{}_{\overline{m}_3}(l_3)\cdot z_2\cdot z_3\cdot l_3\cdot z_2\cdot z_3\cdot{}_{m_3}(l_3)
\end{align*}

We have already shown that $\tilde{X_4}$ is diffeomorphic to $\mathbb{P}^2\#16\overline{\mathbb{P}^2}$; hence $X_4$ is rational or ruled as well. Moreover, $e(X_4)=3$ and $\sigma(X_4)=1$ because the signature of an $8$-holed torus relation is equal to zero and the signature of a $9$-holed torus relation is equal to $1$.
\end{proof}
\label{lemma for $d=4$}
\end{lemma}

\subsection{Topological construction of Lefschetz pencils on $\mathbb{P}^2$ for general $d$}

In this section, we generalize the topological construction of symplectic Lefschetz pencil analogous to the holomorphic Lefschetz pencil of degree $d=4$ curves whose total space is diffeomorphic to $\mathbb{P}^2$ in Lemma ~\ref{lemma for $d=4$} to arbitrary $d\geq 4$.

\begin{figure}[h]
\includegraphics[width=0.75\textwidth]{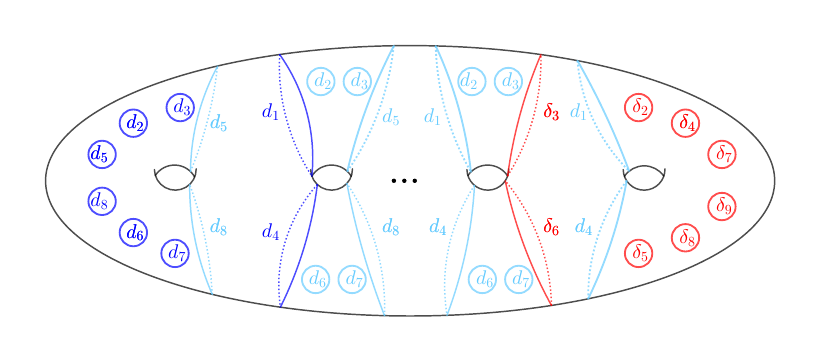}
\caption{Embeddings of $(d-2)$ copies of $8$-holed tori into $\Sigma_{d-2}^{4d-4}$}
\label{fig:d-2many8holedtori}
\end{figure}

\begin{proof}[Proof of Theorem \ref{thm:first}]
\textbf{Step 1.} 
First, one can construct a genus $(d-2)$ Lefschetz pencil with $4d-4$ base points and $8d-12$ many critical points
by combining $(d-2)$ copies of the $8$-holed torus relations of the following form. 

\begin{align}
\prod_{j=2}^{9}T_{\delta_j} &=(a_1^{(1)}a_4^{(1)})c_{11}^{(1)}(\tilde{b}_1^{(1)}\tilde{b}_2^{(1)}\tilde{b}_3^{(1)}b_5^{(1)}b_6^{(1)}a_7^{(1)}b_7^{(1)}b_8^{(1)}b_9^{(1)}).
\label{eq:red}
\end{align}
\begin{align}
\prod_{j=1}^{8}T_{d_j} =(c_{12}^{(2)}c_{11}^{(2)})(b_7^{(2)}b_8^{(2)}a_1^{(2)}b_1^{(2)}b_2^{(2)}\tilde{b}_3^{(2)}\tilde{b}_4^{(2)}\tilde{b}_6^{(2)})(c_3^{(2)}c_1^{(2)}). \label{eq:blue}\\
\simeq(c_{12}^{(2)}c_{10}^{(2)})c_{11}^{(2)}({}_{\overline{a}_1}(b_7^{(2)}b_8^{(2)})b_1^{(2)}b_2^{(2)}\tilde{b}_3^{(2)}\tilde{b}_4^{(2)}\tilde{b}_6^{(2)})(c_3^{(2)}c_1^{(2)}) \label{eq:lightblue}.
\end{align}
The first and the second relations were obtained in the proof of Lemma \ref{lemma for $d=4$}, and the last factorization is derived from the second one by Hurwitz moves.
We breed one copy of \eqref{eq:red}, the $8$-holed torus relation on $\mathbb{P}^2\#\overline{\mathbb{P}^2}$, 
and $(d-3)$ copies of the $8$-holed torus relation on $\mathbb{P}^1\times\mathbb{P}^1$, precisely one copy of \eqref{eq:blue} and $(d-4)$ copies of \eqref{eq:lightblue}, as shown in Figure~\ref{fig:d-2many8holedtori}, where the support of the first relation is indicated in red, that of the second relation in blue, and those of the third relations in light blue. As a result, we obtain the following relation in $\Gamma_{d-2}^{4d-4}$:
\begin{multline}
i_{d-2}(d_2d_3d_5d_8d_6d_7)i_{d-3}(d_2d_3d_6d_7)\cdots i_2(d_2d_3d_6d_7)i_1(\delta_2\delta_4\delta_7\delta_9\delta_8\delta_5)\\
=i_{d-2}(c_{12}^{(2)}c_{11}^{(2)})i_{d-3}(c_{11}^{(2)})\cdots i_2(c_{11}^{(2)})i_1(c_{11}^{(1)})i_{d-2}(\mathbb{B}_1)i_{d-3}(\mathbb{B}_2)\cdots i_2(\mathbb{B}_2)i_1(\mathbb{A}),
\label{eq:first}
\end{multline}
where \begin{align*}
\mathbb{A} &=\tilde{b}_1^{(1)}\tilde{b}_2^{(1)}\tilde{b}_3^{(1)}b_5^{(1)}b_6^{(1)}a_7^{(1)}b_7^{(1)}b_8^{(1)}b_9^{(1)},\\
\mathbb{B}_1 &=b_7^{(2)}b_8^{(2)}a_1^{(2)}b_1^{(2)}b_2^{(2)}\tilde{b}_3^{(2)}\tilde{b}_4^{(2)}\tilde{b}_6^{(2)},\\
\mathbb{B}_2 &={}_{\overline{a}_1^{(2)}}(b_7^{(2)}b_8^{(2)})b_1^{(2)}b_2^{(2)}\tilde{b}_3^{(2)}\tilde{b}_4^{(2)}\tilde{b}_6^{(2)}.
\end{align*}

In fact, one can show that the total space of this genus $(d-2)$ Lefschetz pencil $Y_d\dashrightarrow\mathbb{P}^1$ is diffeomorphic to $\mathbb{P}^2\#\overline{\mathbb{P}^2}$. By applying Sato's theorem to a relatively minimal genus $(d-2)$ Lefschetz fibration $Y_{d}\#(4d-4)\overline{\mathbb{P}^2}\rightarrow\mathbb{P}^1$ obtained by blowing up the base points, one can deduce that $Y_{d}\#(4d-4)\overline{\mathbb{P}^2}$ is either rational or ruled because $4d-4>2(d-2)-2$. Moreover, $H_1(Y_{d}\#(4d-4)\overline{\mathbb{P}^2}:\mathbb{Z})=0$ since the homology class of some vanishing cycles are given by $\lambda(i_{d-2}(b_2))=m_{d-2}, \lambda(i_{d-2}(a_1))=l_{d-2}, \lambda(i_{d-3}(b_2))=m_{d-2}+l_{d-3}, \lambda(i_{d-3}(\tilde{b}_3))=m_{d-3}+l_{d-3},\cdots, \lambda(i_2(b_2))=m_3+l_2, \lambda(i_2(\tilde{b}_3))=m_2+l_2, \lambda(i_1(a_7))=m_1,$ and $\lambda(i_1(b_7))=l_1$. On the other hand, $e(Y_d)=4$ because $e(Y_d\#(4d-4)\overline{\mathbb{P}^2})=4-4(d-2)+(8d-12)=4d$ and $\sigma(Y_d)=0$ because the signature of the $8$-holed torus relation is zero. Hence, $Y_d$ is diffeomorphic to either $\mathbb{P}\#\overline{\mathbb{P}^2}$ or $\mathbb{P}^1\times\mathbb{P}^1$. Now, it's enough to show that $Y_d$ contains a closed embedded surface of odd self-intersection. 
First, there exists a subsurface $S_0$, bounded by $d-1$ vanishing cycles $i_1(c_{11}), i_2(c_{11}),\allowbreak \cdots, i_{d-3}(c_{11}), i_{d-2}(c_{11}), i_{d-2}(c_{12})$ and containing $d-2$ base points, of the fiber of the Lefschetz pencil $Y_d\dashrightarrow\mathbb{P}^1$. By attaching core disks of the $2$-handles along $i_1(c_{11}), i_2(c_{11}),\cdots, i_{d-3}(c_{11}), i_{d-2}(c_{11}), i_{d-2}(c_{12})$ to $S_0$, we obtain a closed embedded surface $S$ in $Y_d$. If we denote its proper transform by $\tilde{S}$ in $Y_d\#(4d-4)\overline{\mathbb{P}^2}$, then the self-intersection $[S]^2-(d-2)=[\tilde{S}]^2=-(d-1)$.  Hence, $[S]^2=-1$. 

\begin{figure}[h]
\includegraphics[width=0.95\textwidth]{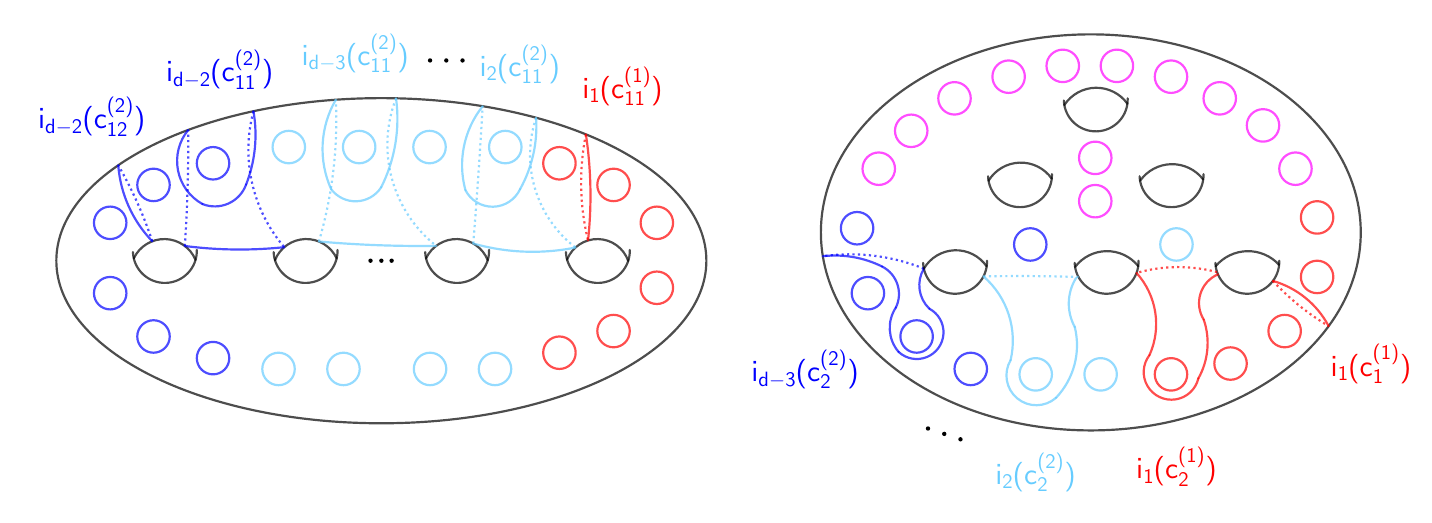}
\caption{Curves for breeding for arbitrary $d$}
\label{fig:breedingforgenerald}
\end{figure}

\textbf{Step 2.} Next, we show that $X_{d-1}$ breeding can be applied for any $d\geq 4$ to the genus $(d-2)$ Lefschetz pencil constructed in Step 1, generalizing a $9$-holed torus breeding used in the proof of Lemma~\ref{lemma for $d=4$}.
First, we observe that $(d-1)$ many vanishing cycles $\{i_1(c_{11}^{(1)}),i_2(c_{11}^{(2)}),\cdots i_{d-3}(c_{11}^{(2)}), i_{d-2}(c_{11}^{(2)}), i_{d-2}(c_{12}^{(2)})\}$ together with $(d-2)$ many boundary parallel curves bound a subsurface $\Sigma_0^{2d-3}$ of the fiber surface $\Sigma_{d-2}^{4d-4}$ of the Lefschetz pencil $Y_d\dashrightarrow\mathbb{P}^1$ as shown in Figure \ref{fig:breedingforgenerald}(a). Secondly, we claim that one can construct the monodromy relation of the Lefschetz pencil $X_{d-1}\dashrightarrow\mathbb{P}^1$ of the following form:
\begin{equation}
\prod_{j=1}^{(d-1)^2}T_{\delta_j'}=B_{3(d-2)^2}^{(d-1)}\cdots B_{d-1}^{(d-1)}(i_{d-3}(c_2^{(2)})\cdots i_2(c_2^{(2)})i_1(c_2^{(1)}c_1^{(1)})),
\label{d-1breeding}
\end{equation}
where $(d-2)$ vanishing cycles $\{i_1(c_1^{(1)}), i_1(c_2^{(1)}), i_2(c_2^{(2)}), \cdots, i_{d-3}(c_2^{(2)})\}$ together with $(d-1)$ many boundary components bound a subsurface $\Sigma_0^{2d-3}$ of the fiber surface $\Sigma_{g_{d-1}}^{(d-1)^2}$ as shown in Figure~\ref{fig:breedingforgenerald}(b).
For the proof of the claim, we first modify the $8$-holed torus relations as follows:
\begin{align*}
\prod_{j=2}^{9}T_{\delta_j} &=(a_1^{(1)}b_1^{(1)}b_2^{(1)}b_3^{(1)})(a_4^{(1)}b_4^{(1)}b_5^{(1)}b_6^{(1)})(a_7^{(1)}b_7^{(1)}b_8^{(1)}b_9^{(1)})\\
						&\simeq (a_1^{(1)}b_2^{(1)}b_3^{(1)})(a_4^{(1)}b_4^{(1)}b_5^{(1)}b_6^{(1)})(a_7^{(1)}b_7^{(1)}b_8^{(1)}b_9^{(1)})\cdot{}_{a_1^{(1)}}(b_1^{(1)})\\
						&\simeq (a_1^{(1)}b_2^{(1)}b_3^{(1)})(a_4^{(1)}b_4^{(1)}b_5^{(1)}b_6^{(1)})(a_7^{(1)}b_7^{(1)}b_8^{(1)})\cdot{}_{a_1}(b_1^{(1)})\cdot c_2^{(1)}\\
						&\simeq (a_1^{(1)}b_2^{(1)}b_3^{(1)})(a_4^{(1)}b_4^{(1)}b_5^{(1)}b_6^{(1)})(\tilde{b}_7^{(1)}\tilde{b}_8^{(1)}\tilde{b}_1^{(1)})c_1^{(1)}c_2^{(1)}\\
						&\simeq (a_1^{(1)}b_2^{(1)}b_3^{(1)})(c_{11}'\cdot a_4^{(1)}b_5^{(1)}b_6^{(1)})(\tilde{b}_7^{(1)}\tilde{b}_8^{(1)}\tilde{b}_1^{(1)})c_1^{(1)}c_2^{(1)}\\
						&\simeq (a_1^{(1)}b_2^{(1)}c_{11}'\cdot{}_{\overline{c}_{11}'}(b_3^{(1)})a_4^{(1)}b_5^{(1)}b_6^{(1)})(\tilde{b}_7^{(1)}\tilde{b}_8^{(1)}\tilde{b}_1^{(1)})c_1^{(1)}c_2^{(1)}\\
						&\simeq (a_1^{(1)}c_{11}^{(1)})b_2^{(1)}\cdot{}_{\overline{c}_{11}'}(b_3^{(1)})\cdot a_4^{(1)}(b_5^{(1)}b_6^{(1)}\tilde{b}_7^{(1)}\tilde{b}_8^{(1)}\tilde{b}_1^{(1)})c_1^{(1)}c_2^{(1)}\\
						&\simeq (a_1^{(1)}a_4^{(1)})c_{11}^{(1)}(\tilde{b}_2^{(1)}\tilde{b}_3^{(1)}b_5^{(1)}b_6^{(1)}\tilde{b}_7^{(1)}\tilde{b}_8^{(1)}\tilde{b}_1^{(1)})c_2^{(1)}c_1^{(1)}
	\end{align*}
where $\tilde{b}_7^{(1)}={}_{a_7}(b_7^{(1)}), \tilde{b}_8^{(1)}={}_{a_7}(b_8^{(1)}), \tilde{b}_1^{(1)}={}_{a_7a_1}(b_1^{(1)}), \tilde{b}_2^{(1)}={}_{\overline{a}_4}(b_2^{(1)}), \tilde{b}_3^{(1)}={}_{\overline{a}_4\overline{c}_{11}'}(b_3^{(1)})$.

\begin{align*}
\prod_{j=1}^{8}T_{d_j} &=(c_{12}^{(2)}c_{11}^{(2)}))(b_7^{(2)}b_8^{(2)}a_1^{(2)}b_1^{(2)}b_2^{(2)}\tilde{b}_3^{(2)}\tilde{b}_4^{(2)}\tilde{b}_6^{(2)})(c_3^{(2)}c_1^{(2)})\\
					&\simeq (c_{12}^{(2)}c_{11}^{(2)}))(b_7^{(2)}b_8^{(2)}{}_{\overline{c}_8'}(c_2^{(2)})a_1^{(2)}b_2^{(2)}\tilde{b}_3^{(2)}\tilde{b}_4^{(2)}\tilde{b}_6^{(2)})(c_3^{(2)}c_1^{(2)})\\
					&\simeq (c_{12}^{(2)}c_{11}^{(2)}))(b_7^{(2)}c_2^{(2)}b_8^{(2)}a_1^{(2)}b_2^{(2)}\tilde{b}_3^{(2)}\tilde{b}_4^{(2)}\tilde{b}_6^{(2)})(c_3^{(2)}c_1^{(2)})\\
					&\simeq(c_{12}^{(2)}c_{11}^{(2)})({}_{\overline{c}_2}(b_7^{(2)})b_8^{(2)}a_1^{(2)}b_2^{(2)}\tilde{b}_3^{(2)}\tilde{b}_4^{(2)}\tilde{b}_6^{(2)})(c_2^{(2)}c_3^{(2)}c_1^{(2)})\\
					&\simeq (c_{12}^{(2)}a_1^{(2)}c_{11}^{(2)})({}_{\overline{a}_1\overline{c}_2}(b_7^{(2)})\cdot{}_{\overline{a}_1}(b_8^{(2)})b_2^{(2)}\tilde{b}_3^{(2)}\tilde{b}_4^{(2)}\tilde{b}_6^{(2)})(c_2^{(2)}c_3^{(2)}c_1^{(2)}).			
\end{align*}
By applying the breeding method as in Step 1, but with the above modified relations, we obtain the following relation in $\Gamma_{d-3}^{4d-8}$:
\begin{align*}
&i_{d-3}(d_2d_3d_5d_6d_7d_8)i_{d-4}(d_2d_3d_6d_7)\cdots i_2(d_2d_3d_6d_7)i_1(\delta_2\delta_4\delta_5\delta_7\delta_8\delta_9)\\
&=i_{d-3}(c_{12}^{(2)}c_{11}^{(2)}\mathbb{B}_1'c_2^{(2)})i_{d-4}(c_{11}^{(2)}\mathbb{B}_2'c_2^{(2)})\cdots i_2(c_{11}^{(2)}\mathbb{B}_2'c_2^{(2)})i_1(c_{11}^{(1)}\mathbb{A}'c_2^{(1)}c_1^{(1)}),
\end{align*}
where \begin{align*} 
 \mathbb{A}'&=\tilde{b}_3^{(1)}\tilde{b}_2^{(1)}b_5^{(1)}b_6^{(1)}\tilde{b}_7^{(1)}\tilde{b}_8^{(1)}\tilde{b}_1^{(1)},\\
 \mathbb{B}_1'&={}_{\overline{c}_2}(b_7^{(2)})b_8^{(2)}a_1^{(2)}b_2^{(2)}\tilde{b}_3^{(2)}\tilde{b}_4^{(2)}\tilde{b}_6^{(2)},\\
 \mathbb{B}_2'&={}_{\overline{a}_1\overline{c}_2}(b_7^{(2)}){}_{\overline{a}_1}(b_8^{(2)})b_2^{(2)}\tilde{b}_3^{(2)}\tilde{b}_4^{(2)}\tilde{b}_6^{(2)}.
\end{align*}
By the commutativity relations, it can be rewritten as follows.
\begin{multline}
i_{d-3}(d_2d_3d_5d_6d_7d_8)i_{d-4}(d_2d_3d_6d_7)\cdots i_2(d_2d_3d_6d_7)i_1(\delta_2\delta_4\delta_5\delta_7\delta_8\delta_9)\\
=(i_{d-3}(c_{12}^{(2)}c_{11}^{(2)})i_{d-4}(c_{11}^{(2)})\cdots i_2(c_{11}^{(2)})i_1(c_{11}^{(1)}))i_{d-3}(\mathbb{B}_1')i_{d-4}(\mathbb{B}_2')\cdots i_2(\mathbb{B}_2')i_1(\mathbb{A}')\\
(i_{d-3}(c_2^{(2)})\cdots i_2(c_2^{(2)})i_1(c_2^{(1)}c_1^{(1)})).
\label{eq:genus$d-3$}
\end{multline}
By the induction hypothesis, we assume that $X_{d-2}$ breeding is possible by using 
\begin{equation}
\prod_{j=1}^{(d-2)^2}T_{\delta_j'}=B_{3(d-3)^2}^{(d-2)}\cdots B_{d-2}^{(d-2)}(i_{d-4}(c_2^{(2)})\cdots i_2(c_2^{(2)})i_1(c_2^{(1)}c_1^{(1)})),
\label{eq:$X_{d-2}$breed}
\end{equation}
where $(d-3)$ vanishing cycles $\{i_1(c_1^{(1)}), i_1(c_2^{(1)}), i_2(c_2^{(2)}), \cdots, i_{d-4}(c_2^{(2)})\}$ together with $(d-2)$ boundary components, say $\delta_{(d-2)(d-3)+1}',\delta_{(d-2)(d-3)+2}',\cdots,\delta_{(d-2)^2}'$, bound a subsurface $\Sigma_0^{2d-5}$ of the fiber surface $\Sigma_{g_{d-2}}^{(d-2)^2}$.

Hence, breeding the relation \eqref{eq:$X_{d-2}$breed} to the genus $(d-3)$ Lefschetz pencil \eqref{eq:genus$d-3$}, by identifying $(d-3)$ boundary curves $i_1(\delta_2), i_2(d_2), \cdots, i_{d-4}(d_2),i_{d-3}(d_2)$ of \eqref{eq:genus$d-3$} with $(d-3)$ vanishing cycles $i_1(c_1^{(1)}), i_1(c_2^{(1)}), \allowbreak i_2(c_2^{(2)}), \cdots, i_{d-4}(c_2^{(2)})$ of \eqref{eq:$X_{d-2}$breed} and identifying $(d-2)$ vanishing cycles $i_{d-3}(c_{12}^{(2)}), i_{d-3}(c_{11}^{(2)}), i_{d-4}(c_{11}^{(2)}), \cdots,\allowbreak i_2(c_{11}^{(2)}), i_1(c_{11}^{(1)})$ of \eqref{eq:genus$d-3$} with $(d-2)$ boundary curves $\delta_{(d-2)(d-3)+1}',\allowbreak\delta_{(d-2)(d-3)+2}',\cdots,\delta_{(d-2)^2}'$ of \eqref{eq:$X_{d-2}$breed}, induces the desired relation for the $X_{d-1}$ breeding. It follows that
\begin{multline}
j(\delta_1'\cdots \delta_{(d-2)(d-3)}')i_{d-3}(d_3d_5d_6d_7d_8)i_{d-4}(d_3d_6d_7)\cdots i_2(d_3d_6d_7)i_1(\delta_4\delta_5\delta_7\delta_8\delta_9)\\
=B_{3(d-2)^2}^{(d-1)}\cdots B_{d-1}^{(d-1)}(i_{d-3}(c_2^{(2)})\cdots i_2(c_2^{(2)})i_1(c_2^{(1)}c_1^{(1)})),
\label{eq:fourth}
\end{multline}
where $B_{3(d-2)^2}^{(d-1)}\cdots B_{d-1}^{(d-1)}=j(B_{3(d-3)^2}^{(d-2)}\cdots B_{d-2}^{(d-2)})i_{d-3}(\mathbb{B}_1')i_{d-4}(\mathbb{B}_2')\cdots i_2(\mathbb{B}_2')i_1(\mathbb{A}')$.\\
Finally, by renaming the boundary curves on the left-hand side as $\delta_1',\cdots, \delta_{(d-1)^2}'$  we obtain the relation \eqref{d-1breeding}.

\textbf{Step 3.} Now,  by breeding \eqref{d-1breeding} to \eqref{eq:first} as in Figure \ref{fig:breedingforgenerald} we obtain the following monodromy relation of the Lefschetz pencil $X_d\dashrightarrow\mathbb{P}^1$ for arbitrary $d\geq 4$:
\begin{multline}
j(\delta_1'\cdots\delta_{(d-1)(d-2)}')i_{d-2}(d_3d_5d_8d_6d_7)i_{d-3}(d_3d_6d_7)\cdots i_2(d_3d_6d_7)i_1(\delta_4\delta_7\delta_9\delta_8\delta_5)\\
=j(B_{3(d-2)^2}^{(d-1)}\cdots B_{d-1}^{(d-1)})(i_{d-2}^{(1)}(\mathbb{B}_1)i_{d-3}^{(1)}(\mathbb{B}_2)\cdots i_2^{(1)}(\mathbb{B}_2)i_1^{(1)}(\mathbb{A})).
\end{multline}
Moreover, one can easily check that the genus, the number of base points, and the number of critical points of the Lefschetz pencil constructed topologically by the breeding operation coincide with those for the algebraic pencil of curves of degree $d$ in $\mathbb{P}^2$ as follows:
\begin{align*}
g_d=(d-2)+\frac{(d-2)(d-3)}{2}=\frac{(d-1)(d-2)}{2},\\
b_d=(8+4(d-3))+(d-1)^2-((d-2)+(d-1))=d^2,\\
N_d=\{(8d-12)-(d-1)\}+\{3(d-2)^2-(d-2)\}=3(d-1)^2.
\end{align*}

\textbf{Step 4.} Finally, we need to show that the total space $X_d$ is diffeomorphic to $\mathbb{P}^2$.
First, as in the case of $d=4$, we can apply the result of Sato to conclude that the total space $X_d\#d^2\overline{\mathbb{P}^2}$ of the Lefschetz fibration $\tilde{f_d}:X_d\#d^2\overline{\mathbb{P}^2}\rightarrow\mathbb{P}^1$ is either rational or ruled. This is because the Lefschetz fibration $\tilde{f}_d$ obtained by blowing up the base points from the Lefschetz pencil $f_d:X_d\dashrightarrow \mathbb{P}^1$, admits $d^2$ many $(-1)$ sphere sections and $d^2$ is greater than $2g_d-2$, which is equal to $d^2-3d$. Hence, $X_d$ is also rational or ruled. Moreover, $\sigma(X_d)=1$ since the monodromy factorization of the Lefschetz pencil $f_d:X_d\dashrightarrow \mathbb{P}^1$ was constructed by using the $8$-holed torus relations, which has signature zero, and one copy of the $9$-holed torus relation, which has signature $1$. Therefore, $X_d$ is diffeomorphic to $\mathbb{P}^2$.
\end{proof}

\begin{proof}[Proof of Corollary \ref{cor}]
By a further inductive argument, the following subword of $f_{d-1}:X_{d-1}\dashrightarrow\mathbb{P}^1$ in the monodromy factorization that we constructed in Theorem \ref{thm:first} can be obtained as a result of breeding $8$-holed torus relations and a $9$-holed torus relation as follows: 
\begin{align*}
j(B_{3(d-2)^2}\cdots B_{d-1}) &=i_1^{(d-2)}(a_1^{(1)}b_2^{(1)}b_3^{(1)}a_4^{(1)}b_4^{(1)}b_5^{(1)}b_6^{(1)}\tilde{b}_7^{(1)}\tilde{b}_8^{(1)}\tilde{b}_1^{(1)})\\
 &i_2^{(d-3)}(\mathbb{B}_1')i_1^{(d-3)}(\mathbb{A}')i_3^{(d-4)}(\mathbb{B}_1')i_2^{(d-4)}(\mathbb{B}_2')i_1^{(d-4)}(\mathbb{A}')\\
 &\cdots i_{d-3}^{(2)}(\mathbb{B}_1')i_{d-4}^{(2)}(\mathbb{B}_2')\cdots i_2^{(2)}(\mathbb{B}_2')i_1^{(2)}(\mathbb{A}').
\end{align*}
\end{proof}

\begin{figure}
\centering
\begin{tikzpicture}[scale=0.45]
\draw[blue] (0,2) -- (2,2) -- (2,0) -- (0,0) --cycle;
\draw[blue] (1,0) -- (1,2);
\draw[blue] (0,1) -- (2,1);
\draw[blue] (0,0) -- (2,2);
\draw[blue] (0,1) -- (1,2);
\draw[blue] (1,0) -- (2,1);

\node at (2.5,1) {+};

\draw[red] (3,0) -- (4,0) -- (5,1) -- (6,2) -- (3,2) --cycle;
\draw[red] (3,1) --(5,1);
\draw[red] (4,0) --(4,2);
\draw[red] (5,1) --(5,2);
\draw[red] (3,0) -- (5,2);
\draw[red] (3,1) --(4,2);

\node at (2.5,-0.5) {+};

\draw[magenta] (1,-1) -- (4,-1) -- (1,-4) --cycle;
\draw[magenta] (1,-2) --(3,-2);
\draw[magenta] (1,-3) --(2,-3);
\draw[magenta] (2,-3) --(2,-1);
\draw[magenta] (3,-2) --(3,-1);
\draw[magenta] (3,-1) --(1,-3);
\draw[magenta] (2,-1) --(1,-2);

\node at (6,-0.5) {$\Rightarrow$};

\draw (7,1) -- (11,1) -- (7,-3) --cycle;
\draw (7,0) --(10,0);
\draw (7,-1) --(9,-1);
\draw (7,-2) --(8,-2);
\draw (8,-2) --(8,1);
\draw (9,-1) --(9,1);
\draw (10,0) --(10,1);
\draw (10,1) --(7,-2);
\draw (7,-1) --(9,1);
\draw (7,0) --(8,1);
\end{tikzpicture}
\caption{Building blocks for $d=4$}
\label{fig:intuition}
\end{figure}
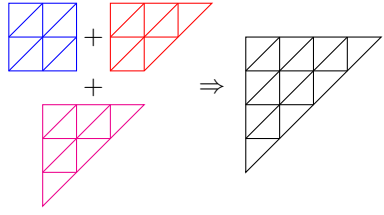

\begin{remark}
In fact, the topological construction described in this section is somehow related to the holomorphic Lefschetz pencil of degree $d$ curves in $\mathbb{P}^2$ through the degeneration diagram which will be explained in the next section. For example,   three pieces on the left-hand-side in Figure \ref{fig:intuition} are three building blocks for the genus $3$ holomorphic Lefschetz pencil of degree $4$ curves in $\mathbb{P}^2$ on the right-hand-side. Precisely, the red (or blue) diagram represents the $8$-holed torus relation of non-spin (or spin) type, the pink one represents the $9$-holed torus relation, and the process of assembling those buildling blocks corresponds to the breeding operations explained in the proof.   
\end{remark}

\section{Topological construction of the holomorphic Lefschetz pencil of degree $4$ curves in $\mathbb{P}^2$}
In this section, we will prove Theorem~\ref{thm:second}.
\subsection{Monodromy factorization of the holomorphic Lefschetz pencil of degree $4$ curves in $\mathbb{P}^2$}
We will find the explicit vanishing cycles of the holomorphic Lefschetz pencil $h_4:\mathbb{P}^2\dashrightarrow\mathbb{P}^1$ in Theorem \ref{thm:second}, in a manner similar to how Hamada and Hayano computed those for the Lefschetz pencil of degree $3$ curves on $\mathbb{P}^2$ in \cite{HH:21}. First, we need to compute the braid monodromy of the branch curve $S$ of the generic projection $\pi|_{V_4}:V_4\rightarrow\mathbb{P}^2$, restricted to the Veronese surface of order $4$.
We begin with the following Theorem on the projective degeneration of the Veronese surface $V_d$ of arbitrary order $d$. 
\begin{theorem}\cite{MT:BGT3}
Let $V_d\subset \mathbb{P}^N$ be the Veronese surface of order $d$. Then there exists a sequence of projective degenerations from $V_d$ to a union of $d^2$ planes.   
\end{theorem} 

\begin{figure}
\includegraphics[width=0.9\textwidth]{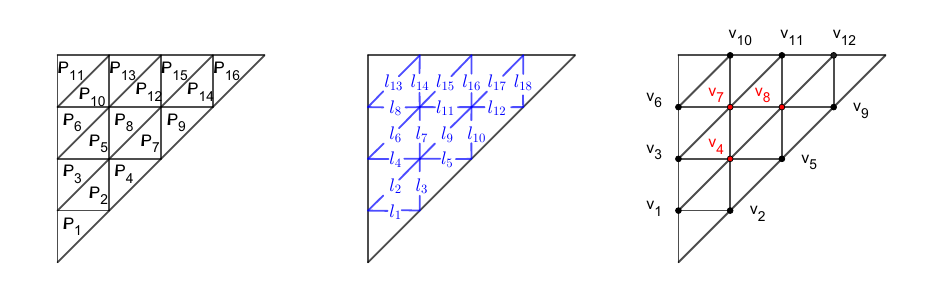}
\caption{Planes, lines, and multiple points in $(V_4)_0$}
\label{fig:degdiagram}
\end{figure}

In particular, the result of the sequence of projective degenerations of $V_4$ is a union of $16$ projective planes, denoted by $(V_4)_0$, and it can be described pictorially as in Figure \ref{fig:degdiagram}, where each triangle represents the surface equivalent to the image of the Veronese embedding $v_1$ of the plane and each edge separating two triangles represents the intersection line between the corresponding planes. 
We denote the planes, the intersection lines, and the vertices (corresponding to the multiple points of the line arrangement) in $(V_4)_0$ by $\{P_i\}_{i=1}^{16}$, $\{l_j\}_{j=1}^{18}$, and $\{v_k\}_{k=1}^{12}$ as shown in Figure \ref{fig:degdiagram}.
The vertices are classified into two types. Each of three points $v_4, v_7,$ or $v_8$, where six lines intersect, is referred to as a $6$-point, while the remaining nine points are called $2$-points.\\

\begin{figure}
\includegraphics[width=0.7\textwidth]{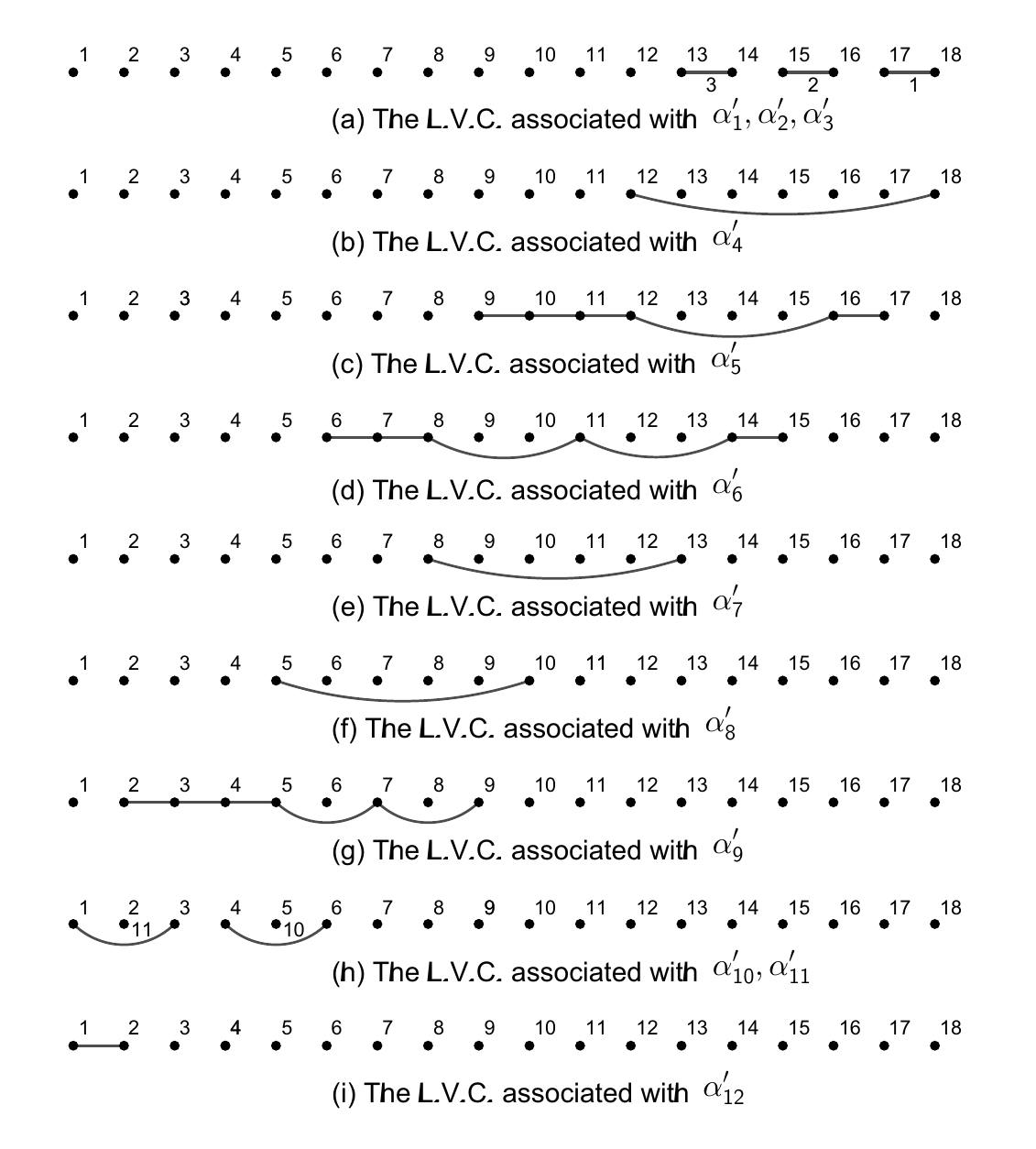}
\caption{Braid monodromies around the multiple points of $S_0$}
\label{fig:braidmonodromyofdegbrcurve}
\end{figure}

First, the branch curve $S_0$ of the generic projection $\pi|_{(V_4)_0}:(V_4)_0\rightarrow\mathbb{P}^2$, restricted to the degenerated surface $(V_4)_0$, is given by the line arrangement $\cup_{j=1}^{18} \pi(l_j)$ in $\mathbb{P}^2$ that can be considered as a subarrangement of the lines in $\mathbb{P}^2$ dual to $12$ generic points. Moreover, the chosen ordering of lines is a reverse lexicographic ordering with respect to the ordering of vertices, given by the $x$-coordinate, so that it concides with the ordering by the $y$-coordinate of the intersection of the reference fiber over a base point $a_0'\gg x(v_{12})$ with the corresponding line. Hence, by Theorem $IX.2.1$ 
in \cite{MT:BGT1} there exists an ordered set $(\alpha_1',\cdots,\alpha_{12}')$ of simple paths $\alpha_m'$ from a fixed base point $a_0'$ in $\mathbb{P}^1\setminus\{a_1',\cdots,a_{12}'\}$  to each $a_n'$, where $a_n'=\pi'\circ\pi(a_n)$ and $a_n=v_{13-n}$, ordered counterclockwise around $a_0'$, so that the braid monodromy $\rho$ of $S_0$ along each loop $l(\alpha_n')$ based at $a_0'$ going around one $a_n'$ is given by the full twist $\Delta_n^2$ (that is, the square of the generalized half-twist) about the consecutive sequence $\gamma_n$ of simple paths (called the Lefschetz vanishing cycle) as shown in Figure~ \ref{fig:braidmonodromyofdegbrcurve}, where the points labeled with $j$ are the intersection between the corresponding line $\pi(l_j)$ and the fiber $\pi'^{-1}(a_0')$. Moreover, its braid monodromy factorization is given by $\prod_{n=1}^{12} C_n\cdot\Delta_n^2=\Delta_{18}^2$ in the braid group $B_{18}$ where $C_n$ is the product of the braid monodromies for the additional nodes not apprearing in $X_0$. 

Since the degree of the curve is doubled after the regeneration, the branch curve $S$ of the projection $V_4\rightarrow\mathbb{P}^2$ is a curve of degree $36$ in $\mathbb{P}^2$ consisting of the local configurations around each of $12$ vertices and the additional nodes both obtained after the regeneration. Moreover, the braid monodromy factorization of $S$ is given by 
$\prod_{n=1}^{12} C_n'\cdot(H_n)_{e_n}=\Delta_{36}^2$ in $B_{36}$, where $H_n$ is the local braid monodromy of $S$ after regeneration around each vertex $a_n$, $(H_n)_{e_n}$ is the conjugation of $H_n$ by some $e_n\in \langle z_{jj'}~|~ a_n\in L_j\rangle$, and $C_n'$ is obtained from $C_n$ by the second regeneration rule. This can be verified as in the Lemma $12$ in \cite{MT:BGT4}. Here, the only thing we need to check is that there are no extra braid factors by comparing the degrees of both sides. Precisely, deg$(\Delta_{36}^2)=36\cdot35=1260$ and deg$(\prod{C_n'(H_n)_{e_n}})=2\cdot(99\cdot4)+3\cdot 126+9\cdot 10=1260$. Additionally, we can take $e_n=1$ for all $n$ in the case of all Veronese surfaces $V_d$, as proved in the Theorem $8.2$ in \cite{Moi:85}. We still have some ambiguity $m\in\mathbb{Z}$ involved in a factor $\alpha^{(1)}(m)$, obtained after the first regeneration rule, of the local monodromy $H_k$ around each $6$-point $v_k$ as explained in \cite{MT:BGT4}. However, we can show, in the following lemma, that the global braid monodromy factorization is actually equivalent to the factorization in which all integers $m_1, m_2,$ and $m_3$, corresponding to the indices $k=4,7,$ and $8$ of three $6$-points, are equal to zero. 

\begin{lemma}
The global braid factorization of the branch curve $S$ of the projection $\pi|_{V_4}:V_4\rightarrow\mathbb{P}^2$ is given by
\[\Delta_{36}^2=\prod_{k=12}^1 C_k' H_k'~~~~\text{in}~~~ B_{36}\]
where $C_k'$ is the braid monodromy for the additional nodes not appearing in $X_0$ after regeneration, 
$H_k$ is the local braid monodromy around each vertex $v_k$ after regeneration, and
\[H_k'=
\begin{cases}
H_k & \text{for} ~k\neq 4, \neq 7, \neq 8\\
H_4(m_1=0) & \text{for} ~k=4\\
H_7(m_2=0) & \text{for} ~k=7\\
H_8(m_3=0) & \text{for} ~k=8 .
\end{cases}
\]
\begin{proof}
As mentioned above, the braid monodromy factorization is given by $\Delta_{36}^2=\prod_{k=12}^{1}C_k'H_k$ by \cite{MT:BGT4} and \cite{Moi:85}.

Here, we need to recall that each of $H_4, H_7,$ and $H_8$ is a product of braid factors including $\alpha^{(1)}(m_1), \alpha^{(1)}(m_2), $ and $\alpha^{(1)}(m_3)$, respectively, for some integers $m_1,m_2,$ and $m_3$ that are not determined. Also, $\alpha^{(1)}(m_1)=(\alpha^{(1)}(0))_{z_{5,5'}^{m_1}}$, $\alpha^{(1)}(m_2)=(\alpha^{(1)}(0))_{z_{11,11'}^{m_2}}$, and $\alpha^{(1)}(m_3)=(\alpha^{(1)}(0))_{z_{12,12'}^{m_3}}$. Taking a global conjugation of the above factorization by $z_{5,5'}^{-m_1}\cdot z_{11,11'}^{-m_2}\cdot z_{12,12'}^{-m_3}$, by the complete invariance for $C_k'$ and for $H_k$ with $k\neq 4, \neq 7, \neq 8$, it is Hurwitz equivalent to the following factorization \[\Delta_{36}^2=\prod_{k=12}^{1}C_k'H_k''\] where 
\[H_k''=\begin{cases}
H_k & \text{for} ~k \neq 4,\neq 7, \neq 8\\
(H_4)_{z_{5,5'}^{-m_1}} & \text{for} ~k=4\\
(H_7)_{z_{11,11'}^{-m_2}} & \text{for} ~k=7\\
(H_8)_{z_{11,11'}^{-m_2}z_{12,12'}^{-m_3}} & \text{for} ~k=8.
\end{cases}\]

Moreover, the invariance rule $2$ and $3$ imply that $(H_4)_{z_{5,5'}^{-m_1}}$ is hurwitz equivalent to $H_4(0)$ and $(H_7)_{z_{7,7'}^{-m_2}}$ is hurwitz equivalent to $H_7(0)$  . We can also show that $(H_8(m_3))_{z_{11,11'}^{-m_2}\cdot z_{12,12'}^{-m_3}}$ is hurwitz equivalent to $H_8(m_2)$. First, it is trivial for any factor neither involving indices $11,11'$ nor $12,12'$. Second, $(\alpha^{(1)}(m_3))_{z_{11,11'}^{-m_2}\cdot z_{12,12'}^{-m_3}}=(\alpha^{(1)}(0))_{z_{11,11'}^{-m_2}}$ is hurwitz equivalent to $(\alpha^{(1)}(0))_{z_{12,12'}^{m_2}}$ by the invariance rule $1$. Finally, for any factor of $H_8(m_3)$ outside of $\alpha^{(1)}(m_3)$ involving indices either $11,11'$ or $12,12'$, we can apply the second and the third invariance rules to conclude that the factor of the form $z^{(2)}_{11,11',j}$ is invariant uner $z_{11,11'}^{-m_2}$, the factor $z^{(3)}_{i,11,11'}$ is invariant under $z_{11,11'}^{-m_2}$, the factor $z^{(3)}_{12,12',j,j'}$ is invariant under $z_{12,12'}^{-m_3}$, and the factor $z^{(2)}_{i,12,12'}$ is invariant under $z_{12,12'}^{-m_3}$. Hence,
 \[\Delta_{36}^2=\prod_{j=12}^{1}C_j'H_j''\] where \[H_j'''=\begin{cases}
H_j & \text{for} ~j \neq 4,\neq 7, \neq 8\\
H_4(0) & \text{for} ~j=4\\
H_7(0) & \text{for} ~j=7\\
H_8(m_2) & \text{for} ~j=8.
\end{cases}\]
Now, using a similar argument as in the $d=3$ case in \cite[Proposition 17]{MT:BGT4}, we obtain the conclusion of this lemma.
\end{proof}
\label{lem:global}
\end{lemma}

\begin{remark}
One can easily observe that the above proof of Lemma~\ref{lem:global} for $d=4$ extends to the cases $d\geq 5$ as well. 
\end{remark}

\begin{center}
\begin{figure}
\includegraphics[width=0.9\textwidth]{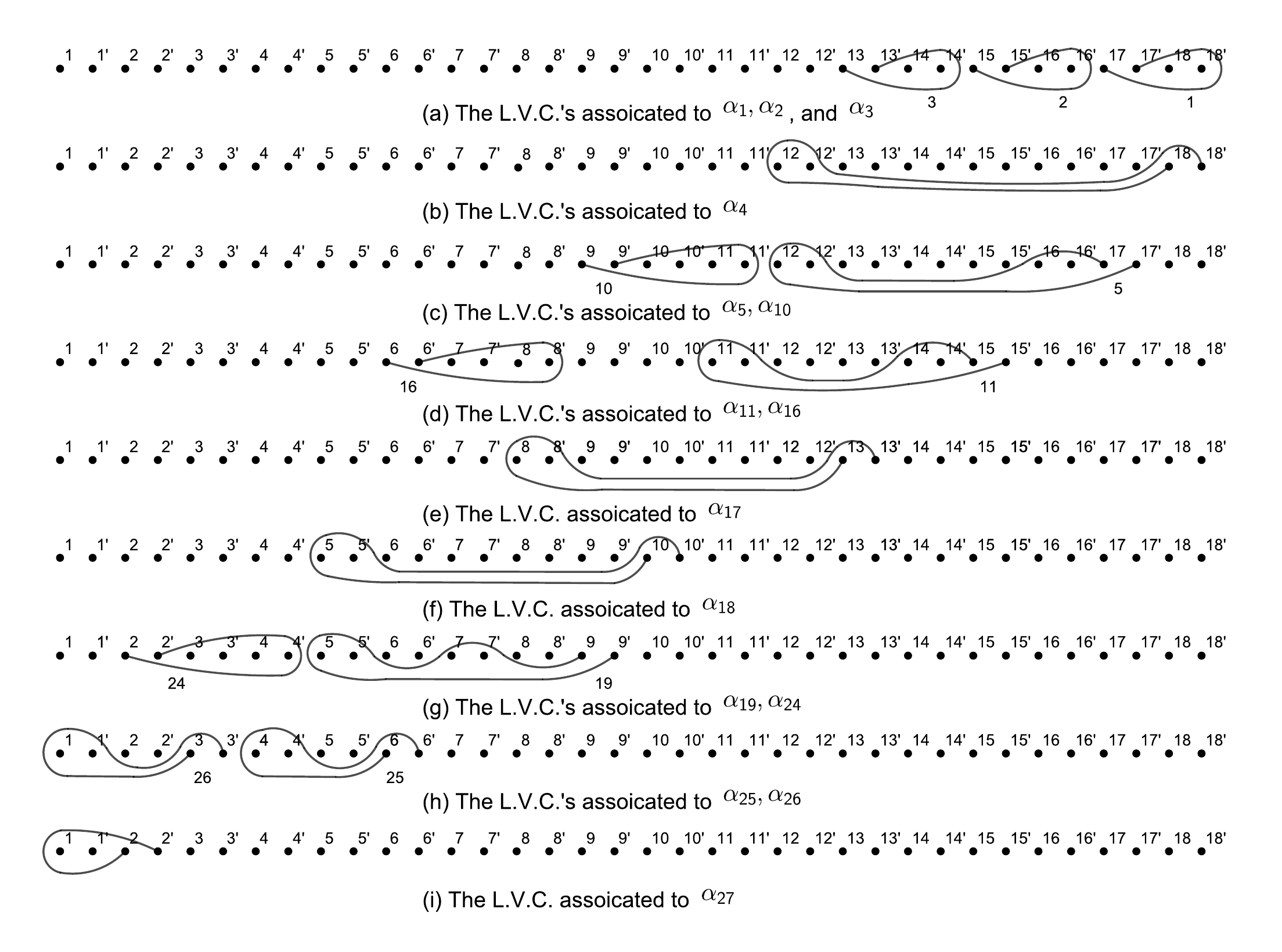}
\caption{Braid monodromy of the branch points of $\pi'|_{S}:S\rightarrow\mathbb{P}^1$}
\label{fig:braidmonodromyofregenerated}
\end{figure}
\end{center}  

\begin{figure}[h]
\includegraphics[width=0.9\textwidth]{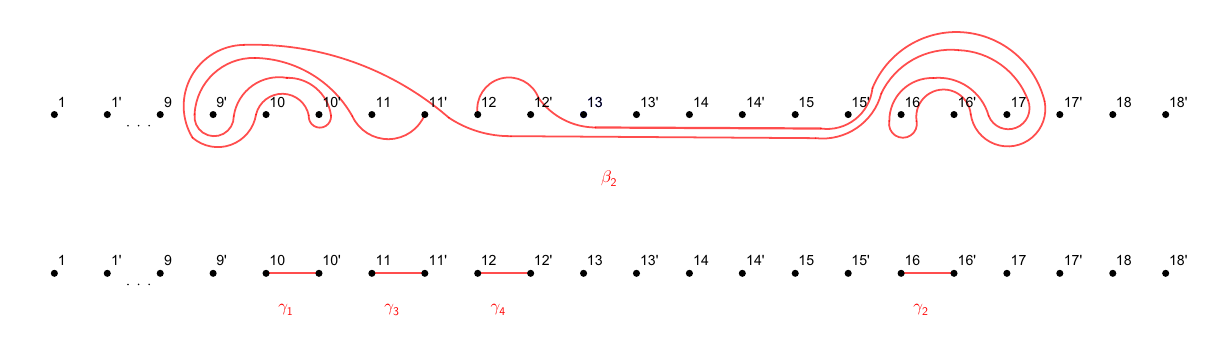}
\caption{L.V.C.'s associated to $\alpha_6, \alpha_7, \alpha_8, \alpha_9$}
\label{braidasswith6,7,8,9}
\end{figure}

\begin{figure}[h]
\includegraphics[width=0.9\textwidth]{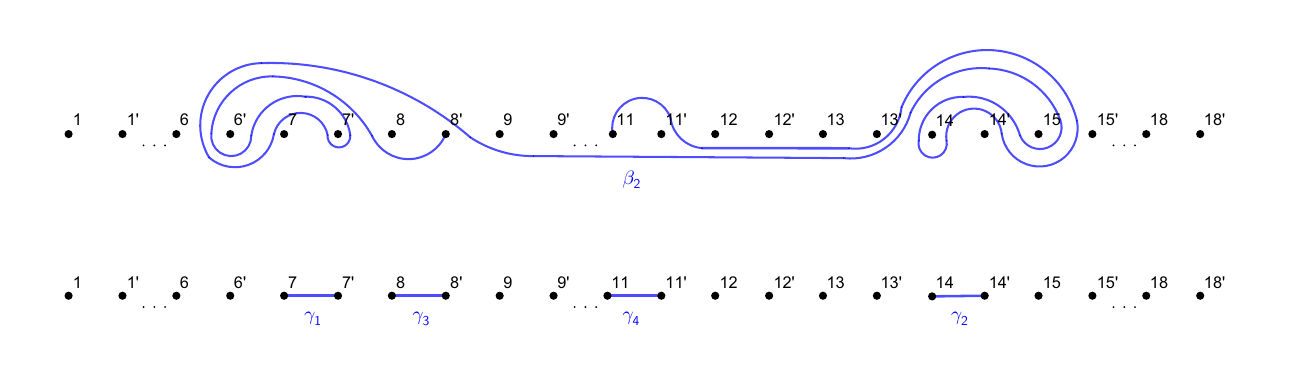}
\caption{L.V.C.'s associated to $\alpha_{12}, \alpha_{13}, \alpha_{14}, \alpha_{15}$}
\label{braidasswith12,13,14,15}
\end{figure}

\begin{figure}[h]
\includegraphics[width=0.9\textwidth]{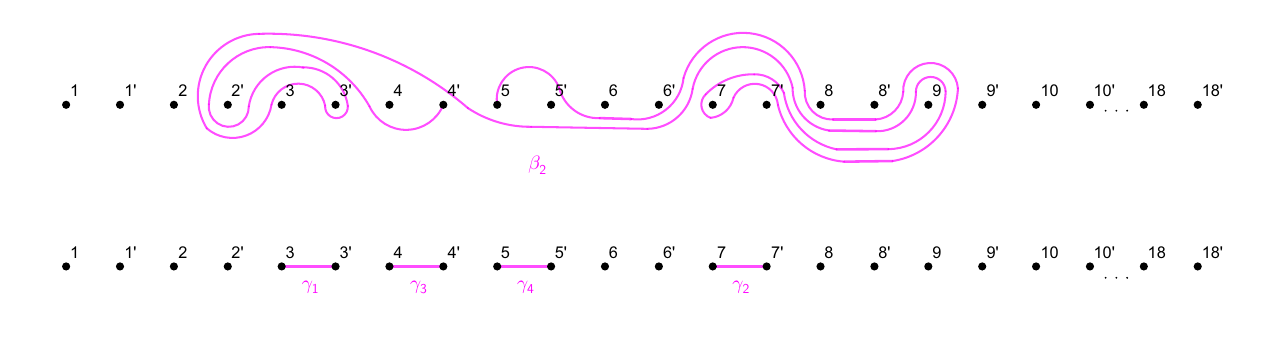}
\caption{L.V.C.'s associated to $\alpha_{20}, \alpha_{21}, \alpha_{22}, \alpha_{23}$}
\label{braidasswith20,21,22,23}
\end{figure}

The right-hand side of the global factorization in the previous Lemma~\ref{lem:global} is given by the product of the braid factors associated with the branch points, the nodes, and the cusps of the curve $S$. Among those factors, we can find the braid factors associated with $27$ branch points, which lift to the Lefschetz critical points, 
by applying Lemma \ref{lem:2pt}, Lemma \ref{lem:6pt}, and Lemma~\ref{lem:global}.

Precisely, for a suitable Hurwitz path system $(\alpha_1,\alpha_2, \cdots, \alpha_{27})$ of $h_4:V_4\dashrightarrow\mathbb{P}^1$ the Lefschetz vanishing cycles of the branch points associated to $\alpha_k$ for $k=1,2, \allowbreak 3,4,5,10,11,16,17,18,19,24,25,26,$ and $27$ are shown in Figure~ \ref{fig:braidmonodromyofregenerated}. The Lefschetz vanishing cycles corresponding to $\alpha_6,\alpha_7,\alpha_8,$ and $\alpha_9$ are given by $\beta_2,$$\tau_{\gamma_3}^{-1}\tau_{\gamma_4}^{-1}(\beta_2),$ $\tau_{\gamma_1}^{-1}\tau_{\gamma_2}^{-1}(\beta),$ and $\tau_{\gamma_1}^{-1}\tau_{\gamma_2}^{-1}\tau_{\gamma_3}^{-1}\tau_{\gamma_4}^{-1}(\beta_2)$, respectively, where the paths
$\beta_2, \gamma_1, \allowbreak \gamma_2,\gamma_3,$ and $\gamma_4$ are shown as in Figure~\ref{braidasswith6,7,8,9}.
Similarly, the Lefschetz vanishing cycles corresponding to $\alpha_{12},\alpha_{13},\alpha_{14}$, and $\alpha_{15}$ (or $\alpha_{20}, \alpha_{21}, \alpha_{22}$, and $\alpha_{23}$) are given by $\beta_2, \tau_{\gamma_3}^{-1} \allowbreak \tau_{\gamma_4}^{-1}(\beta_2), \tau_{\gamma_1}^{-1}\tau_{\gamma_2}^{-1}(\beta),$ and $\tau_{\gamma_1}^{-1}\tau_{\gamma_2}^{-1}\tau_{\gamma_3}^{-1}\tau_{\gamma_4}^{-1}(\beta_2)$, where the paths
$\beta_2, \allowbreak \gamma_1, \gamma_2,\gamma_3,$ and $\gamma_4$ are shown in Figure~\ref{braidasswith12,13,14,15} (or Figure~\ref{braidasswith20,21,22,23}, respectively).

\begin{center}
\begin{figure}[h]
\subfloat[]{\includegraphics[width=0.45\textwidth]{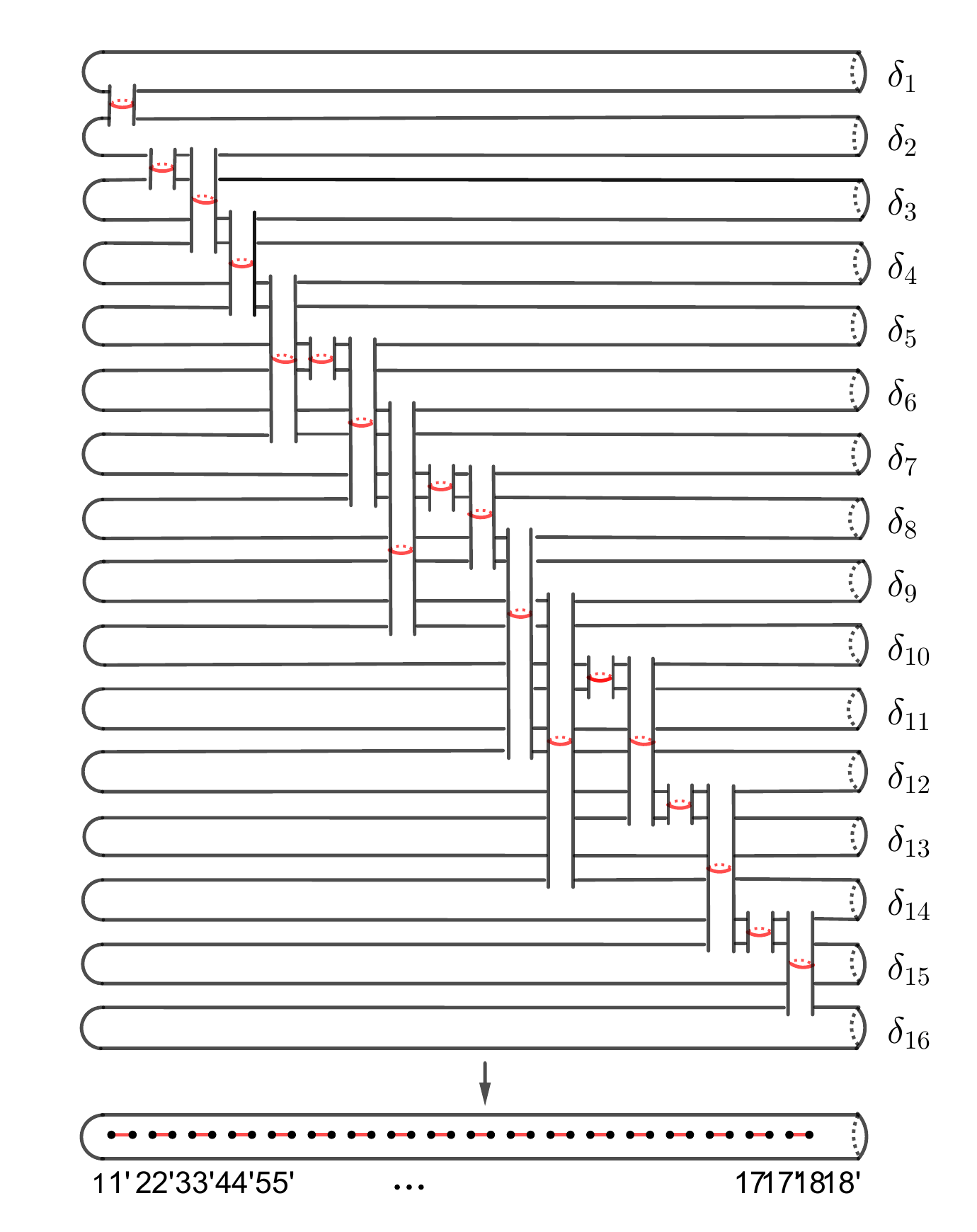}}
\hspace{3mm}
\subfloat[]{\includegraphics[width=0.5\textwidth]{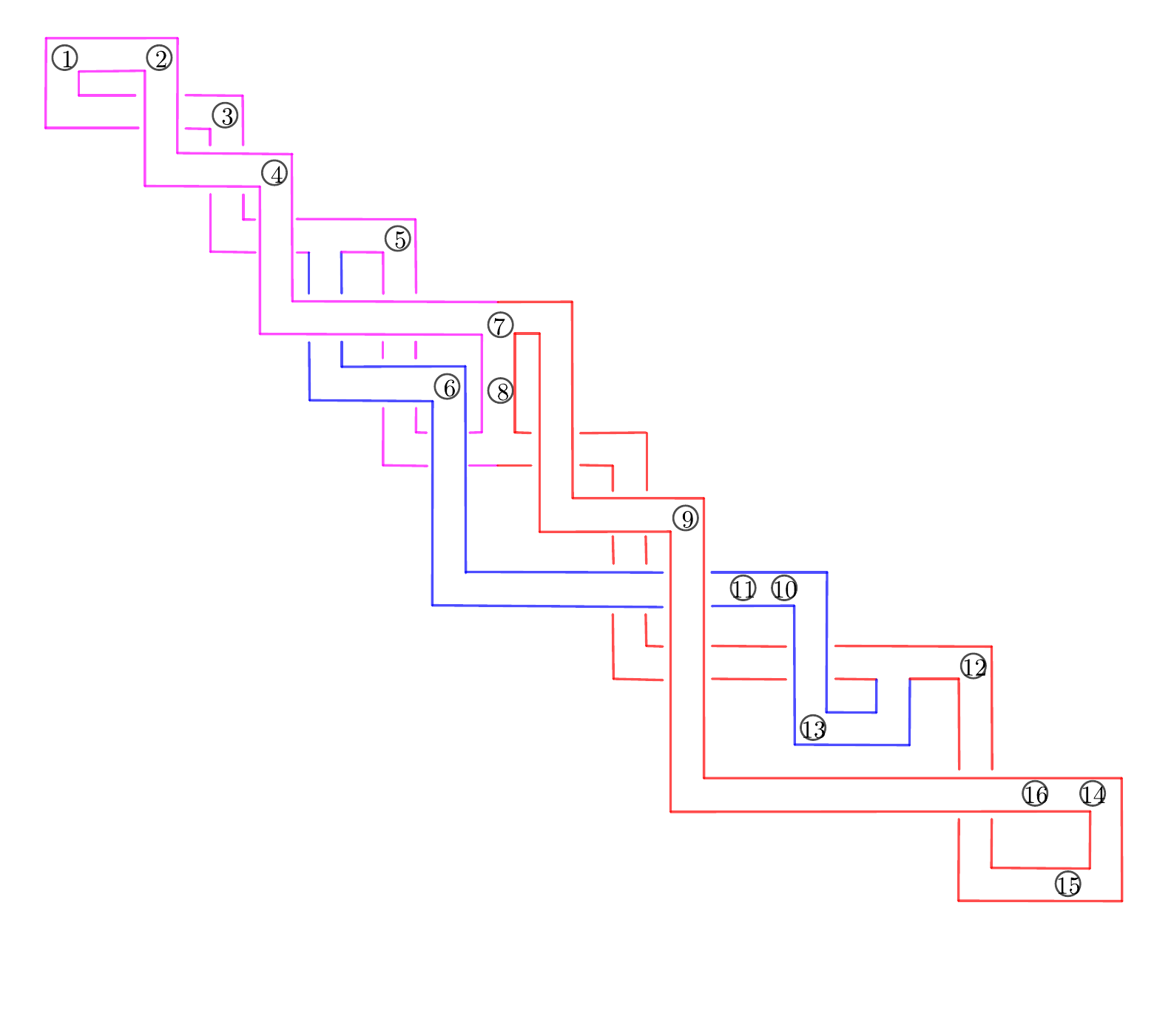}}
\caption{A regular fiber of $h_4:V_4\dashrightarrow\mathbb{P}^1$}
\label{fibercovering}
\end{figure}
\end{center}

\begin{figure}
\subfloat[$c_1, c_2, c_{17}, c_{18}, c_{25}, c_{27}$]{\includegraphics[width=0.5\textwidth]{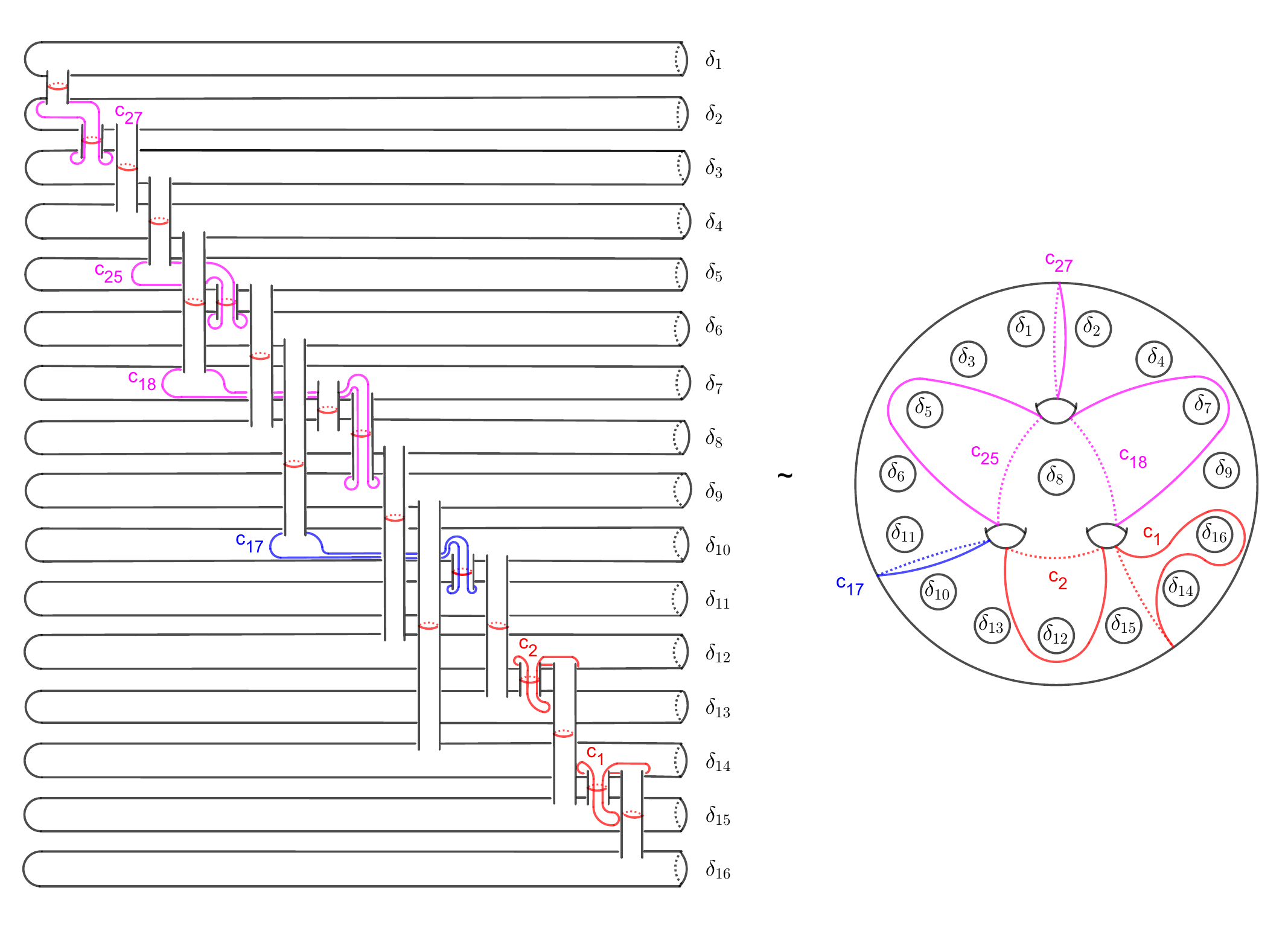}} 
\subfloat[$c_3, c_4, c_{26}$]{\includegraphics[width=0.5\textwidth]{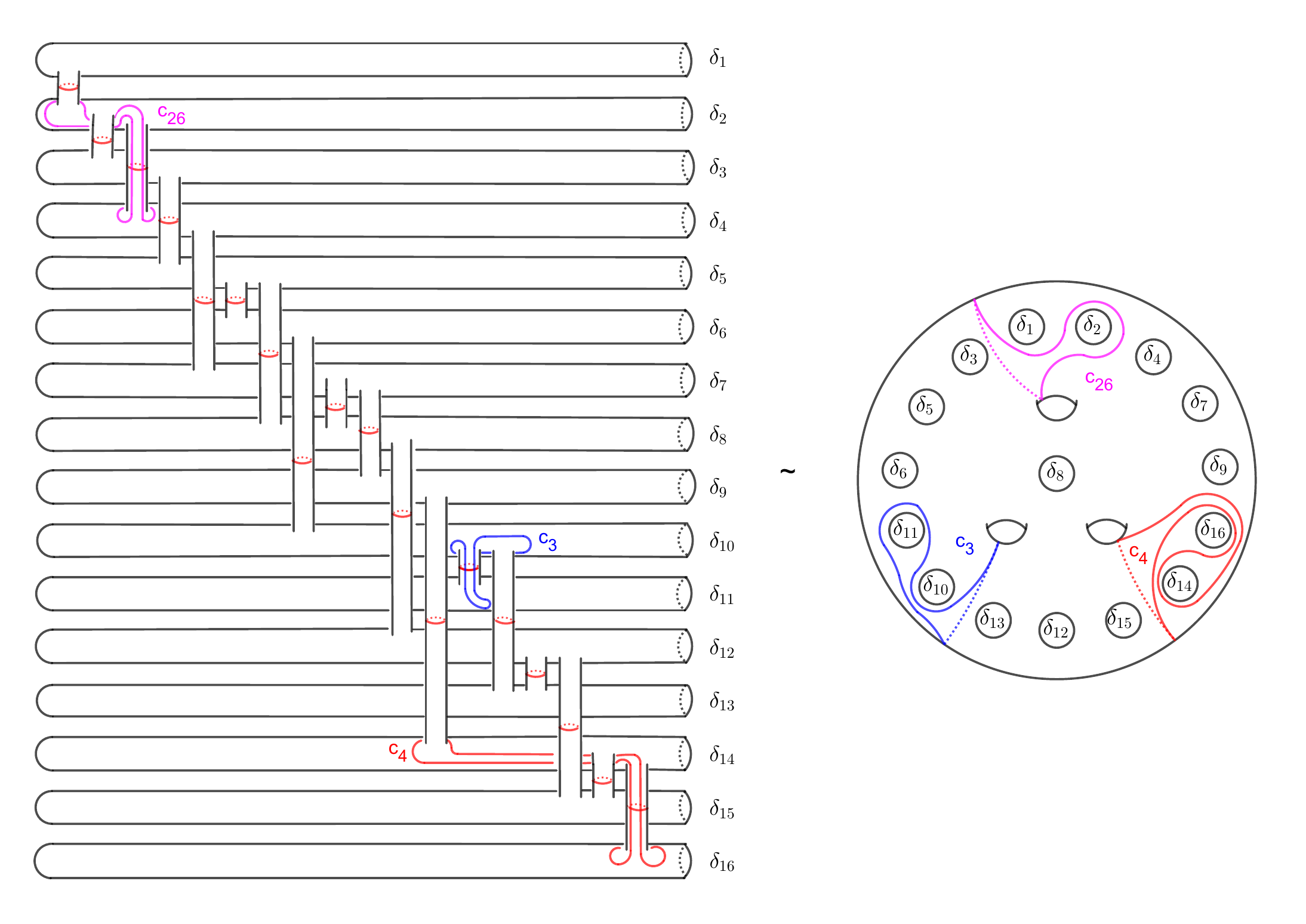}}\\
\subfloat[$c_6$]{\includegraphics[width=0.5\textwidth]{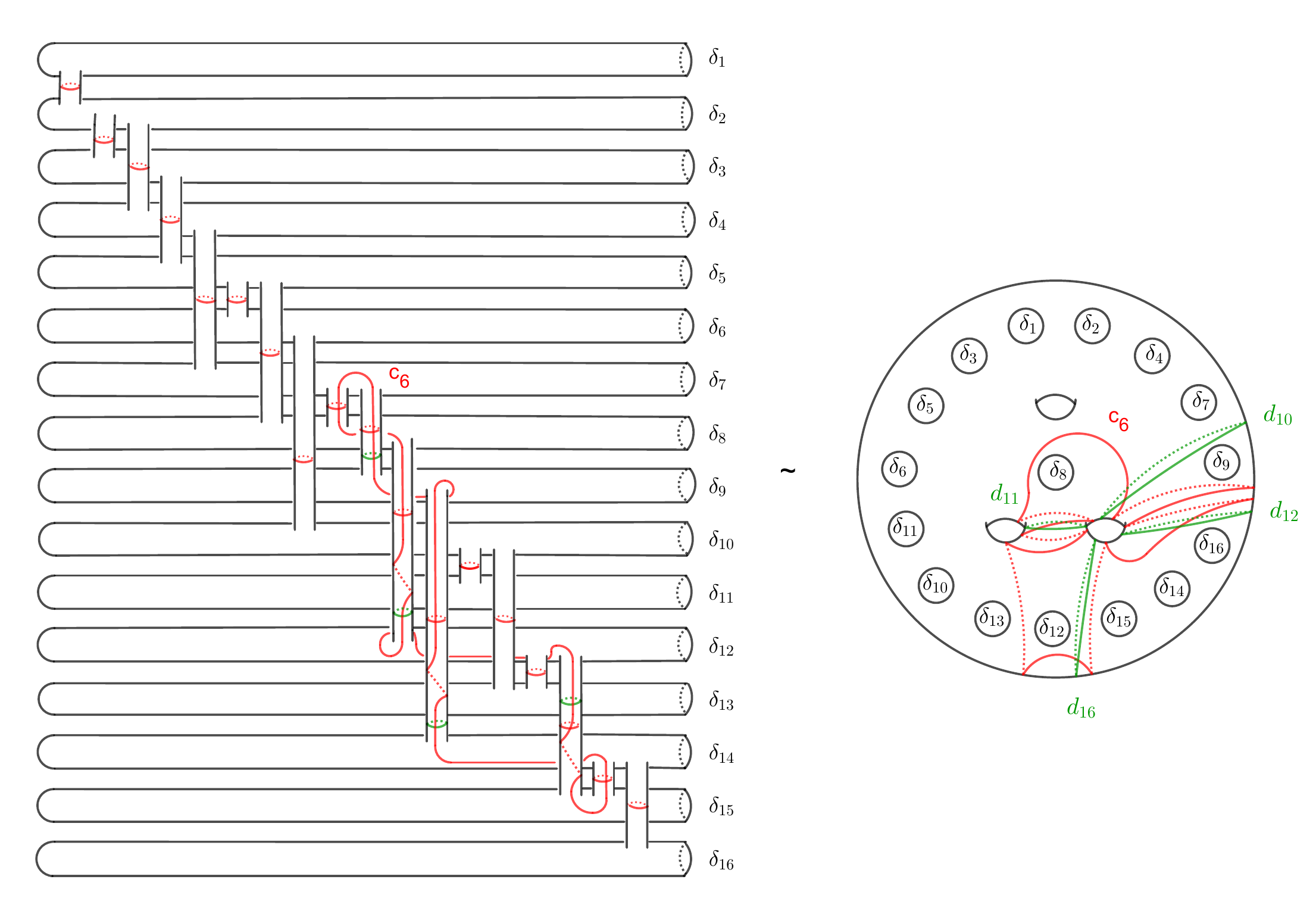}}
\subfloat[$c_5, c_{10}$]{\includegraphics[width=0.5\textwidth]
{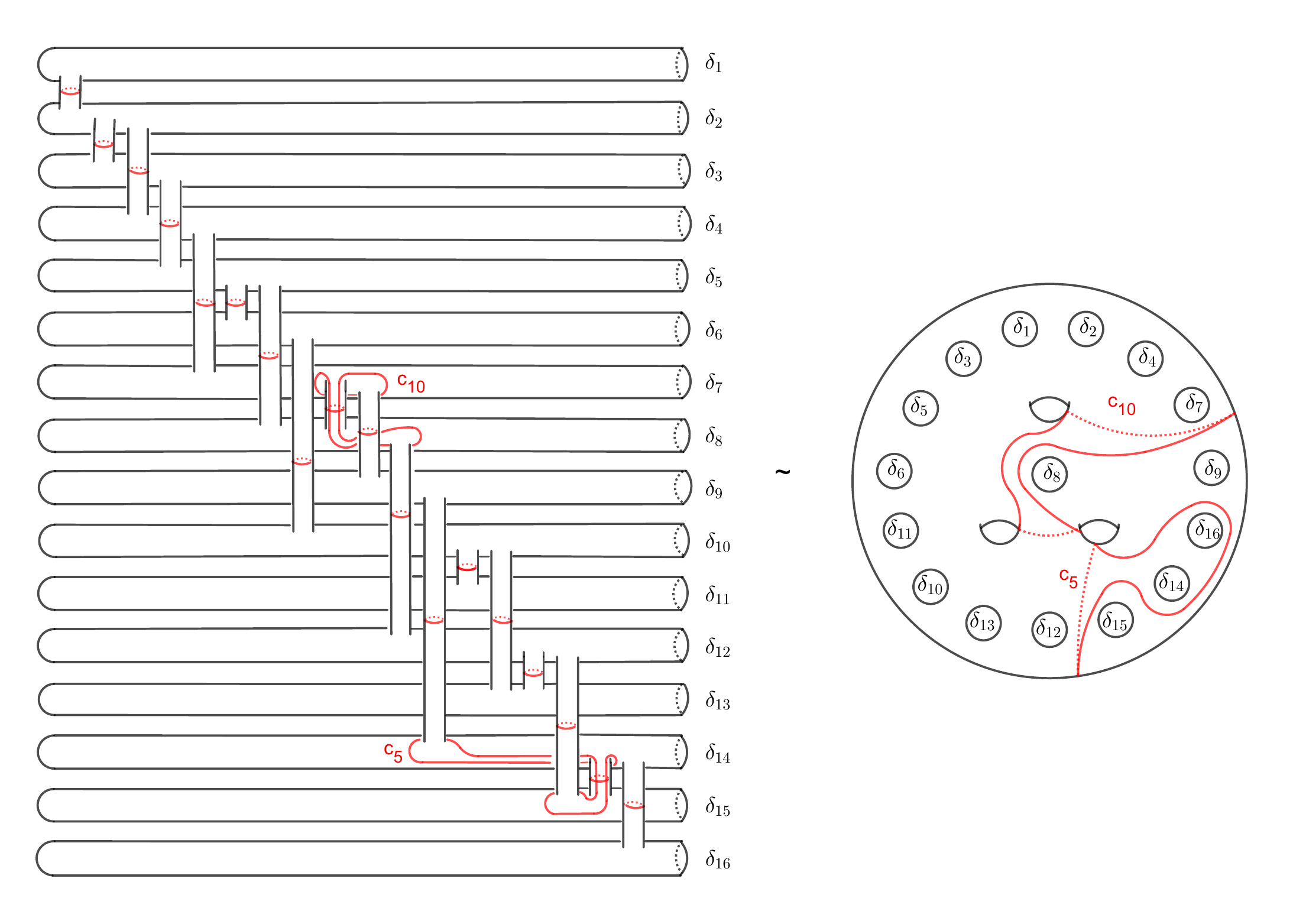}}\\
\subfloat[$c_{12}$]{\includegraphics[width=0.5\textwidth]
{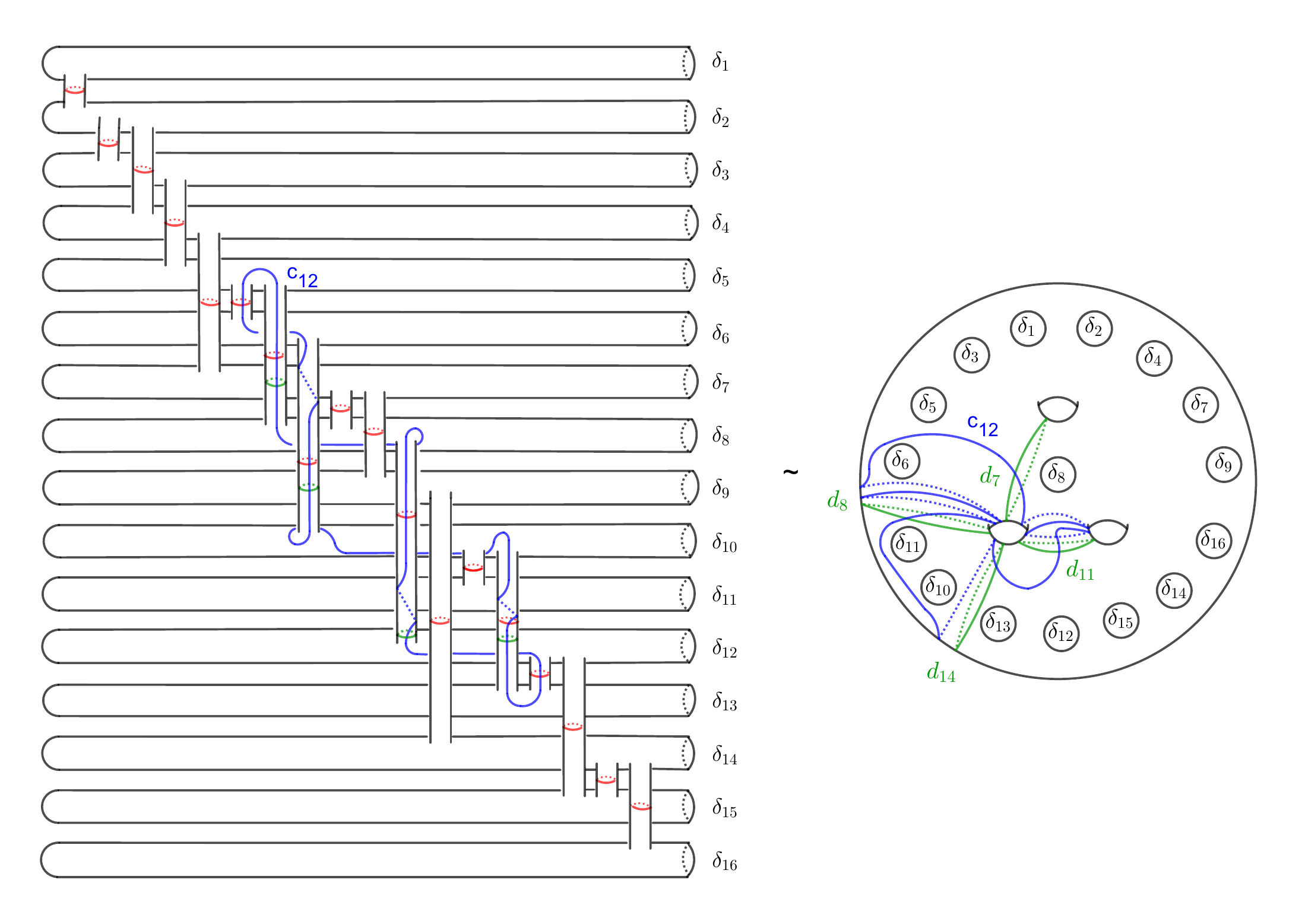}}
\subfloat[$c_{11}, c_{16}$]{\includegraphics[width=0.5\textwidth]
{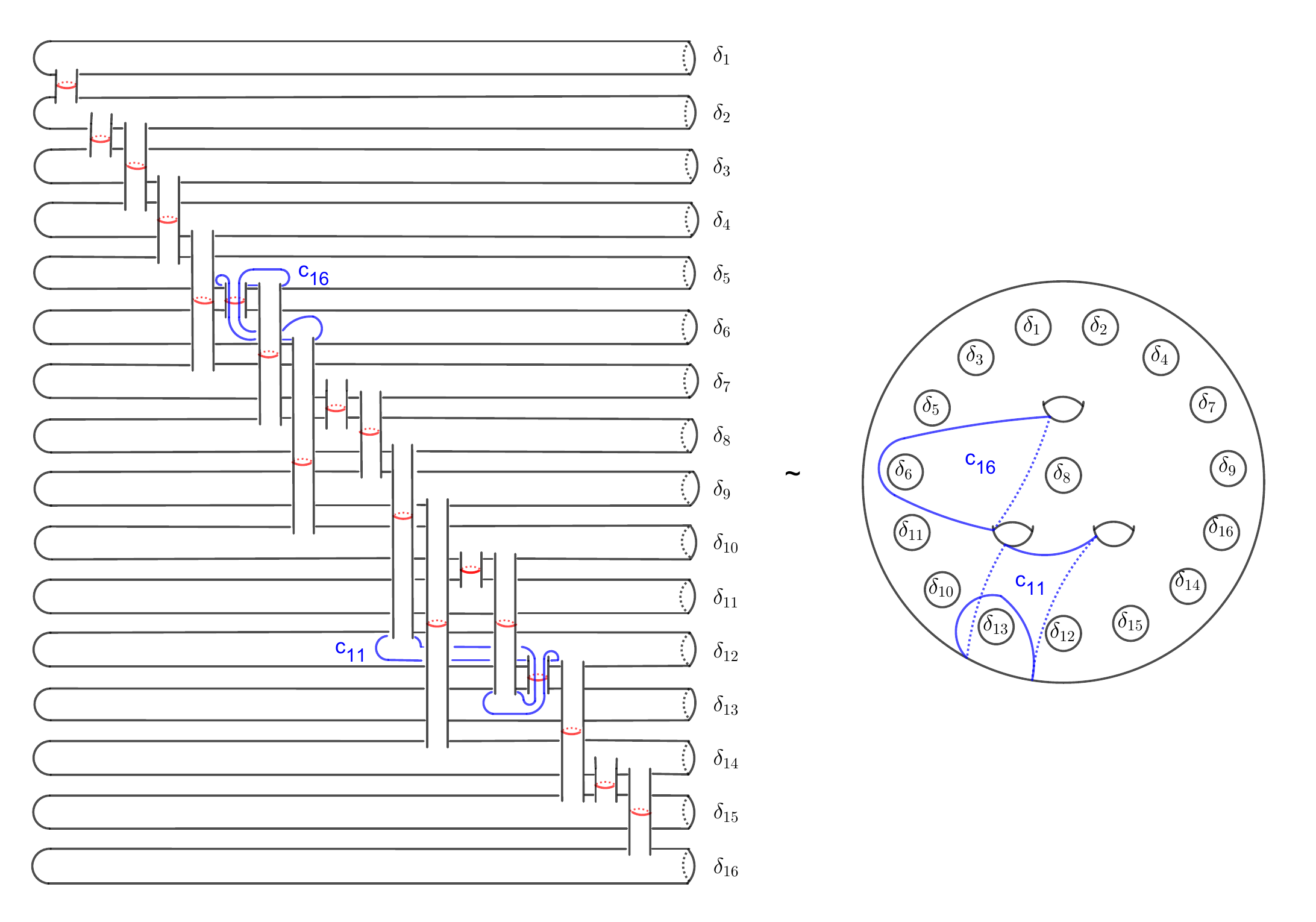}}\\
\subfloat[$c_{20}$]{\includegraphics[width=0.5\textwidth]
{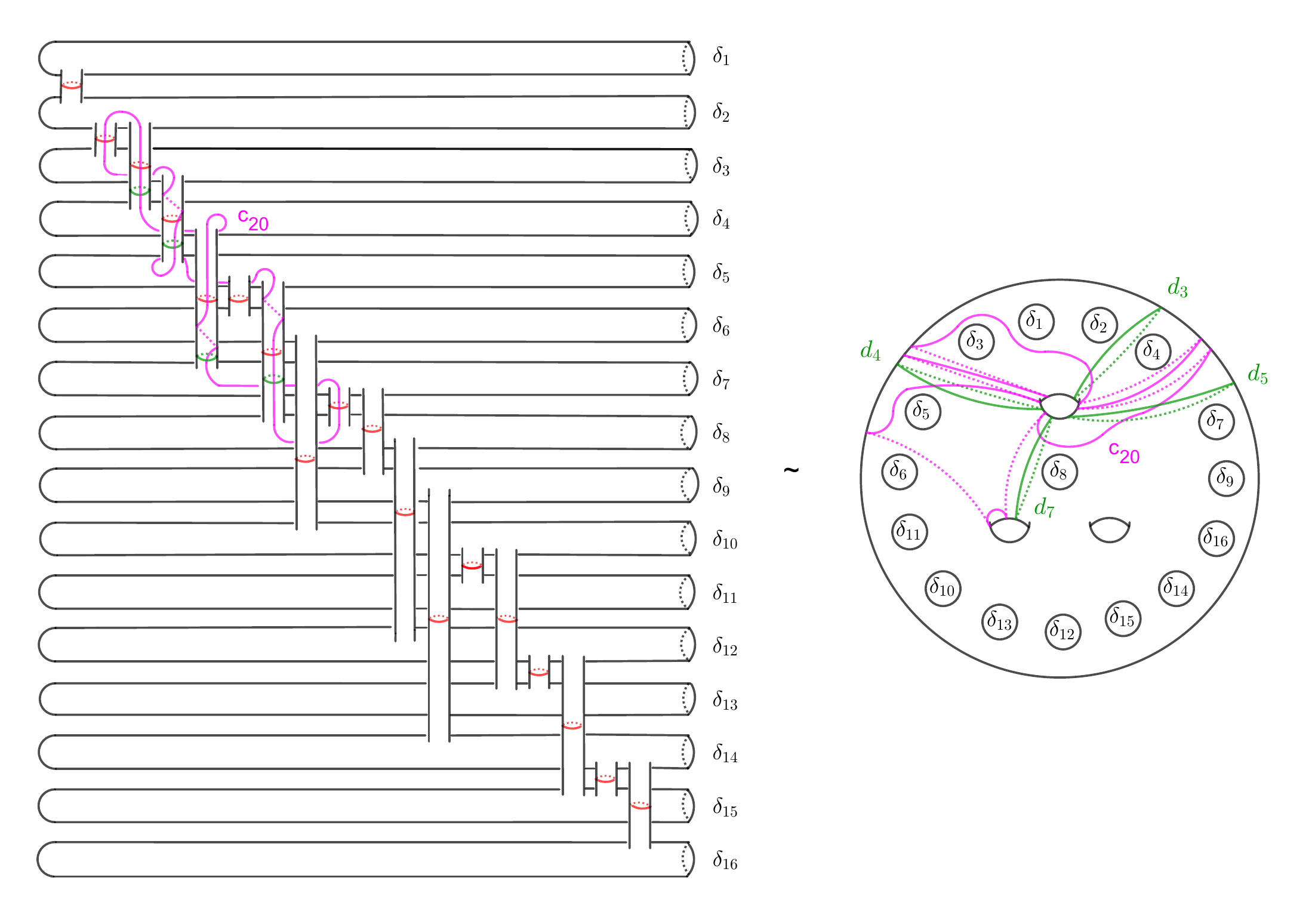}}
\subfloat[$c_{19}, c_{24}$]{\includegraphics[width=0.5\textwidth]
{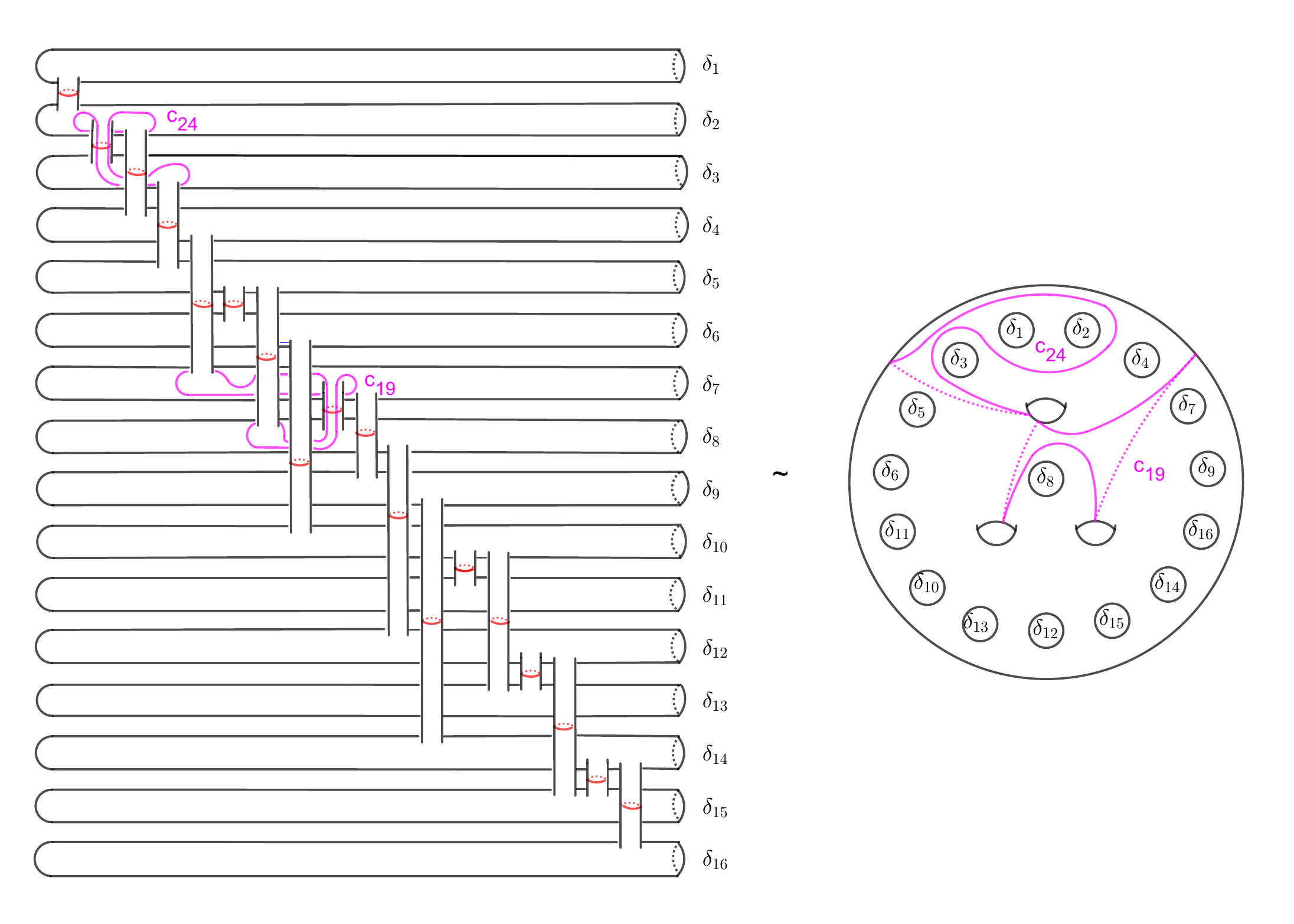}}
\caption{Vanishing cycles of the Lefschetz pencil $h_4:V_4\dashrightarrow\mathbb{P}^1$}
\label{fig:findvanishingcycles}
\end{figure}

In order to obtain the vanishing cycles of $h_4:V_4\dashrightarrow\mathbb{P}^1$,  
we have to lift the braid monodromies around the branch points to the mapping class group of the fiber surface under the branched covering 
$$\phi|_{h_4^{-1}(a_0')}:h_4^{-1}(a_0')\rightarrow\pi'^{-1}(a_0')$$ branched along $\pi'^{-1}(a_0')\cap S$. We denote $\pi'^{-1}(a_0')$ by $D$ and $\pi'^{-1}(a_0')\cap S$ by $Q=\{q_1,q_1',\cdots,q_{18},q_{18}'\}$. One can easily observe that this is a $16$-fold simple branched covering branched over $36$ points given by the following monodromy representation: 
$$\rho:\pi_1(D\setminus Q, q_0)\rightarrow S_{16}$$   
\[
\begin{array}{ccc|ccc|ccc}
\eta_1,\eta_1' & \mapsto & (1,2), & \eta_2,\eta_2' & \mapsto & (2,3), & \eta_3,\eta_3' & \mapsto & (2,4), \\
\eta_4,\eta_4' & \mapsto & (3,5), & \eta_5,\eta_5' & \mapsto & (4,7), & \eta_6,\eta_6' & \mapsto & (5,6), \\
\eta_7,\eta_7' & \mapsto & (5,8), & \eta_8,\eta_8' & \mapsto & (6,10),& \eta_9,\eta_9' & \mapsto & (7,8),\\
\eta_{10},\eta_{10}'& \mapsto & (7,9), & \eta_{11},\eta_{11}' & \mapsto & (8,12),& \eta_{12},\eta_{12}' & \mapsto &(9,14),\\
\eta_{13},\eta_{13}'& \mapsto & (10,11), & \eta_{14},\eta_{14}' & \mapsto & (10,13), & \eta_{15},\eta_{15}' & \mapsto & (12,13),\\
\eta_{16},\eta_{16}'& \mapsto & (12,15), & \eta_{17},\eta_{17}' & \mapsto & (14,15), & \eta_{18},\eta_{18}' & \mapsto & (14,16)
\end{array}
\]
where $q_0\in D\setminus Q$ is a fixed regular value, each $\eta_j$ (or $\eta_j'$) is the loop in $D\setminus Q$ based at $q_0$ and going around one branch point $q_j$ (or $q_j'$, respectively), and its image $(k_j,l_j)$ is the transposition in the symmetric group $S_{16}$ that can be considered as a bijection of $\phi^{-1}(q_0)$ permuting the point close to the plane $P_{k_j}$ and the point close to the plane $P_{l_j}$.

Hence, the branched covering can be described topologically as in Figure~\ref{fibercovering}(a) and we can find a homeomorphism from the covering space to the surface in Figure~\ref{fibercovering}(b) and further to the standard surface of genus $3$ with $16$ boundary components. By taking the circle component of the preimages of the paths in Figures~\ref{fig:braidmonodromyofregenerated}, \ref{braidasswith6,7,8,9}, \ref{braidasswith12,13,14,15}, and \ref{braidasswith20,21,22,23} under this branched covering, as shown in Figure~\ref{fig:findvanishingcycles}, we finally obtain the vanishing cycles $c_1,\cdots,c_{27}$ in the monodromy factorization $t_{c_{27}}\circ \cdots \circ t_{c_2}\circ t_{c_1}=\prod_{j=1}^{16}t_{\delta_j}$. Here, the simple closed curves $c_{1},\cdots,c_{6}, c_{10}, \cdots, c_{12}, c_{16}, \cdots, c_{20}, c_{24}, \allowbreak\cdots, c_{27}$ can be drawn on the standard surface $\Sigma_3^{16}$ after a fixed homeomorphism, as depicted in Figure~\ref{fig:findvanishingcycles}. Meanwile, the remaining simple closed curves $c_{7}, c_8, c_9, c_{13},\allowbreak c_{14}, c_{15},\allowbreak c_{21}, c_{22},$ and $c_{23}$ are given by $t_{d_{11}}^{-1}t_{d_{12}}^{-1}(c_{6}), t_{d_{10}}^{-1}t_{d_{16}}^{-1}(c_6), t_{d_{10}}^{-1}t_{d_{11}}^{-1}t_{d_{12}}^{-1}t_{d_{16}}^{-1}(c_6)$, $t_{d_8}^{-1}t_{d_{11}}^{-1}(c_{12}),\allowbreak t_{d_7}^{-1}t_{d_{14}}^{-1}(c_{12})$, $t_{d_7}^{-1}t_{d_8}^{-1}t_{d_{11}}^{-1}t_{d_{14}}^{-1}(c_{12})$, $t_{d_{4}}^{-1}t_{d_5}^{-1}(c_{20}), t_{d_3}^{-1}t_{d_{7}}^{-1}(c_{20}),$ and \linebreak
$t_{d_3}^{-1}t_{d_4}^{-1}t_{d_5}^{-1}t_{d_7}^{-1}(c_{20})$, respectively.

\subsection{Equivalence between a suitable topological construction and the holomorphic Lefschetz pencil for $d=4$}

In the previous section, we found the monodromy factorization of the holomorphic Lefschetz pencil of degree $4$ curves on $\mathbb{P}^2$. In this section, we will show that the genus $3$ holomorphic Lefschetz pencil can be constructed topologically by breeding the relations for the genus $1$ holomorphic Lefschetz pencil. 

\begin{figure}
\subfloat[]{\includegraphics[width=0.3\textwidth]{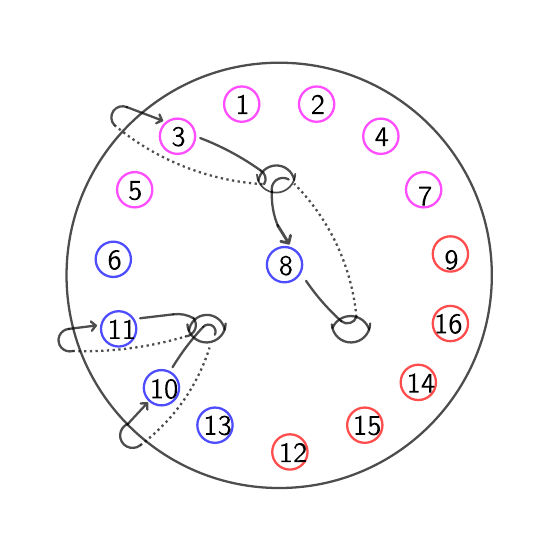}}
\subfloat[]{\includegraphics[width=0.3\textwidth]{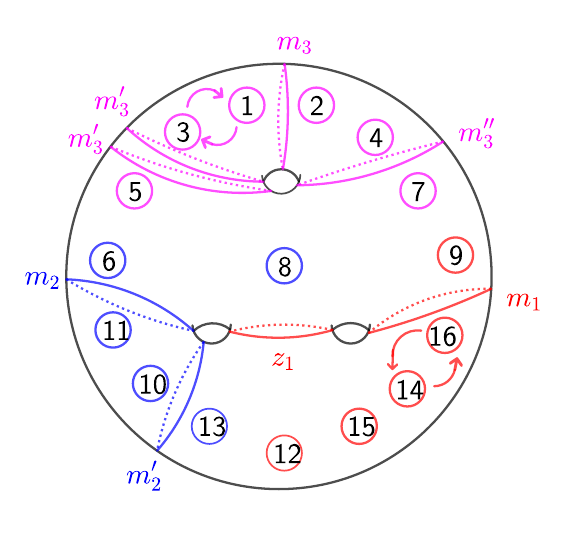}}
\caption{The global conjugations}
\label{globalconjugations}
\end{figure}

\begin{figure}
\subfloat{\includegraphics[width=0.33\textwidth]{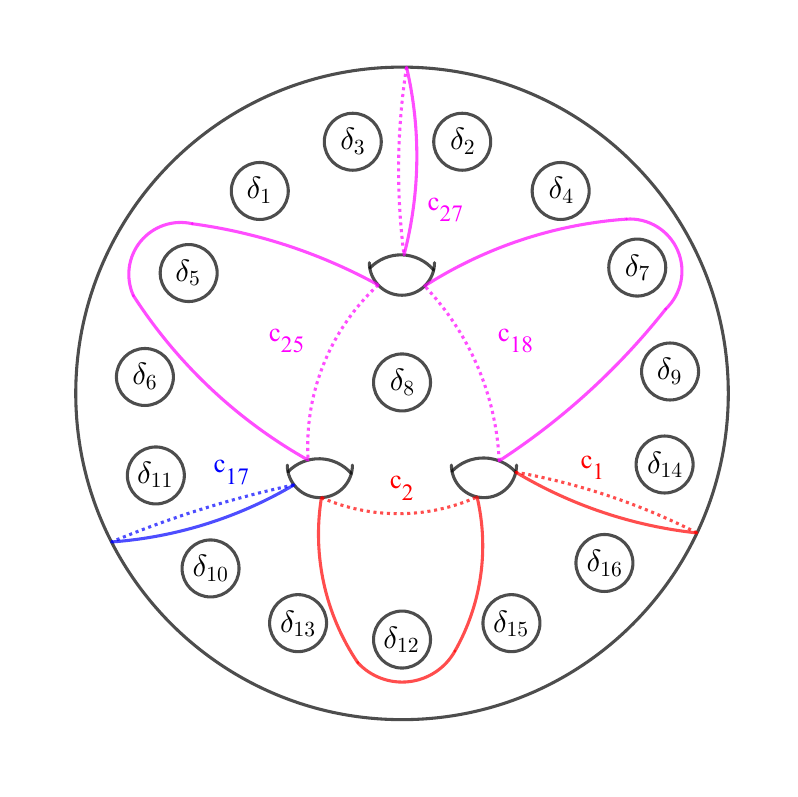}}
\subfloat{\includegraphics[width=0.33\textwidth]{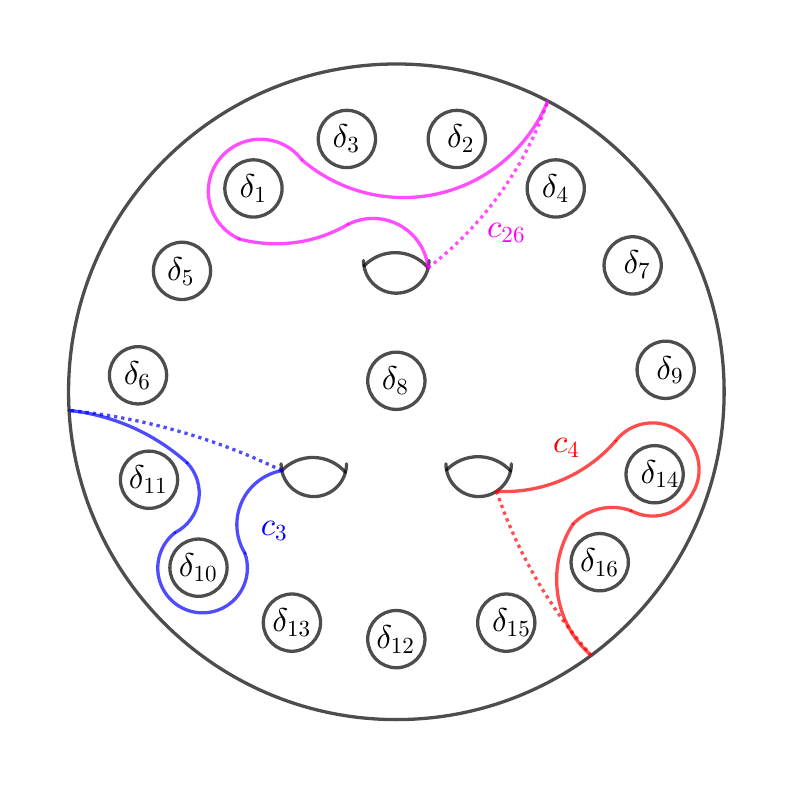}}
\subfloat{\includegraphics[width=0.33\textwidth]{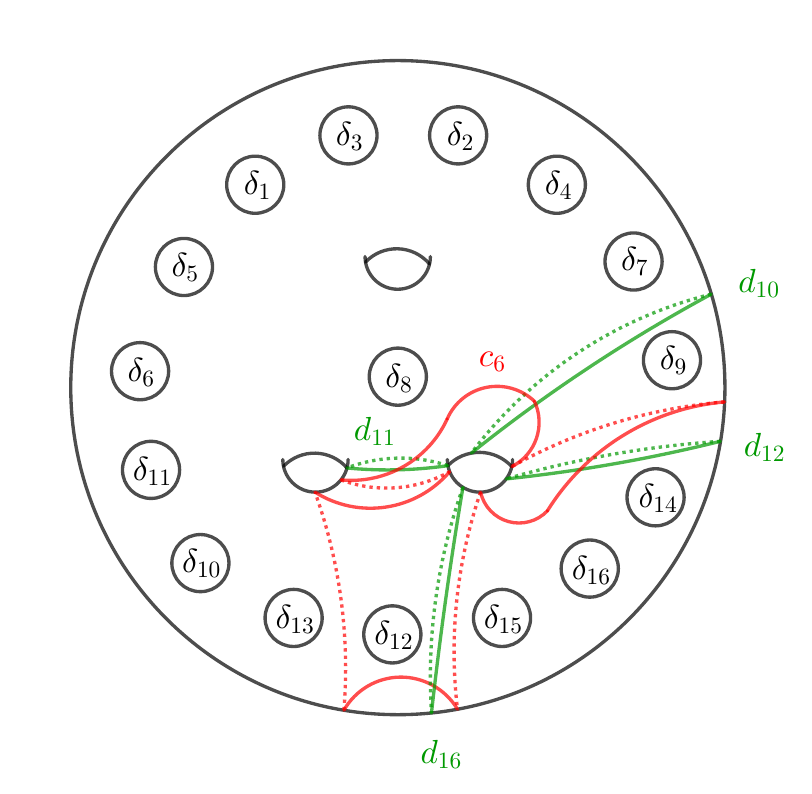}}\\
\subfloat{\includegraphics[width=0.33\textwidth]{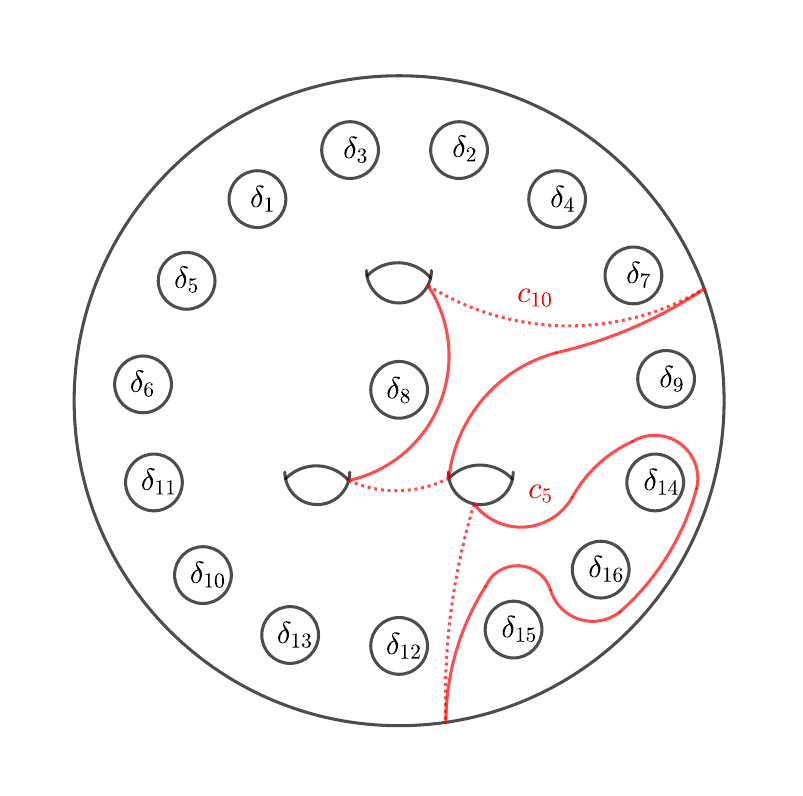}}
\subfloat{\includegraphics[width=0.33\textwidth]{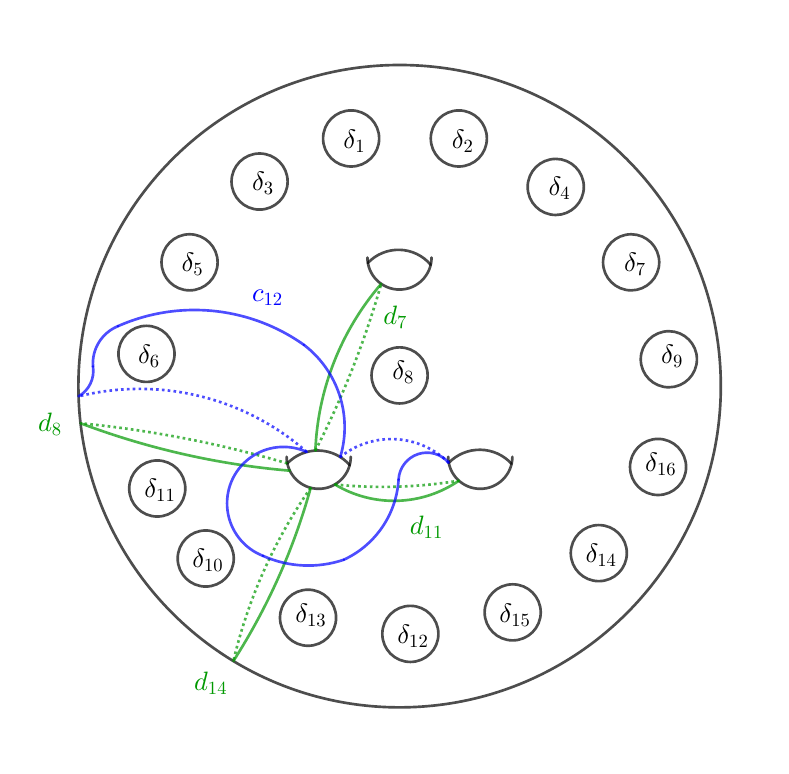}}
\subfloat{\includegraphics[width=0.33\textwidth]{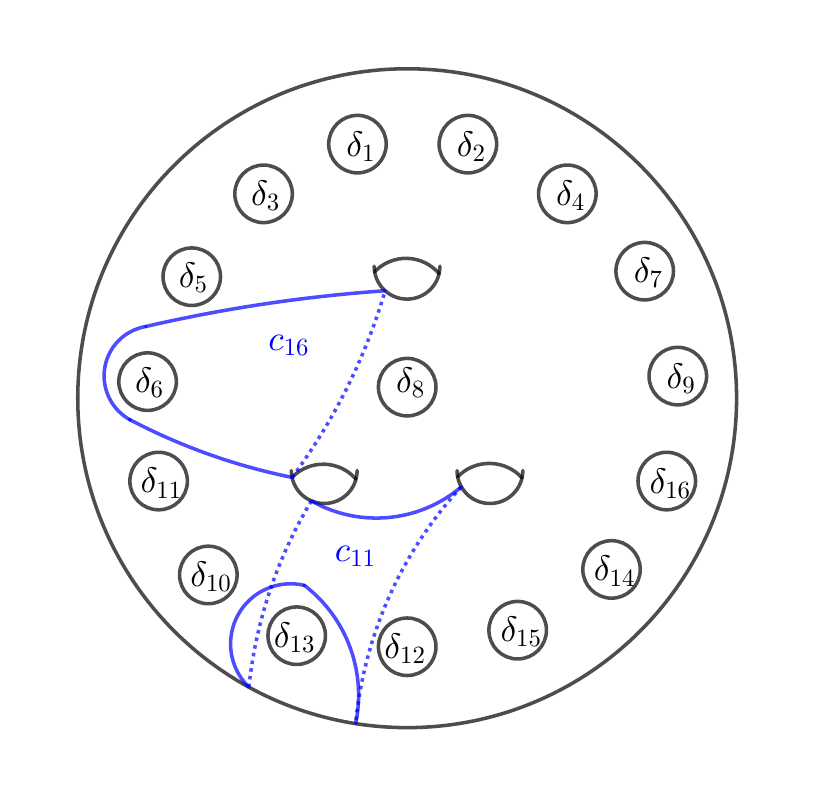}}\\
\subfloat{\includegraphics[width=0.33\textwidth]{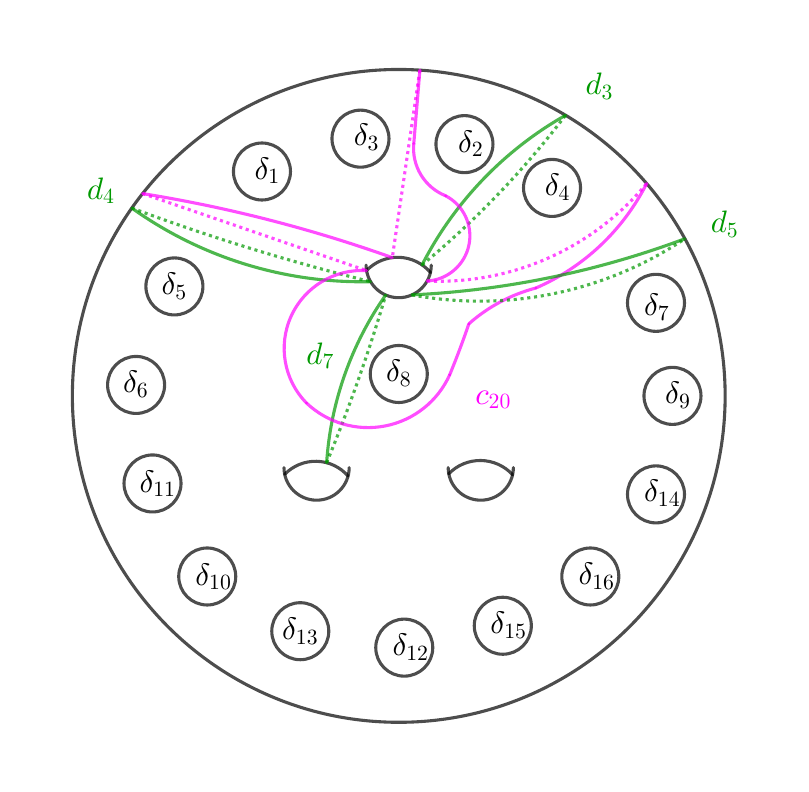}}
\subfloat{\includegraphics[width=0.33\textwidth]{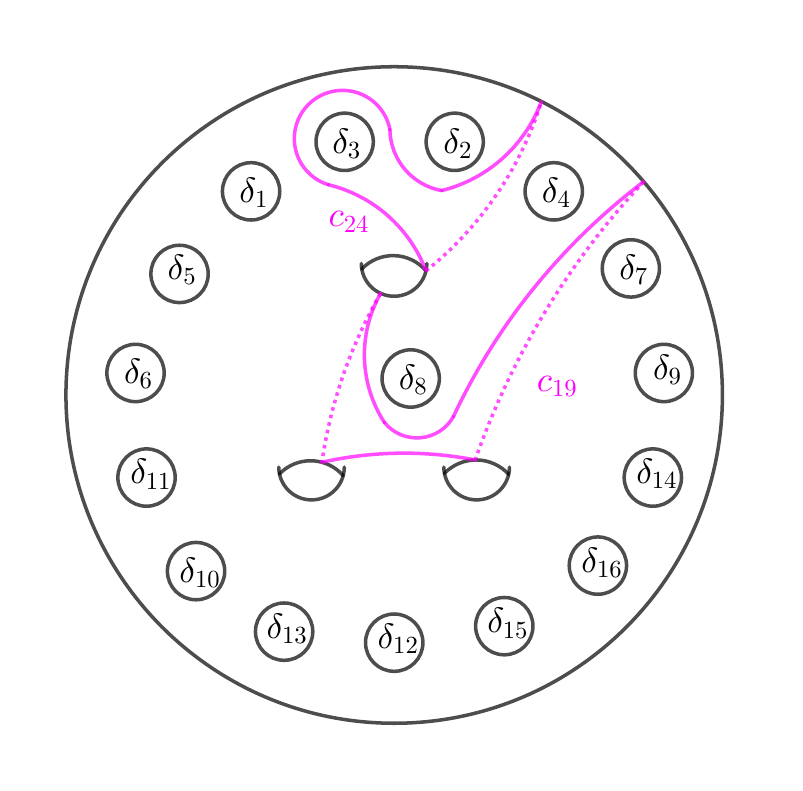}}
\caption{Vanishing cycles of the genus $3$ holomorphic Lefschetz pencil}
\label{fig:bmtprime}
\end{figure}

\begin{figure}[h]
\includegraphics[width=0.9\textwidth]{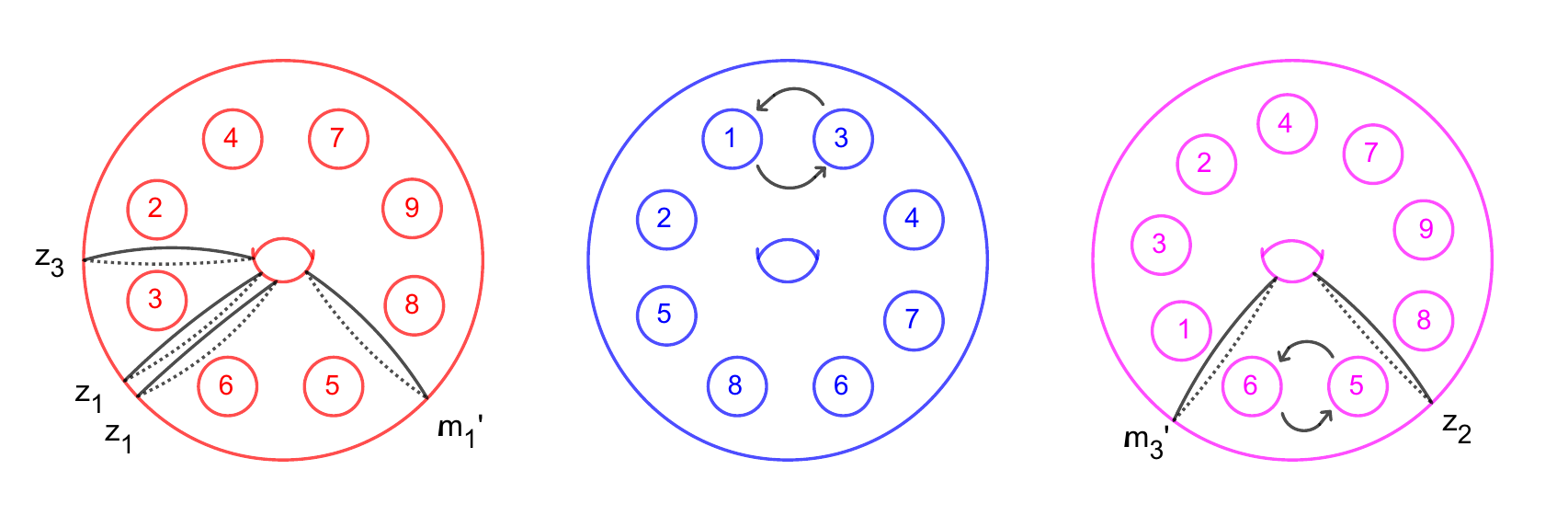}
\caption{Modifications for a suitable topological construction}
\label{fig:modifications}
\end{figure}

\begin{proof}[Proof of Theorem \ref{thm:second} (b)]
First, we perform the following sequence of global conjugations to the monodomry factorization obtained in the previous section based on the braid monodromy technique : first, we push the boundary components $\delta_3,\delta_8,\delta_{10},$ and $\delta_{11}$ along the curves in Figure~\ref{globalconjugations}(a), then take a global conjugation by $(t_{m_3}t_{m_3'}^{-2}t_{m_3''}^{-1})(t_{m_2}^{-1}t_{m_2'}^{-1})(t_{z_1}^{-1}t_{m_1}^{-1})$ where the Dehn twist curves are shown in Figure~\ref{globalconjugations}(b), and finally swap $\delta_1$ and $\delta_3$, also $\delta_{14}$ and $\delta_{16}$ as in Figure~\ref{globalconjugations}(b).   
Then we obtain another monodromy factorization of the holomorphic Lefschetz pencil $h_4:\mathbb{P}^2\dashrightarrow\mathbb{P}^1$ in Theorem \ref{thm:second} (a), where the resulting vanishing cycles are given in Figure~\ref{fig:bmtprime}. 

Now, we claim that one can obtain this monodromy factorization by breeding two copies of the $8$-holed torus relations and one copy of the $9$-holed torus relation as in the proof of Lemma \ref{lemma for $d=4$}. Here, we use the same building blocks up to isomorphism as those in the Lemma \ref{lemma for $d=4$}, but with some modifications using certain global conjugations as in Figure~\ref{fig:modifications} before the breeding operation. Precisely, for the first building block, we begin with the $8$-holed torus relation obtained by capping the boundary component $\delta_1$ from Figure $23$ in \cite{HH:21}, and then take a global conjugation by $(t_{m_1'}^{-1}t_{z_1})(t_{z_1}t_{z_3}^{-1})$. We denote the resulting monodromy factorization by $c_{12}^{(1)}\cdots c_2^{(1)}c_1^{(1)}=\delta_9\cdots\delta_3\delta_2$. Next, we swap $\delta_1$ and $\delta_3$ of the $8$-holed torus rleation in Figure $26$ in \cite{HH:21} to get $c_{12}^{(2)}\cdots c_2^{(2)}c_1^{(2)}=d_{8}\cdots d_2d_1$.
Finally, we begin with the $9$-holed torus relation in Figure $23$ in \cite{HH:21}, then swap $\delta_5$ and $\delta_6$, and then take a global conjugation by $t_{z_2}^{-1}t_{m_3'}$. We denote the resulting monodromy factorization by $B_{12}\cdots B_2B_1=\delta_{9}'\cdots \delta_2'\delta_1'$.\\ 

By breeding the above two copies of the $8$-holed torus relations as in Figure~\ref{twoembeddingsof8holedtorus}, we get the following relation in $\Gamma_2^{12}$:
\begin{multline}
(d_3d_5d_6d_7d_8)(\delta_4\delta_5\delta_7\delta_8\delta_9)\\
=(\overline{d}_2\overline{\delta}_2c_{12}^{(2)}c_{11}^{(2)}c_{11}^{(1)})(c_{10}^{(2)}\cdots c_5^{(2)}c_4^{(2)}c_2^{(2)})(c_9^{(1)}\cdots c_2^{(1)}c_1^{(1)}).
\end{multline}

Next, by applying a $9$-holed torus breeding as in Figure~\ref{fig:9holedtorusbreeding}, we obtain the following relation in in $\Gamma_3^{16}$:
\begin{multline}
(\delta_1'\delta_2'\delta_3'\delta_4'\delta_5'\delta_7')(d_3d_5d_6d_7d_8)(\delta_4\delta_5\delta_7\delta_8\delta_9)\\
=(B_{12}\cdots B_{4}B_{3})(c_{10}^{(2)}\cdots c_{5}^{(2)}c_{4}^{(2)}c_2^{(2)})(c_9^{(1)}\cdots c_2^{(1)}c_1^{(1)}).
\end{multline}

It is easy to check that the resulting vanishing cycles coincide with those in Figure \ref{fig:bmtprime}, and this completes the proof. 

\end{proof}

\bibliographystyle{plain}
\bibliography{LPonCP2}

\end{document}